\documentclass[11pt]{amsart}
\usepackage{verbatim}
\usepackage{amsmath,amssymb,amsthm}             
\usepackage{amssymb}             
\usepackage{amsfonts}            
\usepackage{mathrsfs}            
\usepackage{amsthm}
\usepackage{algorithm}  
\usepackage{algorithmicx}  
\usepackage{algpseudocode}  
\usepackage{mathtools}
\usepackage{commath}
\usepackage{physics}
\usepackage{bm}
\usepackage{graphicx}
\usepackage{geometry}
\usepackage{float}
\usepackage{listings}
\usepackage{subfigure}
\usepackage{multirow}
\usepackage{diagbox}
\usepackage{color}
\usepackage[export]{adjustbox}
\usepackage{geometry}                
\usepackage{graphicx}
\usepackage{amssymb}
\usepackage{epstopdf}
\usepackage{color}
\usepackage{bbm, dsfont}
\newtheorem{theorem}{Theorem}
\newtheorem{lemma}{Lemma}[section]
\newtheorem{corollary}{Corollary}[section]
\newtheorem{proposition}{Proposition}[section]

\theoremstyle{definition}
\newtheorem{definition}{Definition}
\theoremstyle{remark}
\newtheorem{remark}{Remark}

\theoremstyle{example}

\usepackage{latexsym,amssymb,amsmath,amsfonts,amsthm}
\usepackage{amsthm,amsmath}
\usepackage{float}
\usepackage{enumerate}
\usepackage{epic}
\usepackage{fullpage}
\usepackage{bbm}
\usepackage{subfigure}
\usepackage{graphicx}
\usepackage[toc,page]{appendix}
\usepackage{amsfonts}
\usepackage[all]{xy}
\usepackage{caption}
\usepackage{geometry}              
\usepackage{graphicx}
\usepackage{enumerate}
\usepackage{amssymb}
\usepackage{epstopdf}
\usepackage{color}
\usepackage{bbm, dsfont}
\usepackage[utf8]{inputenc}
\usepackage{amssymb,amsthm,amsmath,amssymb,wrapfig,dsfont}
\usepackage[dvipsnames]{xcolor}
\definecolor{myred}{RGB}{251,154,133}
\definecolor{myblue}{RGB}{153,206,227}
\definecolor{mylightblue}{RGB}{0, 150, 255}
\definecolor{mygreen}{RGB}{32, 210, 64}
\definecolor{mygray}{RGB}{220, 220, 220}

\usepackage{tikz}
\usetikzlibrary{decorations.pathmorphing}
\tikzset{snake it/.style={decorate, decoration=snake}}
\usetikzlibrary{shapes.geometric,positioning,decorations.pathreplacing}

\usepackage{amsmath}
\usepackage{mathrsfs}

\newcommand{\BZ}{{\mathbb{Z}}}

\numberwithin{equation}{section}

\newtheorem{claim}{Claim}

\newtheorem*{notation}{Notation}
\newcommand{\RNum}[1]{\uppercase\expandafter{\romannumeral #1\relax}}

\begin{document}

	\title{On Chemical Distance and Local Uniqueness of a Sufficiently Supercritical Finitary Random Interlacement}
	
	\author{Zhenhao Cai}
	\address[Zhenhao Cai]{Peking University}
	\email{caizhenhao@pku.edu.cn}

	\author{Xiao Han}
	\address[Xiao Han]{Peking University}
	\email{jsnjkds@pku.edu.cn}
	
	\author{Jiayan Ye}
	\address[Jiayan Ye]{Texas A\&M University}
	\urladdr{www.math.tamu.edu/~tomye}
	\email{tomye1992@gmail.com }
	
	\author{Yuan Zhang}
	\address[Yuan Zhang]{Peking University}
	\email{zhangyuan@math.pku.edu.cn}

	\maketitle
	\tableofcontents
	
\begin{abstract}
	In this paper, we study geometric properties of the unique infinite cluster $\Gamma$ in a sufficiently supercritical Finitary Random Interlacements $\mathcal{FI}^{u,T}$ in $\BZ^d, \ d\ge 3$. We prove that the chemical distance in $\Gamma$ is, with stretched exponentially high probability, of the same order as the Euclidean distance in $\BZ^d$. This also implies a shape theorem parallel to those for Bernoulli percolation and random interlacements. We also prove local uniqueness of $\mathcal{FI}^{u,T}$, which says any two large clusters in $\mathcal{FI}^{u,T}$ ``close to each other" will with stretched exponentially high probability be connected to each other within the same order of the distance between them. 
\end{abstract}

\section{Introduction}
The model of random interlacements was introduced by Sznitman \cite{Sznitman2009Vacant} in $2007$. It can be thought as a Poisson cloud of bi-infinite simple random walk trajectories on the lattice $\mathbb{Z}^d$, for $d \geq 3$. For a more thorough description of random interlacements, readers are referred to \cite{drewitz2014introduction} and the references therein.

Finitary random interlacements (FRI) was introduced by Bowen \cite{bowen2019finitary} to solve a special case of the Gaboriau-Lyons problem. It can be seen as a variant of random interlacements, where the random walk trajectories are geometrically truncated. See Section 2 for the precise definition of FRI. Denote the random graph generated by FRI process on  the lattice $\mathbb{Z}^d$, $d \geq 3$, by $\mathcal{FI}^{u,T}$, where $u>0$ is similar to the intensity parameter in random interlacements $\mathcal{I}^u$ that controls the average number of trajectories traversing a given finite subset, and $T>0$ determines the expectations of geometric ``cut points" at which random walk trajectories are truncated. Bowen \cite{bowen2019finitary} showed that the measure of FRI $\mathcal{FI}^{u,T}$ converges to the one of random interlacements $\mathcal{I}^{2du}$ weakly when $T$ goes to infinity. Procaccia, Ye, and Zhang \cite{procaccia2019percolation} proved that 
$\mathcal{FI}^{u,T}$ has a non-trivial phase transition for all fixed $u>0$. In particular, there is $T_1 (u,d) >0$ such that for all $T >T_1$, $\mathcal{FI}^{u,T}$ has a unique infinite connected component $\Gamma$.

We are interested in the chemical distance in $\Gamma$ of $\mathcal{FI}^{u,T}$. Antal and Pisztora \cite{antal1996chemical} proved that the chemical distance in the unique infinite connected component of supercritical Bernoulli percolation is of the same order as $\ell_{\infty}$ distance with high probability. Procaccia and Shellef \cite{procaccia2014range} proved a similar chemical distance result for trace of random walk on discrete torus and random interlacements, with an iterated log correction. {\v C}ern{\'y} and Popov \cite{vcerny2012internal} improved the chemical distacne result for random interlacements in \cite{procaccia2014range} by removing the log correction. Drewitz, R\'{a}th, and Sapozhnikov \cite{drewitz2014chemical} extended the results in \cite{antal1996chemical, vcerny2012internal, procaccia2014range} and showed that all correlated percolation models that satisfy a list of conditions have similar chemical distance in their infinite connected components. Many models satisfy the list of conditions in \cite{drewitz2014chemical}, including random interlacements, vacant set of random interlacements, and level sets of Gaussian free field. In \cite{procaccia2019percolation}, Procaccia, Ye, and Zhang conjectured that the unique infinite connected component $\Gamma$ of $\mathcal{FI}^{u,T}$ has similar chemical distance. 

In this paper, we prove that for all fixed $u >0$, the chemical distance in $\Gamma$ of $\mathcal{FI}^{u,T}$ is also of the same order as $\ell_{\infty}$ distance when $T$ is large enough. This gives a positive answer to the conjecture in \cite{procaccia2019percolation}. 

A key difference between FRI and the regular random interlacements is that the former is composed of simple random walk trajectories truncated at a finite length. As a result, in order to connect a point to infinity, one has to switch infinitely many different trajectories. Thus one may not always be able to obtain a sufficiently long connected path without using up the Poisson intensity of the random measure. This brings significant challenge in employing the independence of Poisson Point Processes to ``entangle” independent trajectories and create good connectivity/distance event. In order to overcome such obstacle, the following scheme is carefully implemented in this paper:

Inspired by \cite{vcerny2012internal}, an infinite connected subset $\bar \Gamma$ with desired chemical distance can be constructed using some of the trajectories in $\mathcal{FI}^{u,T}$. Note that $\Gamma$ is the unique infinite connected component of $\mathcal{FI}^{u,T}$. The “good” cluster $\bar \Gamma$ can be seen as a “highway system” embedded in $\Gamma$, and the chemical distance between any two points in $\Gamma$ can be bounded from above by using this “highway system”. Moreover, suppose we have a point in $\Gamma$, and a connected path from this point to infinity. Among all truncated simple random walk trajectories composing this path, with high probability there have to be a ``good fraction” of them with ``decent lengths” so that they can almost independently find ways to our ``highway system” $\bar \Gamma$. A decomposition-based argument invented in \cite{rodriguez2013phase} plays a key role in this step. 

With chemical distance proved, we show that FRI has local uniqueness property (see Theorem $2$ for precise definition), and obtain a shape theorem for FRI as a corollary. By the result of Procaccia, Rosenthal, and Sapozhnikov \cite{procaccia2016quenched}, random walks on the unique infinite connected component $\Gamma$ of FRI satisfy a quenched version of invariance principle. 

\begin{remark}
	In \cite{drewitz2014chemical}, local uniqueness is one of the conditions in order to prove that chemical distance in the underlying random set is of the same order as $\ell_{\infty}$ distance. Once we have proved local uniqueness for FRI, one may verify that all conditions in \cite{drewitz2014chemical} are satisfied by FRI and thus give an alternative proof for chemical distance. However, to our knowledge it is difficult to prove local uniqueness directly in our case because of the aforementioned difference between FRI and random interlacements. As can be seen in the following sections, the proof of local uniqueness for FRI is essentially equivalent to that of chemical distance.
\end{remark}

This paper is organized as follows. We introduce definitions of FRI and our main results in Section $2$. In Section $3$, there are some notations and useful results. The construction of an infinite sub-cluster with good chemical distance is in Section $4$. In Sections $5$ and $6$, we prove that $\Gamma$ has the desired chemical distance. We give a proof that FRI has local uniqueness in Section $7$. 

\section{Main results}
According to \cite{procaccia2019percolation}, there are two equivalent definitions of FRI. For $x\in \mathbb{Z}^d$ and $T>0$, let $P_x^{(T)}$ be the law of a geometrically killed simple random walk starting at $x$ with killing rate $\frac{1}{T+1}$ (one can see Section 4.2 of \cite{lawler2010random} for precise definition of geometrically killed simple random walks). Denote the set of all finite paths on $\mathbb{Z}^d$ by $W^{\left[ 0,\infty\right) }$. Since $W^{\left[ 0,\infty\right) }$ is countable, the measure $v^{(T)}=\sum_{x\in \mathbb{Z}^d}\frac{2d}{T+1}P_x^{(T)}$ is a $\sigma-$ finite measure on $W^{\left[ 0,\infty\right) }$.
\begin{definition}\label{definition1}
	For $0<u,T<\infty$, the finitary random interlacements $\mathcal{FI}^{u,T}$ is a Poisson point process with intensity measure $uv^{(T)}$. The law of $\mathcal{FI}^{u,T}$ is denoted by $P^{u,T}$. 
\end{definition}
\begin{definition}\label{definition2}
	For each site $x\in \mathbb{Z}^d$, $N_x$ is a Poisson random variable with parameter $\frac{2du}{T+1}$. Start $N_x$ independent geometrically killed simple random walks starting at $x$ with killing rate $\frac{1}{T+1}$. Let $\mathcal{FI}^{u,T}$ be the point measure on $W^{\left[ 0,\infty\right) }$ composed of all the trajectories above from all sites in $\mathbb{Z}^d$. 
\end{definition}
In \cite{procaccia2019percolation}, it has been proved that for any $u>0$ and $d\ge 3$, there is $0<T_1(u,d)<\infty$ such that for any $T>T_1$, $\mathcal{FI}^{u,T}$ has a unique infinite cluster almost surely (see Theorem 1 of \cite{procaccia2019percolation}). We denote the unique infinite cluster of $\mathcal{FI}^{u,T}$ by $\Gamma$. In this paper, we prove that the chemical distance on $\Gamma$ has the same order as the $l^\infty$ distance, which means 
\begin{theorem}\label{theorem1}
	For any $u>0$ and $d\ge 3$, there exist $0<T_2(u,d)<\infty$ and $C_1(u,d)>0$ such that $\forall T>T_2$, $\exists \delta(u,d,T)>0$, $c(u,d,T)>0$, $c'(u,d,T)>0$ satisfying: for any $ y\in \mathbb{Z}^d$,
	\begin{equation}
	P^{u,T}\left[0\in \Gamma, y\in \Gamma, \rho(0,y)>C_1 |y| \right] \le ce^{-c'|y|^\delta}, 
	\end{equation}
	where $\rho(\cdot,\cdot)$ is the chemical distance on $\mathcal{FI}^{u,T}$ and its definition can be found in Section \ref{notations}.
\end{theorem}
Based on Theorem \ref{theorem1} and Subadditive Ergodic Theorem (see Section 7 of \cite{vcerny2012internal}), we have the following shape theorem as a corollary: 
\begin{corollary}\label{shapetheorem}
	For any $ u>0$ and $d\ge 3$, recall $T_2(u,d)$ in Theorem \ref{theorem1}. For all $T> T_2$, there exists a compact convex set $D_{u,T} \subset \mathbb{R}^d$ such that $\forall \epsilon \in (0,1)$, there exists a $P(\cdot | 0\in \Gamma)$-almost surely finite random variable $N_{\epsilon,u,T}$ satisfying $\forall n \ge N_{\epsilon,u,T}$, 
	\begin{equation}
		\left((1-\epsilon)n D_{u,T} \cap \Gamma \right) \subset \Lambda^{u,T}(n) \subset (1+\epsilon)n D_{u,T},
	\end{equation} 
	where $\Lambda^{u,T}(n)=\left\lbrace y\in \Gamma: \rho(0,y)\le n \right\rbrace $.
\end{corollary}
We also prove that FRI has local uniqueness property:
\begin{theorem}\label{theorem2}
	For any $ u>0$ and $d\ge 3$, there exists $0<T_3(u,d)<\infty$ such that $\forall T>T_3$, $\exists \delta(u,d,T)>0$, $c(u,d,T)>0$, $c'(u,d,T)>0$ satisfying: for any integer $R> 0$,
	\begin{align}\label{local uniqueness}
	P^{u,T}\bigg[ &\exists two\ clusters\ in\ \mathcal{FI}^{u,T}\cap B(R)\ having\ diameter\ at\ least\\  \nonumber
	&\frac{R}{10}\ not\ connected\ to\ each\ other\ in\ B(2R) \bigg]  \le ce^{-c'R^\delta}.
	\end{align}
\end{theorem}
Moreover, for any $l^\infty $ box $B(N)$, it can be shown that there exists a left-right crossing path in FRI with high probability.
\begin{theorem}\label{theorem3}
	For any $u>0$ and $d\ge 3$, there exists $0<T_4(u,d)<\infty$ such that $\forall T>T_4$, $\exists c(u,d,T)>0$ such that $\forall N>0$,
	\begin{equation}
	P^{u,T}[LR(N)]\ge 1-e^{-cN^{d-1}}, 
	\end{equation} 
	where $LR(N)$ is the event $\left\lbrace left\ of\ B(N) \xleftrightarrow[B(N)]{\mathcal{FI}^{u,T}}right\ of\ B(N)\right\rbrace $ (precise definition of the notation $\xleftrightarrow{}$ will be introduced in Section \ref{notations}) and $left\ of\ B(N):=\left\lbrace x\in B(N):x^{(1)}=-N \right\rbrace $, $right\ of\ B(N):=\left\lbrace x\in B(N):x^{(1)}=N \right\rbrace $ ($x^{(1)}$ is the first coordinate of the vertex $x$).
\end{theorem}

For $\omega \in \{0,1\}^{\mathbb{Z}^d}$ and $x \in \mathbb{Z}^d$, define $\mathcal{S} := \{ x \in \mathbb{Z}^d: \omega (x) = 1 \}$, and let $$deg_{\omega} (x) := \{ y \in \mathcal{S}: |y - x|_1 =1 \}$$ be the degree of $x$ in $\mathcal{S}$. We define a random walk $(X_n)_{n \geq 0}$ on $\mathcal{S}$ with initial position $X_0 = x$ and the following transition kernel:
\begin{equation}
P_{\omega,x} (X_{n+1} =z| X_n = y) =  \left\{
\begin{aligned}
&\frac{1}{2d},   & \text{if } |z-y|_1 =1, z \in \mathcal{S};\\
&1 - \frac{deg_w (y)}{2d},   & \text{if } z=y;\\
&0,   & \text{otherwise}.
\end{aligned}
\right.
\end{equation}
For $n \in \mathbb{N}$ and $t \geq 0$, define
$$
\tilde{B}_n (t) := \frac{1}{\sqrt{n}} \Big(X_{\lfloor tn \rfloor} + (tn- \lfloor tn \rfloor) \cdot (X_{\lfloor tn \rfloor +1} - X_{\lfloor tn \rfloor})\Big).
$$
For $R>0$, let $C[0,R]$ be the space of continuous function from $[0,R]$ to $\mathbb{R}^d$ equipped with supremum norm, and $\mathcal{W}_R$ be the Borel $\sigma$-algebra on $C[0,R]$. By Theorem $1.1$ in \cite{procaccia2016quenched} and Theorem \ref{theorem2}, random walks on the unique infinite cluster $\Gamma$ of a sufficiently supercritical FRI satisfy a quenched invariance principle.

\begin{corollary}
\label{quenched invariance principle}
Let $d \geq 3$ and $0 < \alpha < \beta < \infty$. There exists $0< T_5 (d, \alpha, \beta) < \infty$ such that $\forall T > T_5$, $\forall u \in (\alpha, \beta)$, and $\forall R>0$, for $P^{u,T} (\cdot | 0 \in \Gamma)$-almost every $\omega$ the law of $(\tilde{B}_n (t))_{0 \leq r \leq R}$ on $(C[0,R], \mathcal{W}_R)$ converges weakly to the law of a Brownian motion with zero drift and non-degenerate covariance matrix whose value is a function of $u$ and $T$. In particular, $T_5$ can be chosen as the critical value such that $\forall T > T_5$, (\ref{local uniqueness}) holds and $P^{u,T} (0 \in \Gamma) >0$ for all $u \in (\alpha, \beta)$.
\end{corollary}

We leave the proof of Corollary \ref{quenched invariance principle} in appendix.

\subsection{Open problems}
Let $d_{u}(\cdot, \cdot)$ be the chemical distance in random interlacements $\mathcal{I}^{u}$ and $\rho_{u,T} (\cdot, \cdot)$ be the chemical distance in FRI $\mathcal{FI}^{u,T}$. Given Theorem \ref{theorem1}, it is natural to ask the following questions. Note that question $(1)$ and part of $(2)$ also appear in the open problem section of \cite{procaccia2019percolation}.
\begin{enumerate}
\item In \cite{bowen2019finitary}, it has been proved that for all $u > 0$, the measure of FRI $\mathcal{FI}^{u,T}$ converges to the one of random interlacements $\mathcal{I}^{2du}$ weakly as $T$ goes to infinity. A natural question is to prove that the chemical distance in $\mathcal{FI}^{u,T}$ converges to the one in $\mathcal{I}^{2du}$ as $T \rightarrow \infty$, i.e. for all $u>0$,
$$
\lim_{T \rightarrow \infty} \lim_{||x||_2 \rightarrow \infty} \frac{\rho_{u,T} ([0],[x])}{||x||_2} = \lim_{||x||_2 \rightarrow \infty} \frac{d_{2du} ([0],[x])}{||x||_2},
$$
where $[x]$ denotes the closest vertex in the appropriate infinite cluster to $x \in \mathbb{Z}^d$.
\item Let $0 < \alpha < \beta < \infty$. For $u \in (\alpha, \beta)$, fix a large enough $T$ (depending on $\alpha$ and $\beta$). It would be interesting to prove that the function
$$
u \rightarrow \lim_{||x||_2 \rightarrow \infty} \frac{\rho_{u,T} ([0],[x])}{||x||_2}
$$
is continuous on $(\alpha, \beta)$. For Bernoulli percolation, the continuity of chemical distance has been proved in \cite{garet2017continuity}. Similarly, fix $u>0$, one may ask whether the function
$$
T \rightarrow \lim_{||x||_2 \rightarrow \infty} \frac{\rho_{u,T} ([0],[x])}{||x||_2}
$$ 
is continuous on $(T_2(u,d), \infty)$, where $T_2(u,d)$ is the same critical parameter in Theorem \ref{theorem1}.
\item Through private communications with Eviatar Procaccia, we believe that under proper scaling, the chemical distance in random interlacements $\mathcal{I}^u$ converges to $\ell_2$ distance as $u$ goes to $0$. When $u$ is small, geodesic on $\mathcal{I}^u$ consists of very long random walk trajectories, which become Brownian motions under proper scaling. Similar for FRI, we conjecture that
$$
\lim_{u \rightarrow 0} \lim_{T \rightarrow \infty} \lim_{||x||_2 \rightarrow \infty} \frac{ \rho_{u,T} ([0],[x])}{A(||x||_2)}
$$
converges to some constants, where $A(\cdot)$ is some scaling functions.
\end{enumerate}

\section{Notations and some useful facts}\label{notations}
\textbf{Some basic notations}: In the rest of this paper, we denote the $l^\infty$ distance by $|\cdot|$ and the Euclidean distance by $|\cdot|_2$. We also denote $B_x(R)=\left\lbrace z\in \mathbb{Z}^d: |z-x|\le R \right\rbrace $ with abbreviation $B_0(R)=B(R)$. For any sets $A,B\subset \mathbb{Z}^d$, $d(A,B)=\min\left\lbrace |x-y|:x\in A, y\in B \right\rbrace $. 

Let $\mathbb{L}^d$ be the set of all nearest-neighbor edges in $\mathbb{Z}^d$ (i.e. $\mathbb{L}^d=\left\lbrace \{x,y\}:|x-y|_2=1,x,y\in \mathbb{Z}^d\right\rbrace $). For $1\le i\le d$, we denote by $e_i$ the unit vector on the i-th coordinate. For finite subset $D\subset \mathbb{Z}^d$, let $\partial D=\left\lbrace x\in D:\exists z\in \mathbb{Z}^d\setminus D\ such\ that\ \{x,z\}\in \mathbb{L}^d \right\rbrace $ be its internal boundary and $diam(D)=\max\limits\left\lbrace |x-z|:x,z\in D \right\rbrace $. Without causing confusion, for connected close set $E\subset \mathbb{R}^d$, we also denote its boundary by $\partial E$ (i.e. $\partial E=\{x\in E:\forall r>0,\exists z\in \mathbb{R}^d\setminus E\ such\ that\ |z-x|<r\}$). 

\textbf{Notations about random walks}: Let $P_x$ be the law of a simple random walk starting from $x$. For any simple random walk $\left\lbrace X_i \right\rbrace_{i=0}^{\infty}$ and $A\subset \mathbb{Z}^d$, define $H_A=\min\left\lbrace k\ge 0:X_k\in A \right\rbrace $ and $\bar{H}_A=\min\left\lbrace k\ge 1:X_k\in A \right\rbrace$. In addition, we extend the subscript of random walks to $\mathbb{R}^+$: $\forall t> 0$, define $X_t:=X_{\lfloor t\rfloor}$, where $\lfloor t\rfloor=\max\left\lbrace m\in \mathbb{Z}:m\le t \right\rbrace $. Assuming $\eta=\{\eta(j):0\le j\le l\}$ is a trajectory of the random walk (note that $l$ can be finite or infinite), then we denote by $R(\eta, t)=\left\lbrace \eta(j):0\le j\le \lfloor t\rfloor\land l \right\rbrace $ for $t>0$ the range of this tragectory up to time $t$.

\textbf{Green's funtion}: We need a famous function called the Green's function. See Section 1.5 and Section 1.6 of \cite{lawler2013intersections}, the Green's function is defined as: for any $x,y\in \mathbb{Z}^d$, \begin{equation}
	G(x,y)=\sum\limits_{j=0}^{\infty}P_x(X_j=y).
\end{equation}

\textbf{Statements about constants}: In this article, we will use $c,c_1,c_2,c'...$ as local constants (``local'' means their values may vary according to contexts) and $C,C_1,C_2,C',...$ as global constants (``global'' means constants will keep their values in the whole paper).

\textbf{Stretched exponential}: We say $f(n)$ is s.e.-small if $0\le f(n)\le c_1e^{-c_2n^{c_3}}$ for any $n\ge 1$, and write $f(n)=\boldsymbol{s.e.}(n)$ (See Definition 2.1, \cite{vcerny2012internal}). Moreover, when we use $\boldsymbol{s.e.}(T)$, the corresponding constants $c_1,c_2,c_3$ only depends on $u$ and $d$. When we use $\boldsymbol{s.e.}(\cdot)$ for $N,m,...$(some parameters not rely on $T$), the corresponding constants $c_1,c_2,c_3$ can also depend on $T$ but not only $u$ and $d$.

\textbf{Two kinds of capacities}: We will use two notations about capacity, which are defined as follows: for finite subset $K\subset \mathbb{Z}^d$, $cap(K):=\sum_{x\in K}P_x(\bar{H}_K=\infty)$; for any $T>0$, $cap^{(T)}(K):=\sum_{x\in K}P_x^{(T)}(\bar{H}_K=\infty)$. One may see by definition that $cap^{(T)}(K)\ge cap(K)$. By Proposition 6.5.2 of \cite{lawler2010random}, there exists $c_1,c_2>0$ such that for any $R>0$,   \begin{equation}\label{20}
c_1R^{d-2}\le cap(B(R)) \le c_2R^{d-2}.
\end{equation}

\textbf{Decomposition of FRI}: By the thinning property of Poisson point process, $\mathcal{FI}^{u,T}$ can be discribed as the union of several independent sub-processes. I.e. 
$$\mathcal{FI}^{u,T}\overset{d}{=}\bigcup\limits_{i=1}^{k}\mathcal{FI}^{u_i,T},$$ where $\sum_{i=1}^{k}u_i=u$ and $\left\lbrace \mathcal{FI}^{u_i,T} \right\rbrace_{i=1}^{k}$ are independent.

\textbf{Independence within FRI}: For any disjoint subsets $A_1,...,A_m \subset \mathbb{Z}^d$ and events $E_1,...,E_m$, by Definition \ref{definition2}, if the event $E_i$ only depends on the paths starting from $A_i$ for all $1\le i\le m$, then $E_1,E_2,...,E_m$ are independent.

\textbf{Stochastic domination}: We need two kinds of stochastic domination: one is between random variables and the other is between random collections of sets/bonds.

In details, if $N_1$ and $N_2$ are two non-negative random variables and satisfy that for any $x\ge 0$, $P\left(N_1\ge x \right)\ge P\left(N_2\ge x \right)$, then we say $N_1$ stochastically dominates $N_2$ and denote $N_1 \succeq   N_2 $.

Meanwhile, for two site (or bond) percolation $\left\lbrace Y_j(y) \right\rbrace_{y\in V} $ (or $\left\lbrace Y_j(y) \right\rbrace_{y\in E} $), $j\in \{1,2\}$ on the same countable graph $G=(V,E)$, we denote the state space by $\Omega=\{0,1\}^{V}$ (or $\Omega=\{0,1\}^{E}$). For any $\omega_1,\omega_2\in \Omega$, we write $\omega_1\ge \omega_2$ if $\forall y\in V$ (or $y\in E$), $\omega_1(y)\ge \omega_2(y)$. We say a function $f:\Omega \to \mathbb{R}$ is increasing if $\forall \omega_1\ge \omega_2$, $f(\omega_1)\ge f(\omega_2)$. If for any increasing function $f$ on $\Omega$, $E^{Y_1}\left[f(\omega)\right]\ge E^{Y_2}\left[f(\omega)\right]$, then we say $\left\lbrace Y_1(y) \right\rbrace_{y\in V}$ (or $\left\lbrace Y_1(y) \right\rbrace_{y\in E}$) stochastically dominates $\left\lbrace Y_2(y) \right\rbrace_{y\in V} $ (or $\left\lbrace Y_2(y) \right\rbrace_{y\in E}$). By Lemma 1.0 of \cite{liggett1997domination}, there exists $\left\lbrace \hat{Y}_j(y) \right\rbrace_{y\in V} $ (or $\left\lbrace \hat{Y}_j(y) \right\rbrace_{y\in E} $), $j\in \{1,2\}$ defined on the same probability space such that for $j\in \{1,2\}$, $\left\lbrace \hat{Y}_j(y) \right\rbrace_{y\in V}\overset{d}{=}\left\lbrace Y_j(y) \right\rbrace_{y\in V} $ (or $\left\lbrace \hat{Y}_j(y) \right\rbrace_{y\in E}\overset{d}{=}\left\lbrace Y_j(y) \right\rbrace_{y\in E} $) and $\forall y\in V$ (or $y\in E$), $\hat{Y}_1(y)\ge \hat{Y}_2(y)$ almost surely.

\textbf{Set of ``lucky" paths} (a finite path is ``lucky" if it still goes forward at least $T$ steps after hitting the given set.): Define a mapping $\pi_A$: for any finite path $s=\left(s_0,s_1,...,s_k \right)$, if $s\cap A\neq \emptyset$, let $m=min\left\lbrace t\le k: s_t\in A \right\rbrace $ and   $\pi_A(s)=\left(s_m,s_{m+1},...,s_k \right) $; otherwise $\pi_A(s)=\emptyset$. We write $length(s)=k$.
 
For any subset $A\subset B_x(1.5T^{0.5})$ and a point measure $\mathcal{K}$ composed of nearest-neighbor paths, let $S_x(\mathcal{K},A)$ be the set of paths $\eta$ in $\mathcal{K}$, which start from $B_x(8T^{0.5})\setminus B_x(6T^{0.5})$ with $length(\pi_A(\eta))\ge T $. Especially, let $S_x(\mathcal{K})$ be the set of all the paths starting from $B_x(8T^{0.5})\setminus B_x(6T^{0.5})$ in $\mathcal{K}$ (note that there is no length requirement for paths in $S_x(\mathcal{K})$). We also denote that $\bar{S}_x(\mathcal{K},A)=\left\lbrace \pi_A(\eta): \eta\in S_x(\mathcal{K},A) \right\rbrace $.
\begin{remark}
	In this article, the notation $\sum_{i=s}^{t}$ has the same meaning as $\sum_{i=\lfloor s\rfloor }^{\lfloor t\rfloor }$.
\end{remark}
	\begin{lemma}\label{lemma1}
		There exists $c(u,d)>0$ such that for any $A\subset B(1.5T^{0.5})$ and $y\in B(8T^{0.5})\setminus A$, \begin{equation}\label{4}
		P_y(H_A< T)\ge cT^{\frac{2-d}{2}}cap(A).
		\end{equation}
	\end{lemma}
	\begin{proof}
		For any $y\in B(8T^{0.5})\setminus A$, let $L_A$ be the last time that a simple random walk starting from $y$ hits $A$, then 
		\begin{equation}\label{5}
		\begin{split}
		P_y(H_A< T)\ge& P_y(L_A <T)\\
		\ge &\sum_{z\in A}\sum_{m=1}^{T-1}P_y\left(X_m=z, \forall n>m, X_n \notin A \right) \\
		=&\sum_{z\in A}\left( \sum_{m=1}^{T-1}P_y\left(X_m=z \right) \right) *P_z\left(\bar{H}_A=\infty \right)
		\end{split}
		\end{equation}
		Define that $p_m(x)=P_0(X_m=x)$ and $\bar{p}_m(x)=2\left( \frac{d}{2\pi m}\right)^{\frac{d}{2}}e^{-\frac{d|x|_2^2}{2m}}\cdot1_{\left\lbrace p_m(x)>0 \right\rbrace }$. Let $E(m,x)=|p_m(x)-\bar{p}_m(x)|$. By Theorem 1.2.1, \cite{lawler2013intersections},  \begin{equation}\label{2.3}
		|E(m,x)|\le cm^{-\frac{d+2}{2}}.
		\end{equation}
		Thus we have
		\begin{equation}\label{E}
			\sum_{m=\frac{T}{2}}^{T-1}|E(m,x)|\le c\sum_{m=\frac{T}{2}}^{T-1}m^{-\frac{d+2}{2}}\le cT^{-\frac{d}{2}}.
		\end{equation}
		If $\sum_{i=1}^d y^{(i)}-z^{(i)}$ is even, since $|y-z|_2\le 10\sqrt{d}T^{0.5}$, 
		\begin{equation}\label{pb}
			\begin{split}
			\sum_{m=\frac{T}{2}}^{T-1}\bar{p}_m(y-z)=&\sum_{m=\frac{T}{2}}^{T-1}2\left( \frac{d}{2\pi m}\right)^{\frac{d}{2}}e^{-\frac{d|y-z|_2^2}{2m}}1_{\left\lbrace m\ is\ even\right\rbrace }\\
			=&\sum_{m=\frac{T}{4}}^{\frac{T-1}{2}}2\left( \frac{d}{4\pi m}\right)^{\frac{d}{2}}e^{-\frac{d|y-z|_2^2}{4m}}\\
			\ge &e^{-\frac{100d^2T}{4*\lfloor\frac{T}{4}\rfloor}}*\sum_{m=\frac{T}{4}}^{\frac{T-1}{2}}2\left( \frac{d}{4\pi m}\right)^{\frac{d}{2}}\\
			\ge &e^{-\frac{100d^2T}{4*\lfloor\frac{T}{4}\rfloor}}*\int_{\frac{T}{4}}^{\frac{T-1}{2}}2\left( \frac{d}{4\pi t}\right)^{\frac{d}{2}}dt \ge cT^{\frac{2-d}{2}}.
			\end{split}
		\end{equation}
		If $\sum_{i=1}^d y^{(i)}-z^{(i)}$ is odd, without loss of generality, assume $|y-z+e_1|_2\ge |y-z|_2$ and use (\ref{pb}),
		\begin{equation}\label{pbo}
			\sum_{m=\frac{T}{2}}^{T-1}\bar{p}_m(y-z)\ge \sum_{m=\frac{T}{2}}^{T-1}\bar{p}_m(y-z+e_1)\ge cT^{\frac{2-d}{2}}.
		\end{equation}
		By (\ref{E}), (\ref{pb}) and (\ref{pbo}),  
		\begin{equation}\label{py}
			 \sum_{m=1}^{T-1}P_y\left(X_m=z \right)\ge  \sum_{m=\frac{T}{2}}^{T-1}P_y\left(X_m=z \right)\ge cT^{\frac{2-d}{2}}-O(T^{-\frac{d}{2}}).
		\end{equation}
		Combining (\ref{5}) and (\ref{py}), we conclude the proof of (\ref{4}).
	\end{proof}
In this paper, we will primarily focus on $S_x(\mathcal{FI}^{u,T},A)$. The following lemma guarantees that there are sufficiently many ``lucky paths'' traversing a given subset. This will turn out crucial in adapting the estimates in \cite{vcerny2012internal,rath2011transience} to FRI.
\begin{lemma}\label{lemma2}There exists $c=c(d,u)>0$ such that for any $A\subset B_x(1.5T^{0.5})$,  
	\begin{equation}
	|S_x(\mathcal{FI}^{u,T},A)|\succeq Pois(c*cap(A)).
	\end{equation}
\end{lemma}
\begin{proof}
	By Definition \ref{definition2}, $|S_x(\mathcal{FI}^{u,T})|\sim Pois\left( \frac{2du}{T+1}*O(T^{\frac{d}{2}})\right) $. For any trajectory $\eta \in S_x(\mathcal{FI}^{u,T})$, if $\eta(0)$ is given, we have
	\begin{equation}\label{2.8}
		P_{\eta(0)}\left(length(\pi_A(\eta)\ge T) \right)\ge P_{\eta(0)}\left(length(\eta)\ge 2T, H_A< T \right)=\left(1-\frac{1}{T+1} \right)^{2T+1}*P_{\eta(0)}\left(H_A< T \right).
	\end{equation}
	Combine Lemma \ref{lemma1} and (\ref{2.8}),
	\begin{equation}
		P_{\eta(0)}\left(\eta \in S_x(\mathcal{FI}^{u,T},A) \right)\ge c*T^{\frac{2-d}{2}}cap(A).
	\end{equation}
	Since all the paths in $S_x(\mathcal{FI}^{u,T})$ are independent given their starting points, \begin{equation}
		|S_x(\mathcal{FI}^{u,T},A)|\succeq Pois(c*\frac{1}{T+1}*T^{\frac{d}{2}}*T^{\frac{2-d}{2}}cap(A))\succeq Pois(c*cap(A)).
	\end{equation}
\end{proof}

\textbf{Connection between two sets}: For any sets $A,B\subset \mathbb{Z}^d$, there are two connectivity relationships between them: the first is ``connected by successive vertices'' and the seccond is ``connected by successive edges''. If $D\subset \mathbb{Z}^d$ is a point set such that $\exists$ a finite sequence $(x_0,x_1,...,x_n)$ in $D$ satisfying $x_0\in A$, $x_n\in B$ and $\forall 0\le i\le n-1$, $\{x_i, x_{i+1}\}\in \mathbb{L}^d$, then we say $A$ and $B$ are connected by $D$ and denote it by $A\xleftrightarrow{D} B$.  If an edge set $E\subset \mathbb{L}^d$ satisfying that $\exists$ a finite sequence $(x_0,x_1,...,x_n)$ such that  $x_0\in A$, $x_n\in B$ and $\forall 0\le i\le n-1$, $\{x_i,x_{i+1}\}\in E$, without causing confusion in symbols, we also say $A$ and $B$ are connected by $E$ and denote it by $A\xleftrightarrow{E} B$. For a site percolation $\left\lbrace Z(x) \right\rbrace_{x\in \mathbb{Z}^d} $ and $D\subset \mathbb{Z}^d$, we denote that $A \xleftrightarrow{Z} B=A\xleftrightarrow{\left\lbrace x\in \mathbb{Z}^d:Z(x)=1 \right\rbrace }B$ and $A \xleftrightarrow[D]{Z} B=A\xleftrightarrow{\left\lbrace x\in \mathbb{Z}^d:Z(x)=1 \right\rbrace\cap D }B$. For a bond percolation $\left\lbrace S(e) \right\rbrace_{e\in \mathbb{L}^d}$, we define $A \xleftrightarrow{S} B=A\xleftrightarrow{\left\lbrace e\in \mathbb{L}^d:S(e)=1 \right\rbrace }B$ and $A\xleftrightarrow[D]{S}B=\left\lbrace \exists\ (x_0,x_1,...,x_n)\ in\ D\ satisfying\ x_0\in A, x_n\in B\ and\ \forall 0\le i\le n-1, S(\{x_i,x_{i+1}\})=1\right\rbrace$.

Especially, $\mathcal{FI}^{u,T} $ can be considered as a bond percolation on $\mathbb{L}^d$ (i.e. an edge $e\in \mathbb{L}^d$ is open if $\exists$ a trajectory $\eta$ in $\mathcal{FI}^{u,T}$ such that $e$ is a part of $\eta$). For simplicity, we denote $A\xleftrightarrow{} B=A\xleftrightarrow{\mathcal{FI}^{u,T}} B$ and $A\xleftrightarrow[D]{} B=A\xleftrightarrow[D]{\mathcal{FI}^{u,T}} B$.

\textbf{The largest connected cluster}: For $A,B\subset \mathbb{Z}^d$ and $E\subset \mathbb{L}^d$, let $\mathcal{C}_A^B=\left\lbrace z\in \mathbb{Z}^d: z\xleftrightarrow{B} A \right\rbrace $ and $\mathcal{C}_A^E=\left\lbrace z\in \mathbb{Z}^d: z\xleftrightarrow{E} A \right\rbrace $. For $D\subset \mathbb{Z}^d$, site percolation $\left\lbrace Z(x) \right\rbrace_{x\in \mathbb{Z}^d}$ and bond percolation $\left\lbrace S(e) \right\rbrace_{e\in \mathbb{L}^d}$, we denote $\mathcal{C}_A^Z=\mathcal{C}_A^{\left\lbrace x: Z(x)=1 \right\rbrace }$, $\mathcal{C}_A^{Z}(D)=\mathcal{C}_A^{\left\lbrace x: Z(x)=1 \right\rbrace\cap D }$, $\mathcal{C}_A^S=\mathcal{C}_A^{\left\lbrace e: S(e)=1 \right\rbrace }$ and $\mathcal{C}_A^{S}(D)=\left\lbrace z\in \mathbb{Z}^d: z\xleftrightarrow[D]{\left\lbrace e: S(e)=1 \right\rbrace }A \right\rbrace $. We also denote $\mathcal{C}_A=\mathcal{C}_A^{\mathcal{FI}^{u,T}}$ and $\mathcal{C}_A(D)=\mathcal{C}_A^{\mathcal{FI}^{u,T}}(D)$.

\textbf{Chemical distance}: For any sets $A, B, D\subset \mathbb{Z}^d$, if $A\xleftrightarrow{D} B$, we define that  $\rho_D(A,B)=\min\left\lbrace n\ge 0: \exists \{x_0,x_1,...,x_n\}\subset D\ such\ that\ A\xleftrightarrow{\{x_0,x_1,...,x_n\}} B \right\rbrace $; if $A \stackrel{D}\nleftrightarrow B$, set $\rho_D(A,B)=\infty$. For $E\subset \mathbb{L}^d$, if $A\xleftrightarrow{E} B$, denote $\rho_E(A,B)=\min\limits_{n\ge 0}\left\lbrace n: \exists \{e_1,...,e_n\}\subset E\ such\ that\ A\xleftrightarrow{\{e_1,...,e_n\}} B \right\rbrace $; if $A \stackrel{E}\nleftrightarrow B$, set $\rho_E(A,B)=\infty$. Similarly, for site percolation $\left\lbrace Z(x) \right\rbrace_{x\in \mathbb{Z}^d}$ and bond percolation $\left\lbrace S(e) \right\rbrace_{e\in \mathbb{L}^d}$, define $\rho_{Z}(A,B):=\rho_{\left\lbrace x\in\mathbb{Z}^d:Z(x)=1  \right\rbrace }(A,B)$ and $\rho_{S}(A,B):=\rho_{\left\lbrace e\in\mathbb{L}^d:S(E)=1\right\rbrace }(A,B)$. 
For simplicity, we also denote $\rho(A,B)=\rho_{\mathcal{FI}^{u,T}}(A,B)$.

\textbf{Large deviation bound for Poisson distribution}: By equation 2.11, \cite{drewitz2014introduction}, if $X\sim Pois(\lambda)$, then 
\begin{equation}\label{19}
	P\left(\frac{\lambda}{2} \le X \le 2\lambda\right) \ge 1-e^{-\frac{\lambda}{10}}.
\end{equation}

\textbf{Uniform bound trick}: Consider two independent FRI's $\mathcal{FI}_1^{u_1,T_1}$, $\mathcal{FI}_2^{u_2,T_2}$ and a finite set $D\subset \mathbb{Z}^d$. We define three sets: $\mathcal{W}=\left\lbrace w=\sum_{i=1}^{\infty}\delta_{\eta_i}: \eta_i \in W^{\left[ 0,\infty\right) } \right\rbrace $, $$W_{|D}=\left\lbrace w=\sum_{i=1}^{M}\delta_{\eta_i}:0\le M<\infty, \eta_i \in W^{\left[ 0,\infty\right) }\ and\ \eta_i(0)\in D \right\rbrace $$ and $$W_{\cap D}=\left\lbrace w=\sum_{i=1}^{M}\delta_{\eta_i}:0\le M<\infty, \eta_i \in W^{\left[ 0,\infty\right) }\ and\ \eta_i\cap D\neq \emptyset \right\rbrace. $$ 
Obviously, $W_{|D}$ and $W_{\cap D}$ are both countable subsets of $\mathcal{W}$. Then we also define two mappings:
$$ \pi_{|D}:\mathcal{W}\to W_{|D},\ \sum_{i=1}^{\infty}\delta_{\eta_i}\to \left\{
\begin{array}{rcl}
&\sum_{i=1}^{\infty}\delta_{\eta_i} \cdot \mathbbm{1}_{\left\lbrace \eta_i(0)\in D\right\rbrace }     \ \       & |\{\eta_i:\eta_i(0)\in D\}|<\infty;\\
&0     \ \        & |\{\eta_i:\eta_i(0)\in D\}|=\infty
\end{array} \right. $$
and 
$$ \pi_{\cap D}:\mathcal{W}\to W_{\cap D},\ \sum_{i=1}^{\infty}\delta_{\eta_i}\to \left\{
\begin{array}{rcl}
&\sum_{i=1}^{\infty}\delta_{\eta_i} \cdot \mathbbm{1}_{\left\lbrace \eta_i\cap D\neq\emptyset\right\rbrace }     \ \       & |\{\eta_i:\eta_i\cap D\neq\emptyset\}|<\infty;\\
&0     \ \        & |\{\eta_i:\eta_i\cap D\neq\emptyset\}|=\infty.
\end{array} \right. $$
According to Section 4.2 of \cite{drewitz2014introduction}, there is a $\sigma-$field $\mathcal{A}$ on $\mathcal{W}$ produced by random variables $\left\lbrace \mu(A):= \sum_{i=1}^{\infty}1_{\{\eta_i\in A\}}:A\subset W^{\left[ 0,\infty\right)} \right\rbrace$. In deed, for $\hat{w}=\sum_{i=1}^{M}\delta_{\eta_i}=\sum_{j=1}^{M'}k_j\delta_{\eta_{n_j}}\in W_{|D}$ (where $k_j\in \mathbb{N}^+$ and $\forall k\neq l$, $\eta_{n_k}\neq \eta_{n_l}$), $$\pi_{|D}^{-1}(\hat{w})=\bigcap\limits_{j=1}^{M'}\left\lbrace\mu(\{\eta_{n_j}\})=k_j \right\rbrace \cap \bigcap\limits_{\eta\in W^{\left[ 0,\infty\right)},\hat{w}(\eta)=0,\eta(0)\in D }\left\lbrace \mu(\{\eta\})=0 \right\rbrace  \in \mathcal{A}.$$
Thus $\pi_{|D}$ is measurable. Similarly, $\pi_{\cap D}$ is also measurable.  

For any events $A\subset W_{|D}$ and $B\in \mathcal{A}\times \mathcal{A}$. If there exists $c>0$ such that for any trajectory $\hat{w} \in A$, $P^{u_1,T_1}\times P^{u_2,T_2}\left( B\bigg|\left( \pi_{|D}^{-1}(\hat{w})\right) \times \mathcal{W}\right) \le c$, since $W_{|D}$ is countable,
\begin{equation}
	\begin{split}
	&P^{u_1,T_1}\times P^{u_2,T_2}\left[\left(  \pi_{|D}^{-1}(A) \times \mathcal{W}\right) \cap B\right] \\
	=&\sum\limits_{\hat{w}  \in A}P^{u_1,T_1}\times P^{u_2,T_2}\left( B\bigg|  \pi_{|D}^{-1}(\hat{w}) \times \mathcal{W}\right) P^{u_1,T_1}(\pi_{|D}^{-1}(\hat{w}))\\
	\le &c*\sum\limits_{\hat{w} \in A}P^{u_1,T_1}(\pi_{|D}^{-1}(\hat{w}))=c*P^{u_1,T_1}(\pi_{|D}^{-1}(A)).
	\end{split}
\end{equation}
Equivalently, if $P^{u_1,T_1}\times P^{u_2,T_2}\left( B\bigg|\left( \pi_{|D}^{-1}(\hat{w})\right) \times \mathcal{W}\right) \ge c$ for any $\hat{w}\in A$, then
 $$P^{u_1,T_1}\times P^{u_2,T_2}\left[\left(  \pi_{|D}^{-1}(A) \times \mathcal{W}\right) \cap B\right] \ge c*P^{u_1,T_1}(\pi_{|D}^{-1}(A)).$$
Similarly, let $A\subset W_{\cap D}$ and $B\in \mathcal{A}\times \mathcal{A}$, if $P^{u_1,T_1}\times P^{u_2,T_2}\left( B\bigg|\left( \pi_{\cap D}^{-1}(\widetilde{w})\right) \times \mathcal{W}\right) \le c (\ge c)$ for any $\widetilde{w}\in A $, then we also have: $$P^{u_1,T_1}\times P^{u_2,T_2}\left[\left(  \pi_{\cap D}^{-1}(A) \times \mathcal{W}\right) \cap B\right] \le c*P^{u_1,T_1}(\pi_{\cap D}^{-1}(A))\left(\ge c*P^{u_1,T_1}(\pi_{\cap D}^{-1}(A)) \right). $$
We call these tricks \textbf{uniform bound trick} and we will use them many times in our proof.
\section{Infinite sub-cluster $\bar{\Gamma}$ with good chemical distance}\label{section3}
	In this section, we want to construct an infinite connected subset of $\mathcal{FI}^{u,T}$ with good chemical distance i.e. the chemical distance on such sub-cluster has the same order as $l^{\infty}$ distance. The technique used here to construct $\bar{\Gamma}$ is inspired by \cite{vcerny2012internal,rath2011transience}.
	
	Let $a,b, \epsilon>0$ be small enough constants (only depend on $u,d$) and $n=\lfloor bT^{0.5}\rfloor$. Assume $\mathcal{FI}^{u,T}= \mathcal{FI}_1^{\frac{1}{3}u,T} \cup \mathcal{FI}_2^{\frac{1}{3}u,T} \cup \mathcal{FI}_3^{\frac{1}{3}u,T}$ and $\mathcal{FI}_2^{\frac{1}{3}u,T}=\bigcup_{k=1}^{d-2}\mathcal{FI}_{2,k}^{\frac{u}{3(d-2)},T}$, where $\mathcal{FI}_1^{\frac{1}{3}u,T}$, $\mathcal{FI}_{2,1}^{\frac{u}{3(d-2)},T},...,\mathcal{FI}_{2,d-2}^{\frac{u}{3(d-2)},T} , \mathcal{FI}_3^{\frac{1}{3}u,T}$ are independent. To simplify our symbols, for $1\le k\le d-2$, we define
	$S_x^{(2,k)}(\cdot)=S_x(\mathcal{FI}_{2,k}^{\frac{u}{3(d-2)},T},\cdot)$, $\bar{S}_x^{(2,k)}(\cdot)=\bar{S}_x(\mathcal{FI}_{2,k}^{\frac{u}{3(d-2)},T},\cdot)$ and  $S_x^{(2,k)}=S_x(\mathcal{FI}_{2,k}^{\frac{u}{3(d-2)},T})$. For $i\in\{1,2,3\}$, we denote   $S_x^{(i)}(\cdot)=S_x(\mathcal{FI}_i^{\frac{1}{3}u,T},\cdot)$,  $\bar{S}_x^{(i)}(\cdot)=\bar{S}_x(\mathcal{FI}_i^{\frac{1}{3}u,T},\cdot)$ and   $S_x^{(i)}=S_x(\mathcal{FI}_i^{\frac{1}{3}u,T})$. 
	\subsection{Good line segments}
	The good line segments are the building blocks of our infinite sub-cluster $\bar{\Gamma}$, and here is the strategy of constructing a good line segment.
	
We will use the technique introduced in Section 6 of \cite{vcerny2012internal}. In \cite{vcerny2012internal}, a  cluster is constructed by connecting all paths hitting a set $G_a^{(n)}=\bigcup\limits_{k=0}^nB(ke_1,n^a)$ in the random interlacements. By the definition of random interlacements, all these paths are infinite. In order to adopt their approach, we consider all paths in $\bar{S}^{(1)}(G_a^{(x,i)})$ ($G_a^{(x,i)}$ is a set similar to $G_a^{(n)}$) and extend them to be trajectories of simple random walks starting from $\partial G_a^{(x,i)}$. We will claim that the cluster produced by these extended paths satisfies some properties and leave the proof in our appendix. Finally we show that with high probability our desired line segment is unchanged after the extending.

Similar to \cite{vcerny2012internal}, let $G_a^{(x,i)}=\bigcup\limits_{j=-n}^{n}B_{x+j \boldsymbol{e}_i}(n^a)$ and for integer $k\in [-n^{1-a}-1,n^{1-a}+1]$, define $U_k^{(x,i)}=B_{x+[kn^a]e_i}(n^a)$ and  $\bar{S}_x^{(1)}(G_a^{(x,i)}, k)=\left\lbrace \eta \in \bar{S}_x^{(1)}(G_a^{(x,i)}):\eta(0)\in U_k^{(x,i)}  \right\rbrace $. In words, $\bar{S}_x^{(1)}(G_a^{(x,i)}, k)$ is a subset of $\bar{S}_x^{(1)}(G_a^{(x,i)})$, composed of paths starting from $U_k^{(x,i)}\cap \partial G_a^{(x,i)}$.
 
 Enumerate $\bar{S}_x^{(1)}(G_a^{(x,i)})$ by $X^{(k)}$, $1\le k\le |S_x^{(1)}(G_a^{(x,i)})|$. We define:
\begin{itemize}
	\item Let $L^{(k)}$ be the length of $ X^{(k)}$ for all $k$. Define  $$\hat{X}^{(k)}(i)=X^{(k)}(i \land L^{(k)})+\mathbbm{1}_{\left\lbrace i > L^{(k)}\right\rbrace }\cdot Y^{(k,x)}(i-L^{(k)}), $$
	where $\left\lbrace Y^{(k,x)}\right\rbrace_{x\in \mathbb{Z}^d,k\in \mathbb{N}^+} $ are independent simple random walks starting from $0$ and independent to $\mathcal{FI}^{u,T}$. Obviously, $\hat{X}^{(k)}$ has the same distribution as a simple random walk starting from $X^{(k)}(0)$.
	\item $j_k^{x,i}=\inf\left\lbrace j\ge 0: \hat{X}^{(k)}_{n^{2(a+\epsilon)}+(j-1)n^{2a}+i}\notin G_a^{(x,i)}, \forall i=0,1,...,n^{2a} \right\rbrace $;
	\item $\hat{t}^{x,i}_k=n^{2(a+\epsilon)}+j_k^{x,i}*n^{2a}$;
	\item $\hat{\mathcal{I}}_{x,i}=\bigcup_{k=1}^{|S_x^{(1)}(G_a^{(x,i)})|}R_k(\hat{t}_k^{x,i})$, where $R_k(t):=R(\hat{X}^{(k)},t)$ (recall $R(\cdot,\cdot)$ in Section \ref{notations});
	\item $ \bar{\mathcal{I}}_{x,i}=\bigcup_{k=1}^{|S_x^{(1)}(G_a^{(x,i)})|}R_k(\hat{t}_k^{x,i}\land L^{(k)})$.  
	\item $\hat{\rho}_{x,i}$ is the chemical distance on $\hat{\mathcal{I}}_{x,i}$ and $\bar{\rho}_{x,i}$ is the chemical distance on $\bar{\mathcal{I}}_{x,i}$.
	\item $l_{x,i}=\left\lbrace x+k*e_i: k\in [-n,n]\cap\mathbb{Z}^d \right\rbrace $; 
	\item For any $y\in l_{x,i}$, let $\hat{\zeta}_1^{x,i}(y)=\inf\left\lbrace j\ge 1: y+j*e_i\in \hat{\mathcal{I}}_{x,i} \right\rbrace $ and $\hat{\varphi}_1^{x,i}(y)=y+\hat{\zeta}_1^{x,i}(y)*e_i$.
	\item For $\bar{\mathcal{I}}_{x,i}$, we define $\bar{\zeta}_1^{x,i}(y)$ and $\bar{\varphi}_1^{x,i}(y)$ similar to $\hat{\zeta}_1^{x,i}(y)$ and $\hat{\varphi}_1^{x,i}(y)$ respectively.
\end{itemize}

We claim that $\hat{\mathcal{I}}_{x,i}$ shares properties of $\hat{\mathcal{I}}$ in Section 6 of  \cite{vcerny2012internal} and we leave the proof in the appendix, since the proofs are similar.
\begin{claim}\label{claim1}
	The following events occur with probability $1-\boldsymbol{s.e.}(T)$:
	\begin{enumerate}
		\item for any integer $k\in [-n^{1-a}-1,n^{1-a}+1]$ and two paths $X^{(p)}$ and $X^{(q)}$ in $\bar{S}_x^{(1)}(G_a^{(x,i)},k)$, $\hat{\rho}_{x,i}\left(R\left(\hat{X}^{(p)},2n^{2a} \right) ,R\left(\hat{X}^{(q)},2n^{2a} \right)\right)\le cn^{2a} $, where $c=c(d,u)>0$ is a constant;
		\item $\forall 1\le k\le |S_x^{(1)}(G_a^{(x,i)})|$, $j_k^{x,i} < n^\epsilon$;
		\item $\forall y\in l_{x,i}$, $|\hat{\varphi}_1^{x,i}(y)-y|\le T^\epsilon$;
		\item $\hat{\mathcal{I}}_{x,i}$ is connected;
		\item $\forall y_1,y_2\in l_{x,i}\cap \hat{\mathcal{I}}_{x,i}$, $\hat{\rho}_{x,i}\left(y_1,y_2\right)\le cn $, where $c=c(d,u)$ is a constant.
	\end{enumerate}
\end{claim}
\begin{definition}(good line segments)\label{goodlineseg}
		We say the line segment $l_{x,i}$ is good if $\bar{\mathcal{I}}_{x,i}$ satisfies
	\begin{enumerate}
		\item $\bar{\mathcal{I}}_{x,i}$ is connected.
		\item $\forall y\in l_{x,i}$, $|\bar{\varphi}_1^{x,i}(y)-y|\le T^\epsilon$;  
		\item $\forall y_1,y_2\in\bar{\mathcal{I}}_{x,i}$, $\bar{\rho}_{x,i}\left(y_1,y_2 \right)\le  cn$, where $c=c(d,u)$ is a constant. 
	\end{enumerate}
\end{definition}
Based on Claim \ref{claim1}, we can prove that any line segment with length $2n$ is good with high probability.
\begin{lemma}\label{goodline}
For any $x\in \mathbb{Z}^d$ and $1\le i\le d$, \begin{equation}\label{3.18}
	P\left(l_{x,i}\ is\ good\right)\ge 1-\boldsymbol{s.e.}(T).
	\end{equation}
\end{lemma}
\begin{proof}
	For any $X^{(k)}\in \bar{S}_x^{(1)}(G_a^{(x,i)},k)$, we have $L^{(k)}\ge T$. Since $P\left(j_k^{x,i}<n^\epsilon \right)\ge 1-\boldsymbol{s.e.}(T) $ (recall (2) of Claim \ref{claim1}), $P\left(|\bar{S}_x^{(1)}|\le cT^{\frac{d}{2}} \right)\ge 1-\boldsymbol{s.e.}(T) $ and $L^{(k)}\ge T> 2n^{2(a+\epsilon)}$, we have \begin{equation}
	\begin{split}
	P\left(\bar{\mathcal{I}}_{x,i}=\hat{\mathcal{I}}_{x,i}  \right)=&P\left(\bigcap_{k=1}^{|\bar{S}_x^{(1)}(G_a^{(x,i)},k)|}\left\lbrace \hat{t}^{x,i}_k\le L^{(k)}  \right\rbrace  \right)\\
	\ge &P\left(\bigcap_{k=1}^{|\bar{S}_x^{(1)}(G_a^{(x,i)},k)|}\left\lbrace j_k^{x,i}<n^\epsilon  \right\rbrace, |\bar{S}_x^{(1)}(G_a^{(x,i)},k)|\le cT^{\frac{d}{2}}  \right)\\
	\ge &\left( 1-\boldsymbol{s.e.}(T)\right)^{cT^{\frac{d}{2}}}*\left(1-\boldsymbol{s.e.}(T) \right) =1-\boldsymbol{s.e.}(T).
	\end{split}	
	\end{equation}
	
	Now it's sufficient to confirm  that $\hat{\mathcal{I}}_{x,i}$ satisfies these three conditions. By (3) and (4) of Claim \ref{claim1}, we know that condition (1) and (2) hold for $\hat{\mathcal{I}}_{x,i}$ with probability $1-\boldsymbol{s.e.}(T)$.
	
Let $\mathcal{L}_{k}=\left\lbrace x+le_i:k-n^a< l< k+n^a \right\rbrace\cap l_{x,i}$. For any $y_j\in \hat{\mathcal{I}}_{x,i}$, $j=1,2$, assume that $y_j\in X^{(l_j)} \in \bar{S}_x^{(1)}\left(G_a^{(x,i)},k_j\right) $. By Lemma 3.2 of \cite{vcerny2012internal}, for any $x\in \partial U_{k_j}^{(x,i)}$, 
	\begin{equation}\label{px}
	P_x\left(H_{\mathcal{L}_{k_j}}\le n^{2(a+\epsilon)} \right) \ge  \left\{
	\begin{aligned}
	&\frac{cn^a}{n^{a(d-2)}\ln(n)},    & d=3;\\
	&\frac{cn^a}{n^{a(d-2)}},    &d\ge 4.
	\end{aligned}
	\right.
	\end{equation}
	By Lemma \ref{lemma4} (which can be found in the appendix) and (\ref{px}), we know the number of paths in $\bar{S}_x^{(1)}(G_a^{(x,i)}, k_j)$ hitting $\mathcal{L}_{k_j}$ within $n^{2(a+\epsilon)}$ steps stochastically dominates a Poisson random variable with parameter $\frac{cn^a}{(\ln(n))^2}$ for $d=3$ and $cn^a$ for $d\ge 4$. Using the large deviation bound for Poisson distribution, 
	\begin{equation}\label{lk}
	P\left(there\ exists\ a\ path\ in\ \bar{S}_x^{(1)}(G_a^{(x,i)}, k_j)\ hitting\ \mathcal{L}_{k_j}\ within\ n^{2(a+\epsilon)}\ steps \right)\ge 1-\boldsymbol{s.e.}(T). 
	\end{equation}
	If $\hat{t}^{x,i}_{l_j}\le 2n^{2(a+\epsilon)}$, then $y_j\in R(X^{(l_j)},2n^{2(a+\epsilon)})$. When the event in (\ref{lk}) occurs, there exists $ X^{(m_j)}\in \bar{S}_x^{(1)}(G_a^{(x,i)}, k_j)$ such that $R\left(X^{(m_j)},n^{2(a+\epsilon)} \right)\cap  \mathcal{L}_{k_j}\neq \emptyset$.
	
	In addition, if $\hat{\rho}_{x,i}\left(R\left(\hat{X}^{(m_j)},2n^{2a} \right) ,R\left(\hat{X}^{(k_j)},2n^{2a} \right)\right)\le cn^{2a} $ also happens, then 
	 \begin{equation}\label{3.5}
	 	\hat{\rho}_{x,i}(y_j,\hat{\mathcal{I}}_{x,i}\cap \mathcal{L}_{k_j})\le cn^{2a}+3n^{2(a+\epsilon)}.
	 \end{equation}
By (\ref{lk}), (\ref{3.5}) and (1), (2) of Claim \ref{claim1}, we have: for $j=1,2$, \begin{equation}\label{3.25}
\begin{split}
P\left(\hat{\rho}_{x,i}(y_j,\hat{\mathcal{I}}_{x,i}\cap \mathcal{L}_{k_j})\le cn^{2a}+3n^{a+\epsilon}  \right)\ge 1-\boldsymbol{s.e.}(T).
\end{split}
\end{equation}	
If $\hat{\rho}_{x,i}(y_j,\hat{\mathcal{I}}_{x,i}\cap \mathcal{L}_{k_j})\le cn^{2a}+3n^{2(a+\epsilon)}$, then  $\exists z_j\in l_{x,i}$ such that $\hat{\rho}_{x,i}\left(y_j,z_j \right)\le cn^{2a}+3n^{2(a+\epsilon)}$. By $\hat{\rho}_{x,i}\left(y_1,y_2\right)\le\hat{\rho}_{x,i}\left(y_1,z_1 \right)+\hat{\rho}_{x,i}\left(z_1,z_2 \right) +\hat{\rho}_{x,i}\left(z_2,y_2 \right)$ and (5) of Claim \ref{claim1}, we know that condition (3) holds with probability $1-\boldsymbol{s.e.}(T)$.
\end{proof}	
\subsection{Good sites}\label{goodsites}
This subsection is mainly based on Section 4.2 and Section 4.3 of \cite{rath2011transience}.

If $l_{x,i}$ is good, then for any $y\in l_{x,i}$, $|\bar{\varphi}_1^{x,i}(y)-y|\le T^\epsilon$. For any $m\le n$, denote $Q^m_{x,i,y}=\mathcal{C}_{\bar{\varphi}_1^{x,i}(y)}^{\bar{\mathcal{I}}_{x,i}}\left(B_{\bar{\varphi}_1^{x,i}(y)}(m^{0.5}) \right)  $ and we have $|Q^m_{x,i,y}|\ge m^{0.5}$ since $\bar{\varphi}_1^{x,i}(y) \xleftrightarrow{\bar{\mathcal{I}}_{x,i}} \partial B_{\bar{\varphi}_1^{x,i}(y)}(m^{0.5})$. We also define that $\Psi_y(k,A,m)=\bigcup\limits_{\eta \in \bar{S}^{(2,k)}_y(A)}R\left(\eta,m \right) $. Let $U_{x,i,y}^{(1)}(m)= \Psi_y(1,Q^m_{x,i,y},m)$ and $U_{x,i,y}^{(k)}(m)=\Psi_y(k,U_{x,i,y}^{(k-1)}(m),m)$ for $2 \le k\le d-2$. Especially, we denote that $Q_{x,i,y}=Q^{2n^{2(a+\epsilon)}}_{x,i,y} $ and $U_{x,i,y}^{(k)}=U_{x,i,y}^{(k)}(2n^{2(a+\epsilon)})$, $1\le k\le d-2$.
\begin{lemma}\label{lemma5}For any $x\in \mathbb{Z}^d$, $1\le i\le d$, $m\le n$ and $y\in l_{x,i}$, 
	\begin{equation}\label{18}
	P\left(|\bar{S}_y^{(2,1)}(Q^m_{x,i,y})|\ge 1\big|l_{x,i}\ is\ good \right)\ge 1-\boldsymbol{s.e.}(m).
	\end{equation} 
\end{lemma}
\begin{proof}
	In this proof, we assume that $l_{x,i}$ is good and fix $Q^m_{x,i,y}$.
	
	By Proposition6.5.1, \cite{lawler2010random}, for any $z\in \mathbb{Z}^d$ such that $4T^{0.5} \le |z-y|\le  10T^{0.5}$, 
	\begin{equation}\label{cap1}
	cap(Q^m_{x,i,y})= \frac{cT^{0.5(d-2)}}{1+O(\frac{m^{0.5}}{T^{0.5}})}*P_z\left(H_{Q^m_{x,i,y}}<\infty \right)\ge cT^{0.5(d-2)}*P_z\left(H_{Q^m_{x,i,y}}<\infty \right).
	\end{equation}
	By Lemma 3.3 of \cite{vcerny2012internal},
	\begin{equation}\label{cap2}
	P_z\left(H_{Q^m_{x,i,y}}<\infty \right)\ge \frac{c|Q^m_{x,i,y}|^{1-\frac{2}{d}}}{(11T^{0.5})^{d-2}} \ge cm^{0.5-\frac{1}{d}}*T^{-0.5(d-2)}.
	\end{equation}
	Combine (\ref{cap1}) and (\ref{cap2}), we have 
	\begin{equation}\label{cap3}
		cap(Q^m_{x,i,y})\ge cm^{0.5-\frac{1}{d}}.
	\end{equation}
	By Lemma \ref{lemma2} and (\ref{cap3}), $|S_y^{(2,1)}(Q^m_{x,i,y})|\succeq Pois(cm^{0.5-\frac{1}{d}})$. Then we get (\ref{18}) by the large diviation bound for Poisson distribution and the uniform bound trick.
\end{proof}
We are going to prove the following result by induction: for $1\le k\le d-2$, $$P\left( cap(U_{x,i,y}^{(k)})\ge cn^{k(a+\epsilon)(1-\epsilon_1)}\big|l_{x,i}\ is\ good \right)\ge 1-\boldsymbol{s.e.}(T).$$  First, by Lemma \ref{lemma5}, we know that $U_{x,i,y}^{(1)}$ contains at least one path with probability at least $1-\boldsymbol{s.e.}(T)$. Taking $N=1$, $T=2n^{2(a+\epsilon)}$ in Lemma 6 of \cite{rath2011transience}, we have: for $0<\epsilon_1<1$,  \begin{equation}
	P\left(cap(U_{x,i,y}^{(1)})\ge cn^{(a+\epsilon)(1-\epsilon_1)}\big|l_{x,i}\ is\ good \right)\ge 1-\boldsymbol{s.e.}(T).
\end{equation}
For any $1\le k\le d-3$, if $\left\lbrace l_{x,i}\ is\ good, cap(U_{x,i,y}^{(k)})\ge cn^{k(a+\epsilon)(1-\epsilon_1)} \right\rbrace $ happens, by Lemma \ref{lemma2} and the large deviation bound for Poisson distribution, we know that $U_{x,i,y}^{(k+1)}$ consists of at least $cn^{k(a+\epsilon)(1-\epsilon_1)}$ paths with probability at least $1-\boldsymbol{s.e.}(T)$. Use inductive hypotheses and take $N=cn^{k(a+\epsilon)(1-\epsilon_1)}$, $T=2n^{2(a+\epsilon)}$ in Lemma 6 of \cite{rath2011transience}, \begin{equation}
\begin{split}
&P\left(cap\left( U_{x,i,y}^{(k+1)}\ge cn^{(k+1)(a+\epsilon)(1-\epsilon_1)}\right)\bigg|l_{x,i}\ is\ good\right)\\
\ge &P\left(cap\left( U_{x,i,y}^{(k+1)}\ge cn^{(k+1)(a+\epsilon)(1-\epsilon_1)}\right) \bigg|l_{x,i}\ is\ good,cap\left( U_{x,i,y}^{(k)}\right) \ge cn^{(k)(a+\epsilon)(1-\epsilon_1)} \right)\\
&*P\left(cap\left( U_{x,i,y}^{(k)}\right) \ge cn^{k(a+\epsilon)(1-\epsilon_1)} \bigg|l_{x,i}\ is\ good\right)\\ 
\ge & \left( 1-\boldsymbol{s.e.}(T)\right)*\left( 1-\boldsymbol{s.e.}(T)\right)= 1-\boldsymbol{s.e.}(T).
\end{split}
\end{equation}
Now we complete the induction and get: \begin{equation}\label{3.31}
	P\left( cap(U_{x,i,y}^{(d-2)})\ge cn^{(d-2)(a+\epsilon)(1-\epsilon_1)}\big|l_{x,i}\ is\ good \right)\ge 1-\boldsymbol{s.e.}(T).
\end{equation}

We also prove the following inequalities by induction: for $1\le k\le d-2$,$$P\left(\bigcup_{j=1}^{k} U_{x,i,y}^{(j)}\subset B_y((k+2)n^{(a+\epsilon)(1+\epsilon_1)})\bigg|l_{x,i}\ is\ good\right)\ge 1-\boldsymbol{s.e.}(T).$$By Lemma 3.4 of \cite{vcerny2012internal}, for simple random walk $\left\lbrace X_m\right\rbrace_{m=0}^{\infty} $, $$P\left(diam\left\lbrace R\left(  X_{\bullet},2n^{2(a+\epsilon)}\right)  \right\rbrace\le n^{(a+\epsilon)(1+\epsilon_1)}  \right)\ge 1-\boldsymbol{s.e.}(T).$$ Since $Q_{x,i,y}\subset B_y(2n^{(a+\epsilon)(1+\epsilon_1)})$ and $P\left( |S_y^{(2,1)}|\le cn^d\right)\ge 1-\boldsymbol{s.e.}(T) $, we have \begin{equation}
P\left(U_{x,i,y}^{(1)}\subset B_y(3n^{(a+\epsilon)(1+\epsilon_1)})\bigg|l_{x,i}\ is\ good \right)\ge 1-\boldsymbol{s.e.}(T). 
\end{equation}  
Similarly, by $P\left( |S_y^{(2,k+1)}|\le cn^d\right)\ge 1-\boldsymbol{s.e.}(T) $ , Lemma 3.4 of \cite{vcerny2012internal} and inductive hypotheses,  \begin{equation}
\begin{split}
&P\left(\bigcup_{j=1}^{k+1} U_{x,i,y}^{(j)}\subset B_y((k+3)n^{(a+\epsilon)(1+\epsilon_1)})\bigg|l_{x,i}\ is\ good \right) \\
\ge& P\left(\bigcup_{j=1}^{k+1} U_{x,i,y}^{(j+1)}\subset B_y((k+3)n^{(a+\epsilon)(1+\epsilon_1)})\bigg|l_{x,i}\ is\ good ,\bigcup_{j=1}^{k} U_{x,i,y}^{(j)}\subset B_y((k+2)n^{(a+\epsilon)(1+\epsilon_1)}) \right)\\
&*P\left(\bigcup_{j=1}^{k} U_{x,i,y}^{(j)}\subset B_y((k+2)n^{(a+\epsilon)(1+\epsilon_1)})\bigg|l_{x,i}\ is\ good \right)    \\
\ge &\left(1-\boldsymbol{s.e.}(T) \right)*\left(1-\boldsymbol{s.e.}(T) \right)=1-\boldsymbol{s.e.}(T).
\end{split}
\end{equation}
Now the induction is completed. Particularly, we have: 
\begin{equation}\label{3.34}
	P\left(\bigcup_{j=1}^{d-2} U_{x,i,y}^{(j)}\subset B_y(dn^{(a+\epsilon)(1+\epsilon_1)})\bigg|l_{x,i}\ is\ good \right) \ge 1-\boldsymbol{s.e.}(T).
\end{equation}

Combine Lemma \ref{goodline}, (\ref{3.31}) and (\ref{3.34}), we have the following estimate: for any $0<\epsilon_1<\frac{1}{2}$,
\begin{align}\label{35}
P\left(\bigcup_{j=1}^{d-2} U_{x,i,y}^{(j)}\subset B_y(dn^{(a+\epsilon)(1+\epsilon_1)}),cap(U_{x,i,y}^{(d-2)})\ge cn^{(a+\epsilon)(d-2)(1-\epsilon_1)}\bigg|l_{x,i}\ is\ good\right) \ge 1-\boldsymbol{s.e.}(T), 
\end{align}
where we require $(a+\epsilon)(1+\epsilon_1)<\frac{1}{2d}$.\\
\begin{lemma}\label{connect}
	For any subsets $U$, $V\subset B_y(dn^{(a+\epsilon)(1+\epsilon_1)})$ satisfying $cap(U)\ge cn^{(a+\epsilon)(d-2)(1-\epsilon_1)}$ and $cap(V)\ge cn^{(a+\epsilon)(d-2)(1-\epsilon_1)}$, $0<\epsilon_1<\frac{1}{3}$,  \begin{equation}
	P\left(U \xleftrightarrow[B_y(n^{\frac{1}{d}})]{S_y^{(3)}(U)}V \right) \ge 1-\boldsymbol{s.e.}(T).
	\end{equation}
\end{lemma}
\begin{proof}
	For any $z\in B_y(dn^{(a+\epsilon)(1+\epsilon_1)})$, by Lemma \ref{lemma1},
	\begin{equation}
	P_z\left(H_V<d^2n^{2(a+\epsilon)(1+\epsilon_1)} \right)>c*\left(2dn^{(a+\epsilon)(1+\epsilon_1)} \right)^{2-d}*cap(V).
	\end{equation}
	By Lemma \ref{lemma2}, we know that $|S_y^{(3)}(U)|\succeq Pois(c*cap(U))$, thus $\exists c'(u,d)>0$ such that 
	 $$P\left(|S_y^{(3)}(U)|< c'*cap(U) \right)\le \boldsymbol{s.e.}(T). $$
	Since the paths in $S_y^{(3)}(U)$ are independent given their starting points and $2(a+\epsilon)(1+\epsilon_1)<\frac{1}{d}$, 
	\begin{equation}
		\begin{split}
			&P\left(U \xleftrightarrow[B_y(n^{\frac{1}{d}})]{S_y^{(3)}(U)}V \right)\\
		\ge &1-P\left(|S_y^{(3)}(U)|< c'*cap(U) \right)-\left(1-c*\left(dn^{(a+\epsilon)(1+\epsilon_1)} \right)^{2-d}*cap(V) \right)^{c'*cap(U)} \\
		\ge &1-\boldsymbol{s.e.}(T)-\exp{-c*\left(dn^{(a+\epsilon)(1+\epsilon_1)} \right)^{2-d}*cap(V)*cap(U)}\\
		\ge &1-\boldsymbol{s.e.}(T)-\exp{-c*n^{(a+\epsilon)(d-2)(1-3\epsilon_1)}}=1-\boldsymbol{s.e.}(T).
		\end{split}
	\end{equation}
\end{proof}
\begin{definition}\label{defofgoodsite}
		We say a site $y\in \mathbb{Z}^d$ is good if any line segment $l_{x,i}$ including $y$ is good and for any $l_{x_1,i_1}$ and $l_{x_2,i_2}$ including $y$, $\bar{\mathcal{I}}_{x_1,i_1}\xleftrightarrow[B_y(T^{\frac{1}{2d}})]{ S_y^{(2)}\cup S_y^{(3)}}\bar{\mathcal{I}}_{x_2,i_2}$.
\end{definition}
\begin{lemma}\label{goodsite}
 For any $y\in \mathbb{Z}^d$,
	\begin{equation}\label{goodsite0}
	P\left(y\ is\ good\right) \ge 1-\boldsymbol{s.e.}(T).
	\end{equation}
\end{lemma}
\begin{proof}
	By Definition \ref{defofgoodsite}, we have: for any $y\in \mathbb{Z}^d$, \begin{equation}\label{422}
	\begin{split}
&P\left(y\ is\ not\ good \right)\\
\le &\sum\limits_{(x,i):y\in l_{x,i}}P\left(\left\lbrace \bigcup_{j=1}^{d-2} U_{x,i,y}^{(j)}\subset B_y(dn^{(a+\epsilon)(1+\epsilon_1)}),cap(U_{x,i,y}^{(d-2)})\ge cn^{(a+\epsilon)(d-2)(1-\epsilon_1)}\right\rbrace^c,l_{x,i}\ is\ good\right)\\
&+\sum\limits_{(x,i):y\in l_{x,i}}P\left(l_{x,i}\ is\ not\ good \right)+\sum\limits_{(x_1,x_2,i_1,i_2):y\in l_{x_1,i_1}\cap l_{x_2,i_2}}P\Bigg(\bigcap\limits_{k=1,2}\bigg\{ l_{x_k,i_k}\ is\ good,\\
&\bigcup_{j=1}^{d-2} U_{x_k,i_k,y}^{(j)}\subset B_y(dn^{(a+\epsilon)(1+\epsilon_1)}),cap(U_{x_k,i_k,y}^{(d-2)})\ge cn^{(a+\epsilon)(d-2)(1-\epsilon_1)} \bigg\}\cap\left\lbrace \bar{\mathcal{I}}_{x_1,i_1}\xleftrightarrow[B_y(T^{\frac{1}{2d}})]{ S_y^{(2)}\cup S_y^{(3)}}\bar{\mathcal{I}}_{x_2,i_2} \right\rbrace^c  \Bigg).
	\end{split}
	\end{equation}
	
For any $l_{x,i}$, by Lemma \ref{goodline}, we have  
	\begin{equation}\label{423}
		P\left(l_{x,i}\ is\ not\ good \right)\le \boldsymbol{s.e.}(T).
	\end{equation}
	
For any $y\in l_{x,i}$, by (\ref{35}), we have
	\begin{equation}\label{424}
		P\left(\left\lbrace \bigcup_{j=1}^{d-2} U_{x,i,y}^{(j)}\subset B_y(dn^{(a+\epsilon)(1+\epsilon_1)}),cap(U_{x,i,y}^{(d-2)})\ge cn^{(a+\epsilon)(d-2)(1-\epsilon_1)}\right\rbrace^c,l_{x,i}\ is\ good\right)\le \boldsymbol{s.e.}(T).
	\end{equation}

Note that $\bar{\mathcal{J}}_{x,i}(y):=\bigcup_{i=1}^{d-2} U_{x,i,y}^{(i)}$ is connected to $\bar{\mathcal{I}}_{x,i}$. For any good line segments $l_{x_1,i_1}$ and $l_{x_2,i_2}$ including $y$, by Lemma \ref{connect} and the uniform bound trick,
\begin{equation}\label{425}
	\begin{split}
	P\Bigg(&\bigcap\limits_{k=1,2}\bigg\{ l_{x_k,i_k}\ is\ good,
	\bigcup_{j=1}^{d-2} U_{x_k,i_k,y}^{(j)}\subset B_y(dn^{(a+\epsilon)(1+\epsilon_1)}),cap(U_{x_k,i_k,y}^{(d-2)})\ge cn^{(a+\epsilon)(d-2)(1-\epsilon_1)} \bigg\}\\
	&\cap\left\lbrace \bar{\mathcal{I}}_{x_1,i_1}\xleftrightarrow[B_y(T^{\frac{1}{2d}})]{ S_y^{(2)}\cup S_y^{(3)}}\bar{\mathcal{I}}_{x_2,i_2} \right\rbrace^c  \Bigg)\\
	\le P\Bigg(&\bigcap\limits_{k=1,2}\bigg\{ l_{x_k,i_k}\ is\ good,
	\bigcup_{j=1}^{d-2} U_{x_k,i_k,y}^{(j)}\subset B_y(dn^{(a+\epsilon)(1+\epsilon_1)}),cap(U_{x_k,i_k,y}^{(d-2)})\ge cn^{(a+\epsilon)(d-2)(1-\epsilon_1)} \bigg\}\\
	&\cap\left\lbrace \bar{\mathcal{J}}_{x_1,i_1}(y) \xleftrightarrow[B_y(T^{\frac{1}{2d}})]{S_y^{(3)}(\bar{\mathcal{J}}_{x_1,i_1}(y)) }\bar{\mathcal{J}}_{x_2,i_2}(y) \right\rbrace^c  \Bigg)\le \boldsymbol{s.e.}(T).
	\end{split}
\end{equation}

Combine (\ref{422})-(\ref{425}), then (\ref{goodsite0}) holds.
\end{proof}
\subsection{Good boxes and infinite sub-cluster $\bar{\Gamma}$}
\begin{definition}\label{defgoodsite}
	Let $\mathcal{V}=n*\mathbb{Z}^d$ and define a site percolation $\left\lbrace Y(y) \right\rbrace_{y\in \mathcal{V}}$  : $Y(y)=1$ if and only if all sites $z\in B_y(n)$ are good. We also say $B_y(n)$ is a good box if $Y(y)=1$. 
\end{definition}
Here we need a Lemma about existence, uniquess and chemical distance of the infinite open cluster for a k-independence site percolation with sufficiently large vertex open probability. Moreover, it's easy to see that this Lemma can be applied in $\left\lbrace Y(y) \right\rbrace_{y\in \mathcal{V}}$.
\begin{remark}
	When we denote a collection containing one element, we may omit the braces. For example, we abbreviate $\{0\}$ to $0$.
\end{remark}
\begin{lemma}\label{lemmaZ}
	For a site percolation $\left\lbrace Z(y) \right\rbrace_{y\in \mathbb{Z}^d}$, $\forall y\in \mathbb{Z}^d$, $ P\left(Z(y)=1\right) \ge 1-\boldsymbol{s.e.}(T) $. If $\left\lbrace Z(y) \right\rbrace_{y\in \mathbb{Z}^d}$ is k-independent for some $k(u,d)>1$ (i.e. for any $A,B\subset \mathbb{Z}^d$ satisfying $d(A,B)\ge k$, $\left\lbrace Z(y) \right\rbrace_{y\in A} $ and $\left\lbrace Z(y) \right\rbrace_{y\in B} $ are independent), then $\exists T'(u,d)>0, c(d)>0$ such that for any $T>T'$, $\left\lbrace Z(y) \right\rbrace_{y\in \mathbb{Z}^d}$ has a unique infinite open cluster $\Gamma_Z$ and for any $y\in \mathbb{Z}^d$,
	\begin{equation}\label{3.24}
	P\left(0, y\in \Gamma_Z, \rho_{\Gamma_Z}(0,y)\ge  c|y|\right)\le \boldsymbol{s.e.}(|y|).
	\end{equation}
\end{lemma}
\begin{proof}
	 Consider an auxiliary bond percolation $\left\lbrace R(e)  \right\rbrace_{e\in \mathbb{L}^d}$, where $R(e)=1$ if and only if $Z(x)=Z(y)=1$ for $e=\left\lbrace x,y \right\rbrace $. Let $E_1=\left\lbrace \{e_1,e_2\}:e_1,e_2\in  \mathbb{L}^d, e_1\cap e_2\neq \emptyset \right\rbrace $ and define a graph $G_1=(\mathbb{L}^d, E_1)$. Since $\left\lbrace Z(y) \right\rbrace_{y\in \mathbb{Z}^d}$ is k-independent, $\left\lbrace R(e)\right\rbrace_{e\in \mathbb{L}^d}$ is also k-independent (i.e. $\forall D,E\subset \mathbb{L}^d$ and $\min\left\lbrace |x-z|:\exists e_1\in D,e_2\in E\ such\ that\ x\in e_1,z\in e_2 \right\rbrace\ge k$, $\left\lbrace R(e)  \right\rbrace_{e\in D}$ and $\left\lbrace R(e)  \right\rbrace_{e\in E}$ are independent). By Theorem 1.3 of \cite{liggett1997domination}, $\left\lbrace R(e) \right\rbrace_{e\in \mathbb{L}^d}$ stochastically dominates a Bernoulli bond percolation $\left\lbrace W(e) \right\rbrace_{e\in \mathbb{L}^d}$ with parameter $1-\boldsymbol{s.e.}(T)$. For $T$ large enough, $\left\lbrace W(e) \right\rbrace_{e\in \mathbb{L}^d}$ is supercritical, thus there a.s. exists a unique infinite open cluster $\Gamma_W$ in $\left\lbrace W(e) \right\rbrace_{e\in \mathbb{L}^d}$. By the stochastic domination, there also a.s. exists an infinite open cluster in $\left\lbrace R(e) \right\rbrace_{e\in \mathbb{L}^d}$.
	  
	 By Theorem 4.2 of \cite{grimmettpercolation}, if the parameter of $\left\lbrace W(e) \right\rbrace_{e\in \mathbb{L}^d}$ is greater than $p_{fin}$ ($0<p_{fin}<1$), then $\left( \Gamma_W\right)^c$ is composed of finite connected clusters. Assume $\Gamma_R^1$ and $\Gamma_R^2$ are both infinite open clusters in $\left\lbrace R(e) \right\rbrace_{e\in \mathbb{Z}^d}$. For any $y_1\in \Gamma_R^1$, if $y_1 \stackrel{R}\nleftrightarrow \Gamma_W$, then $\Gamma_R^1\subset \mathcal{C}_{y_1}^{\left( \Gamma_W\right) ^c}$, which is in conflict with the fact that $\Gamma_R^1$ is infinite but $\mathcal{C}_{y_1}^{\left( \Gamma_W\right) ^c}$ is finite. Thus $P\left(y_1 \xleftrightarrow{R} \Gamma_W \right) =1$. In the same way, for any $y_2\in \Gamma_R^2$, we have $P\left(y_2 \xleftrightarrow{R} \Gamma_W \right) =1$. Thus $\Gamma_R^1$ and $\Gamma_R^2$ are a.s. connected , which means $\left\lbrace R(e) \right\rbrace_{e\in \mathbb{L}^d}$ a.s. has a unique infinite open cluster $\Gamma_R$. By the definition of $\left\lbrace R(e)\right\rbrace_{e\in \mathbb{L}^d}$, $\left\lbrace Z(y) \right\rbrace_{y\in \mathbb{Z}^d}$ also a.s. has a unique infinite open cluster $\Gamma_{Z}$ and we have $\Gamma_{Z}=\Gamma_R$.

	Now we are going to prove a useful result: let $\hat{\mathcal{C}}_0=\mathcal{C}_0^{\left( \Gamma_W\right) ^c}$ and $D(0,\partial \hat{\mathcal{C}}_0)=\max\left\lbrace |z|:z\in \partial \hat{\mathcal{C}}_0 \right\rbrace $, then $\exists c,c'>0$ such that 
	\begin{equation}\label{useful1}
	P\left(D(0,\partial \hat{\mathcal{C}}_0)\ge N \right) \le c*e^{-c'N}.
	\end{equation}

	
	 Here we need to comfirm that $\partial \hat{\mathcal{C}}_0$ is connected under the graph $G_2=\left(\mathbb{Z}^d,\widetilde{\mathbb{L}}^d\right) $, where $\widetilde{\mathbb{L}}^d=\left\lbrace \{y_1,y_2\}:y_1,y_2\in \mathbb{Z}^d, |y_1-y_2|=1\right\rbrace $ (note that $y_1$ doesn't have to adjoin $y_2$ in $\mathbb{Z}^d$). Equivalently, we need to prove that for any $y_1,y_2\in \partial \hat{\mathcal{C}}_0$, there exists some points $\left(z_1,z_2,...,z_m \right) $ in $\partial \hat{\mathcal{C}}_0$ such that $|z_i-z_{i+1}|=1$ and $z_1=y_1,z_m=y_2$.
	
	By Lemma 3.1 of \cite{grimmettpercolation}, if $A\subset \mathbb{Z}^d$ is finite, connected and $\mathbb{Z}^d\setminus A$ is connected, then we have $\partial \left(\bigcup_{y\in A}\left\lbrace z\in \mathbb{R}^d: |z-y|\le 1 \right\rbrace \right)$ is a connected region in $\mathbb{R}^d$. Take $A=\hat{\mathcal{C}}_0$ and denote $U=\bigcup_{y\in \hat{\mathcal{C}}_0}\left\lbrace z\in \mathbb{R}^d: |z-y|\le 1 \right\rbrace $. Then $\partial U$ is connected. For any $y_1,y_2\in \partial \hat{\mathcal{C}}_0$, $B_{y_1}(1)\cap \partial U$ and $B_{y_2}(1)\cap \partial U$ are connected in $\partial U$ so that we can find a finite path $\gamma \subset \partial U$ from $B_{y_1}(1)\cap \partial U$ to $B_{y_2}(1)\cap \partial U$. Assume $\gamma$ crosses (d-1)-dimension squares $S_1,S_2,...,S_k\subset \partial U$ in turn and we denote the center of the box including $S_i$ by $z_i$. So we find the desired sequence of vertices $\left(z_1,z_2,...,z_m \right) $. Therefore, $\partial \hat{\mathcal{C}}_0$ is connected under the graph $G_2=\left(\mathbb{Z}^d,\widetilde{\mathbb{L}}^d\right) $.
	
	By definition, we know that for any $y\in \partial \hat{\mathcal{C}}_0$, there must exist an edge $e\in \mathbb{L}^d$ containing $y$ such that $W(e)=0$. In deed, since $y\in \partial \hat{\mathcal{C}}_0$, there exists $z\in \Gamma_W$ satisfying $|y-z|_2=1$ and thus $W(\{y,z\})=0$ (otherwise, $y\in \Gamma_W$).
	
	 We need another auxiliary percolation $\left\lbrace S(\widetilde{e})  \right\rbrace_{\widetilde{e}\in \widetilde{\mathbb{L}}^d}$: for $\widetilde{e}=\{x,y\}\in \widetilde{\mathbb{L}}^d$, $S(e)=1$ if and only if $\exists$ $e_1,e_2\in \mathbb{L}^d$ such that $x\in \widetilde{e}\cap e_1$, $y\in \widetilde{e}\cap e_2$ and $W(e_1)=W(e_2)=0$ (note that $e_1$ and $e_2$ can be the same edge). Since $\left\lbrace W(e)  \right\rbrace_{e\in \mathbb{L}^d}$ is a Bernoulli bond percolation with parameter $1-\boldsymbol{s.e.}(T)$, we know that $\left\lbrace S(e)  \right\rbrace_{e\in \widetilde{\mathbb{L}}^d}$ is 2-independent and $\forall e\in \mathbb{L}^d$, $P\left( S(e)=1 \right)\le \boldsymbol{s.e.}(T)$. In addition, since $\partial \hat{\mathcal{C}}_0$ is connected under $G_2$ and $y\in \partial \hat{\mathcal{C}}_0$, $\exists e\in \mathbb{L}^d$ such that $y\in e$ and $W(e)=0$, we have $\partial\hat{\mathcal{C}}_0\subset \widetilde{\mathcal{C}}_y^{S}$ for any $y\in \partial\hat{\mathcal{C}}_0$, where $\widetilde{\mathcal{C}}$ is defined in the same way as $\mathcal{C}$ by using $\widetilde{\mathbb{L}}^d$ instead of $\mathbb{L}^d$.
	
	 By Theorem 1.3 of \cite{liggett1997domination}, $\left\lbrace 1-S(e) \right\rbrace_{e\in \mathbb{Z}^d}$ stochastically dominates a supercritical Bernoulli bond percolation on $G_2$ with parameter $1-\boldsymbol{s.e.}(T)$. Thus $\left\lbrace S(e) \right\rbrace_{e\in \mathbb{Z}^d}$ is dominated by a subcritical bond percolation $\left\lbrace V(e) \right\rbrace_{e\in \mathbb{Z}^d}$, when $T$ is large enough. 
	
	 Using Theorem 6.1 of \cite{grimmett2013percolation}, since $\partial\hat{\mathcal{C}}_0\subset \widetilde{\mathcal{C}}_y^{S}$ for any $y\in \partial\hat{\mathcal{C}}_0$ and $\left\lbrace V(e) \right\rbrace_{e\in \mathbb{Z}^d}$ stochastically dominates $\left\lbrace S(e) \right\rbrace_{e\in \mathbb{Z}^d}$, we have 
	\begin{equation}\label{3.57}
	\begin{split}
	P\left( D(0,\partial\hat{\mathcal{C}}_0)\ge N \right) 
	\le &\sum_{y,|y|\ge N}P\left(y\in \partial\hat{\mathcal{C}}_0 \right) \\
	\le &\sum_{y,|y|\ge N}P\left(|\widetilde{\mathcal{C}}_y^{S}|\ge |y| \right) \\
	\le &\sum_{y,|y|\ge N}P\left(|\widetilde{\mathcal{C}}_y^{V}|\ge |y| \right) \\
	\le &\sum_{y,|y|\ge N}e^{-c'|y|}\le c*e^{-c'N}.		
	\end{split}
	\end{equation}
	Thus (\ref{useful1}) is proved.
	
	If $\left\lbrace0\in \Gamma_R, \rho_{\Gamma_R}(0,\Gamma_W)\ge N  \right\rbrace $ happens, then $D(0,\partial \hat{\mathcal{C}}_0)\ge 0.4N^{\frac{1}{d}}$ (otherwise, $ \rho_{\Gamma_R}(0,\Gamma_W)\le |\hat{\mathcal{C}}_0|\le |B(0.4N^{\frac{1}{d}})|<N$ ). By (\ref{useful1}), \begin{equation}\label{54}
	P\left( 0\in \Gamma_R, \rho_{\Gamma_R}(0,\Gamma_W)\ge N \right) \le P\left( D(0,\partial \hat{\mathcal{C}}_0)\ge 0.4N^{\frac{1}{d}}\right)\le \boldsymbol{s.e.}(N).
	\end{equation}
	
	By Theorem1.1 of \cite{vcerny2012internal}, there exists $c'(d)>0$ such that for any $z_1,z_2\in \mathbb{Z}^d$,
	\begin{equation}\label{55}
	P\left( z_1,z_2\in \Gamma_W, \rho_{\Gamma_W}(z_1,z_2)\ge c'*|z_1-z_2|\right) \le \boldsymbol{s.e.}(|z_1-z_2|).
	\end{equation}
	Note that $\forall z_1\in B_0(\frac{1}{3}|y|),z_2\in B_y(\frac{1}{3}|y|)$, $2|y|\ge |z_1-z_2|\ge \frac{1}{3}|y|$. By (\ref{54}) and (\ref{55}), take $c=2c'+1$ and then we have
	\begin{equation}\label{3.29}
	\begin{split}
	&P\left(0, y\in \Gamma_R, \rho_{\Gamma_R}(0,y)\ge c|y| \right)\\
	\le &P\left(0, y\in \Gamma_R, \rho_{\Gamma_R}(0,y)\ge c|y|,\rho_{\Gamma_R}(0,\Gamma_W)\le \frac{1}{3}|y|,\rho_{\Gamma_R}(y,\Gamma_W)\le \frac{1}{3}|y|\right)+\boldsymbol{s.e.}(|y|)\\
	\le &\sum_{z_1\in B_0(\frac{1}{3}|y|),z_2\in B_y(\frac{1}{3}|y|)}P\left( \rho_{\Gamma_W}(z_1,z_2)\ge c'|z_1-z_2| \right)+\boldsymbol{s.e.}(|y|)\\
	\le &c|y|^{2d}*\boldsymbol{s.e.}(|y|)+\boldsymbol{s.e.}(|y|)=\boldsymbol{s.e.}(|y|).
	\end{split}
	\end{equation} 
	Then by (\ref{3.29}) and $\Gamma_Z=\Gamma_R$, we get (\ref{3.24}).
\end{proof}
\begin{proposition}\label{Y}
	For any $d\ge 3$, $u>0$, $\exists \bar{T}(d,u)>0$ such that $\forall T>\bar{T}$, site percolation $\left\lbrace Y(y) \right\rbrace_{y\in \mathcal{V}}$ satisfies the following properties:
	\begin{enumerate}
		\item For any $y\in \mathcal{V}$, $P(Y(y)=1)\ge 1-\boldsymbol{s.e.}(T)$.
		\item $Y(y)=0\ or\ 1$ only depends on the paths starting from $B_y(9T^{0.5})$.
		\item For any subsets $A,B$ in $ \mathcal{V}$ and $d(A,B) \ge 20T^{0.5}$, then $\left\lbrace Y(y)\right\rbrace_{y\in A} $ and $\left\lbrace Y(y)\right\rbrace_{y\in B}$ are independent.
		\item If $Y(y)=1$, let $\mathbb{B}_y= \bigcup\limits_{l_{x,i}\cap B_y(n)\neq \emptyset} \bar{\mathcal{I}}_{x,i}$ and $\bar{\mathbb{B}}_y =\mathbb{B}_y \cup \left(\bigcup\limits_{z\in B_y(n)}\mathcal{C}_{\mathbb{B}_y}^{S_z^{(2)}\cup S_z^{(3)}}\left(B_z(T^{\frac{1}{2d}}) \right)    \right) $. Then $\bar{\mathbb{B}}_y$ is connected and $\exists c=c(d,u)>0$ such that for any two sites $z_1,z_2\in \bar{\mathbb{B}}_y$, $\rho_{\bar{\mathbb{B}}_y}(z_1,z_2)<cT^{0.5}$.
		\item For $y_1, y_2\in \mathcal{V}$, if  $|y_1-y_2|_2=n$ and $Y(y_1)=Y(y_2)=1$, then $\exists c=c(d,u)>0$ such that for any $z_1\in \bar{\mathbb{B}}_{y_1}$ and $z_2\in \bar{\mathbb{B}}_{y_2}$, $\rho_{\bar{\mathbb{B}}_{y_1}\cup \bar{\mathbb{B}}_{y_2}}(z_1,z_2)<cT^{0.5}$.
		\item  $\left\lbrace Y(y) \right\rbrace_{y\in \mathcal{V}}$ has a unique infinite open cluster $\Gamma_Y$. 
		\item There exists $c(d)>0$ such that for $y\in \mathcal{V}$, 
		\begin{equation}\label{42}
		P\left(0, y\in \Gamma_Y, \rho_{\Gamma_Y}(0,y)\ge c|y|_{\mathcal{V}}\right)\le \boldsymbol{s.e.}(|y|), 
		\end{equation}
		where $|y|_{\mathcal{V}}:=\frac{|y|}{n}$ for $y\in \mathcal{V}$.
	\end{enumerate}
\end{proposition}
\begin{proof}
	(1)-(3) are elementary so we just omit their proof.
	
	(4): If $z_1\in \bar{\mathcal{I}}_{x_1,i_1}$, without loss of generality, assume $i_1=1$. Let $w_i=\left(y^{(1)},...,y^{(i)},x_1^{(i+1)},...,x_1^{(d)} \right)$ for $i=0,1,2,...,d$ (recall that $z^{(i)}$ is the i-th coordinate of $z$). Especially, $w_0=x_1$ and $w_d=y$. Since $l_{x_1,1}\cap B_y(n)\neq \emptyset$, there exists $v_1=x_1+j*e_1=w_1+(j+x_1^{(1)}-y^{(1)})*e_1\in l_{x_1,1}\cap B_y(n)$, where $|j+x_1^{(1)}-y^{(1)}|\le n$ (otherwise, $|v_1-y|\ge v_1^{(1)}-y^{(1)}>n$ and thus, $v_1\notin B_y(n)$). Then we have $v_1\in l_{w_1,1}$. Since $v_1\in B_y(n)$ and $Y(y)=1$, $v_1$ is a good site. By the definition of good sites (recall Definition \ref{defgoodsite}) and $v_1\in l_{x_1,1}\cap l_{w_1,1}$, we have $\bar{\mathcal{I}}_{x_1,1}\xleftrightarrow[B_{v_1}(T^{\frac{1}{2d}})]{ S_{v_1}^{(2)}\cup S_{v_1}^{(3)}}\bar{\mathcal{I}}_{w_1,1}$. Thus $\exists s_o^{+}\in \bar{\mathcal{I}}_{x_1,1}$ and $s_1^{-}\in \bar{\mathcal{I}}_{w_1,1}$ such that $\rho_{\bar{\mathbb{B}}_y}(s_0^{+},s_1^{-})\le (2T^{\frac{1}{2d}}+1)^d\le c_1T^{0.5}$, where $c_1=c_1(d)$. 
	 Similarly, for any $1\le j\le d-1$, $\exists s_j^{+}\in \bar{\mathcal{I}}_{w_j,j}$ and $s_{j+1}^{-}\in \bar{\mathcal{I}}_{w_{j+1},j+1}$ such that $\rho_{\bar{\mathbb{B}}_y}(s_j^{+},s_{j+1}^{-})\le c_1T^{0.5}$. By subadditivity of chemical distance and (3) of Lemma \ref{goodline}, we have \begin{equation}\label{43}
	\begin{split}
	\rho_{\bar{\mathbb{B}}_y}\left(z_1, \bar{\mathcal{I}}_{y,d}\right)\le &
\rho_{\bar{\mathbb{B}}_y}\left(z_1, s_0^{+}\right)+\rho_{\bar{\mathbb{B}}_y}\left(s_0^{+}, s_1^{-}\right)
+\sum\limits_{j=1}^{d-1}\left( \rho_{\bar{\mathbb{B}}_y}\left(s_j^{-},s_j^{+}\right)+\rho_{\bar{\mathbb{B}}_y}\left(s_j^{+},s_{j+1}^{-}\right)\right)\\ \le &(c_1+c_2)d*T^{0.5}, 
	\end{split}
	\end{equation}
	where $c_2$ is the constant in (3) of Definition \ref{goodlineseg}.
	
	Now for any $z_1\in \mathcal{C}_{\mathbb{B}_y}^{S_z^{(2)}\cup S_z^{(3)}}\left(B_z(T^{\frac{1}{2d}}) \right)$, where $z\in B_y(n)$, we have 
	 \begin{equation}\label{44}
		\rho_{\bar{\mathbb{B}}_y}\left(z_1, \mathbb{B}_y \right)\le c_1T^{0.5}.
	\end{equation}
	Combine (\ref{43}) and (\ref{44}), \begin{equation}\label{45}
	\begin{split}
	\rho_{\bar{\mathbb{B}}_y}\left(z_1, \bar{\mathcal{I}}_{y,d} \right)\le \left(c_1(d+1)+c_2d\right) T^{0.5}.
	\end{split}
	\end{equation}
	In the same way, we have \begin{equation}\label{46}
		\rho_{\bar{\mathbb{B}}_y}\left(z_2, \bar{\mathcal{I}}_{y,d} \right)\le \left(c_1(d+1)+c_2d \right) T^{0.5}.
	\end{equation}
	Combine (\ref{45}), (\ref{46}) and (3) of Definition \ref{goodlineseg}, \begin{equation}
		\rho_{\bar{\mathbb{B}}_y}\left(z_1, z_2 \right)\le \left(2c_1(d+1)+c_2(2d+1) \right) T^{0.5}.
	\end{equation}
	
	(5): Without loss of generality, assume $y_1-y_2=n*e_d$. By (\ref{45}), for $j=1,2$, \begin{equation}\label{48}
		\rho_{\bar{\mathbb{B}}_{y_j}}\left(z_j, \bar{\mathcal{I}}_{y_j,d} \right)\le \left(c_1(d+1)+c_2d \right) T^{0.5}.
	\end{equation}
	Thus for $j=1,2$, there exists $r_j\in \bar{\mathcal{I}}_{y_j,d}$ such that \begin{equation}\label{3.30}
		\rho_{\bar{\mathbb{B}}_{y_j}}\left(z_j,r_j\right)\le  \left(c_1(d+1)+c_2d \right) T^{0.5}.
	\end{equation} 
	Since $y_1\in l_{y_1,d}\cap l_{y_2,d}$ and $y_1$ is a good site, we have  $\bar{\mathcal{I}}_{y_1,d}\xleftrightarrow[B_{y_1}(T^{\frac{1}{2d}})]{ S_{y_1}^{(2)}\cup S_{y_1}^{(3)}}\bar{\mathcal{I}}_{y_2,d}$. Thus, there exists $u_1\in \bar{\mathcal{I}}_{y_1,d}$ and $u_2\in \bar{\mathcal{I}}_{y_2,d}$ such that \begin{equation}\label{3.30.5}
		\rho_{\bar{\mathbb{B}}_{y_1}}\left(u_1,u_2 \right)\le c_1T^{0.5}.
	\end{equation}
	Combine (\ref{3.30}), (\ref{3.30.5}) and (3) of Definition \ref{goodlineseg}, \begin{equation}
	\begin{split}
	\rho_{\bar{\mathbb{B}}_{y_1}\cup \bar{\mathbb{B}}_{y_2}}\left(z_1,z_2 \right)\le& \rho_{\bar{\mathbb{B}}_{y_1}}\left(z_1,r_1 \right)+\rho_{\bar{\mathbb{B}}_{y_1}}\left(r_1,u_1 \right)+\rho_{\bar{\mathbb{B}}_{y_1}}\left(u_1,u_2 \right)+\rho_{\bar{\mathbb{B}}_{y_2}}\left(u_2,r_2 \right)+\rho_{\bar{\mathbb{B}}_{y_2}}\left(r_2,z_2 \right)\\
	\le &\left(c_1(2d+3)+2c_2(d+1) \right)T^{0.5} .
	\end{split}	
	\end{equation}
	
	(6)-(7): By (1)-(3), $\left\lbrace Y(y) \right\rbrace_{y\in \mathcal{V}} $ satisfies all conditions of Lemma \ref{lemmaZ}. Meanwhile, since $\mathbb{Z}^d$ and $\mathcal{V}$ are isomorphic, we can apply Lemma \ref{lemmaZ} on $\left\lbrace Y(y) \right\rbrace_{y\in \mathcal{V}} $ to get (6) and (7).
\end{proof}
\begin{proposition}\label{bargamma}
	Define $\bar{\Gamma}=\bigcup\limits_{y\in \Gamma_Y}\bar{\mathbb{B}}_y$. Then for large enough $T$, $\bar{\Gamma}$ is connected and $\exists c(u,d)>0$ such that 
	\begin{equation}
	P\left(0, z\in \bar{\Gamma}, \rho_{\bar{\Gamma}}(0,z)\ge c|z|\right)\le \boldsymbol{s.e.}(|z|).
	\end{equation}
\end{proposition}
\begin{proof}
	It's sufficient to prove the case when $|z|\ge T$. If $0,z\in \bar{\Gamma}$, then $\exists y_1\in B(4n)$ and $y_2\in B_{z}(4n)$ such that $0\in \bar{\mathbb{B}}_{y_1}$ and $z\in \bar{\mathbb{B}}_{y_2}$. By (4) and (5) of Proposition \ref{Y}, $\exists c_1(u,d)>0$ such that  \begin{equation}\label{58}
		\rho_{\bar{\Gamma}}\left(0,z\right)\le c_1n*\rho_{\Gamma_Y}\left(y_1,y_2\right).
	\end{equation}
	Note that for any $y_1\in B_0(4n)\cap \mathcal{V},y_2\in B_z(4n)\cap \mathcal{V}$, $10|z|\ge |y_1-y_2|\ge \frac{1}{10}|z|$. Combine (\ref{42}), (\ref{58}) and let $c_2$ be the constant $c$ in (\ref{42}), 
	\begin{equation}
		\begin{split}
		&P\left(0, z\in \bar{\Gamma}, \rho_{\bar{\Gamma}}(0,z)\ge 10c_1c_2|z| \right)\\
		\le &\sum_{y_1\in B_0(4n)\cap \mathcal{V},y_2\in B_z(4n)\cap \mathcal{V}}P\left(y_1, y_2\in \Gamma_Y, \rho_{\Gamma_Y}\left(y_1,y_2\right)\ge 10c_2n^{-1}|z| \right)\\
		\le &\sum_{y_1\in B_0(4n)\cap \mathcal{V},y_2\in B_z(4n)\cap \mathcal{V}}P\left(y_1, y_2\in \Gamma_Y, \rho_{\Gamma_Y}\left(y_1,y_2\right)\ge c_2 |y_1-y_2|_{\mathcal{V}}\right)\\
		\le &\sum_{y_1\in B_0(4n)\cap \mathcal{V},y_2\in B_z(4n)\cap \mathcal{V}}\boldsymbol{s.e.}(|y_1-y_2|)
		\le \boldsymbol{s.e.}(|z|).
		\end{split}
	\end{equation}
\end{proof}
\subsection{Proof of Theorem \ref{theorem3}}
Once we have Proposition \ref{Y}, we are able to prove Theorem \ref{theorem3}.

Let $\mathcal{H}=\left\lbrace \{v_1,v_2\}:v_1,v_2\in \mathcal{V}, |v_1-v_2|_2=n \right\rbrace $. Similar to Lemma \ref{lemmaZ}, define an auxiliary bond percolation $\left\lbrace S(e) \right\rbrace_{e\in \mathcal{H}}$: for $e=\left\lbrace y_1,y_2 \right\rbrace $, $S(e)=1$ if and only if $Y(y_1)=Y(y_2)=1$. Obviousely, $\left\lbrace S(e) \right\rbrace_{e\in \mathcal{H}}$ is k-independent for some $k(u,d)>0$. By Theorem 1.3 of \cite{liggett1997domination}, $\left\lbrace S(e) \right\rbrace_{e\in \mathcal{H}}$ stochasitcally dominates a supercritical Bernoulli bond percolation $\left\lbrace W(e) \right\rbrace_{e\in \mathcal{H}}$ for large enough $T$.

It's sufficient to prove the case when $N>T$. Let $\widetilde{N}=N-0.5n$. Consider the box $B^{\mathcal{V}}(\widetilde{N}):=\left\lbrace y\in \mathcal{V}:|y|\le \widetilde{N} \right\rbrace $ and define that $LR_{\mathcal{V}}(m)=\left\lbrace left\ of\ B^{\mathcal{V}}(m)\xleftrightarrow[B^{\mathcal{V}}(\widetilde{N})]{W} right\ of\ B^{\mathcal{V}}(m) \right\rbrace $. By (8.98) of \cite{grimmett2013percolation}, for  $\left\lbrace W(y) \right\rbrace_{e\in\mathcal{H}} $, $\exists c(d,u,T)>0$ such that \begin{equation}
	P\left(LR_{\mathcal{V}}(N) \right) \ge 1-e^{-cN^{d-1}}.
\end{equation}

If $LR_{\mathcal{V}}(\widetilde{N})$ happens, by stochastic domination, there exists an open path in $\left\lbrace Y(y) \right\rbrace_{y\in \mathcal{V}}$ consisting of centers of good boxes crossing $B^{\mathcal{V}}(\widetilde{N})$. Assume $(z_1,...,z_m)\subset B^{\mathcal{V}}(\widetilde{N})$ is an open nearest-neighbor path in $\left\lbrace Y(y) \right\rbrace_{y\in \mathcal{V}}$ satisfying $z_1\in left\ of\ B^{\mathcal{V}}(\widetilde{N})$ and $z_m\in right\ of\ B^{\mathcal{V}}(\widetilde{N})$. For $1\le j\le m-1$, let $z_{j+1}-z_j=\delta_{j}*ne_{i_j}$, where $\delta_{j}\in \{-1,1\}$ and $1\le i_j\le d$. Particularly, we set $i_m=1$. Since the line segment $l_{z_j,i_j}$ and the site $z_{j+1}$ are both good, we have  $\bar{\varphi}_1^{z_j,i_j}(z_j)\xleftrightarrow{\bar{\mathcal{I}}_{z_j,i_j}}\bar{\varphi}_1^{z_{j},i_j}(z_{j+1})$  and $\bar{\varphi}_1^{z_{j},i_j}(z_{j+1})\xleftrightarrow[B_{z_{j+1}}(T^{\frac{1}{2d}})]{}\bar{\varphi}_1^{z_{j+1},i_{j+1}}(z_{j+1})$. Recalling the definition of $\bar{\mathcal{I}}_{x,i}$, we have $\forall x\in \mathbb{Z}^d$, $1\le i\le d$, $\bar{\mathcal{I}}_{x,i}\subset \bigcup\limits_{y\in l_{x,i}}B_y(4n^{2(a+\epsilon)})$. Thus for $1\le j\le m-1$, since $\bar{\mathcal{I}}_{z_j,i_j}\subset \bigcup\limits_{y\in l_{z_j,i_j}}B_y(4n^{2(a+\epsilon)})\subset B(N)$ and $B_{z_{j+1}}(T^{\frac{1}{2d}})\subset B(N)$, we have  $\bar{\varphi}_1^{z_j,i_j}(z_j)\xleftrightarrow[B(N)]{}\bar{\varphi}_1^{z_{j+1},i_{j+1}}(z_{j+1})$. Meanwhile, since $l_{z_1-ne_1,1}$ is good and $\bar{\mathcal{I}}_{z_1-ne_1,1}\cap \partial B(N)\subset left\ of\ B(N)$, we have $\bar{\varphi}_1^{z_1-ne_1,1}(z_1-2ne_1)\in B_{z_1-2ne_1}(T^\epsilon)\subset (B(N))^c$ and thus $\bar{\varphi}_1^{z_1-ne_1,1}(z_1)\xleftrightarrow[B(N)]{}left\ of\ B(N)$. Note that $\bar{\varphi}_1^{z_1-ne_1,1}(z_1)\xleftrightarrow[B_{z_1}(T^{\frac{1}{2d}})]{}\bar{\varphi}_1^{z_1,i_1}(z_1)$, we have $\bar{\varphi}_1^{z_1,i_1}(z_1)\xleftrightarrow[B(N)]{}left\ of\ B(N)$. Similarly, we have $\bar{\varphi}_1^{z_m,i_m}(z_m)\xleftrightarrow[B(N)]{}right\ of\ B(N)$. In conclusion, \begin{equation}
	left\ of\ B(N)\xleftrightarrow[B(N)]{}\bar{\varphi}_1^{z_1,i_1}(z_1)\xleftrightarrow[B(N)]{}...\xleftrightarrow[B(N)]{}\bar{\varphi}_1^{z_m,i_m}(z_m)\xleftrightarrow[B(N)]{}right\ of\ B(N).
\end{equation}
Therefore, $LR_{\mathcal{V}}(\widetilde{N})$ implies $\left\lbrace left\ of\ B(N)\xleftrightarrow[B(N)]{}right\ of\ B(N) \right\rbrace $. Thus for FRI, 
\begin{equation}
	P\left(LR(N) \right)\ge P\left(LR_{\mathcal{V}}(\widetilde{N}) \right)\ge 1-e^{-cN^{d-1}}.
\end{equation} \qed
\section{Connecting to the good sub-cluster}
Assume $\mathcal{FI}^{u,T}=\mathcal{FI}_{1}^{0.5u,T}\cup \mathcal{FI}_{2}^{0.5u,T}$ and $\mathcal{FI}_{1}^{0.5u,T}=\mathcal{FI}_{1,1}^{0.25u,T}\cup \mathcal{FI}_{1,2}^{0.25u,T}$, $\mathcal{FI}_{2}^{0.5u,T}=\mathcal{FI}_{2,1}^{0.25u,T}\cup \mathcal{FI}_{2,2}^{0.25u,T}$, where $\mathcal{FI}_{1,1}^{0.25u,T}$, $\mathcal{FI}_{1,2}^{0.25u,T}$, $\mathcal{FI}_{2,1}^{0.25u,T}$ and $\mathcal{FI}_{2,2}^{0.25u,T}$ are independent. Define two independent site percolations $\left\lbrace Y_i(y) \right\rbrace_{y\in \mathcal{V}}$, $i=1,2$: $Y_i(y)=1$ if and only if the box $B_y(n)$ is good in $\mathcal{FI}_{i,1}^{0.25u,T}$. Let $\left\lbrace Y_{12}(y) \right\rbrace_{y\in \mathcal{V}}:=\left\lbrace Y_1(y) Y_2(y) \right\rbrace_{y\in \mathcal{V}}$. Let $\Gamma_Y^{(i)}$ be the unique infinite open cluster in $\left\lbrace  Y_i(y) \right\rbrace_{y\in \mathcal{V}}$ and $\Gamma_Y^{(12)}$ be the unique infinite open cluster in $\left\lbrace  Y_{12}(y) \right\rbrace_{y\in \mathcal{V}}$.

For any $y\in \mathcal{V}$, we define $\bar{\mathbb{B}}^{i}_{y}$ in the same way as $\bar{\mathbb{B}}_{y}$, using $\mathcal{FI}_{i,1}^{0.25u,T}$ instead of $\mathcal{FI}^{u,T}$, for $i=1,2$ (See (4) of Proposition \ref{Y}). We also define that $\bar{\Gamma}_{i}=\bigcup\limits_{y\in \Gamma_Y^{(i)}}\bar{\mathbb{B}}^{i}_{y}$ and $\bar{\Gamma}_{12}=\bigcup\limits_{y\in \Gamma_Y^{(12)}}\left( \bar{\mathbb{B}}^{1}_{y}\cup \bar{\mathbb{B}}^{2}_{y}\right) $.

In this section, we denote that $S_x^{i}=S_x\left(\mathcal{FI}_{i,2} \right)$, for $i=1,2$ (by the definition of $S_x(\cdot)$ in Section \ref{notations}, $S_x^{i}$ is the collection of all paths starting from $B_x(8T^{0.5})\setminus B_x(6T^{0.5})$ in $\mathcal{FI}_{i,2}$). Let $\bar{n}=[\frac{1}{2}bT^{0.5}]$, which is a half of the block size. Unless otherwise specified, all the arguements in this section are based on the condition ``$T$ is large enough''.
\begin{proposition}\label{prop12}For integer $N>0$,
	\begin{equation}
	P\left[0\xleftrightarrow{} \partial B(N^{\frac{1}{2d}}), \rho(0,\bar{\Gamma}_{12})>N    \right]\le \boldsymbol{s.e.}(N).
	\end{equation}
\end{proposition}
An important tool in the proof of Proposition \ref{prop12} is the decomposition constructed in \cite{rodriguez2013phase}. Based on this decomposition, we can decompose the event $\left\lbrace 0\xleftrightarrow{} \partial B(N^{\frac{1}{2d}}),\rho(0,\bar{\Gamma}_{12})>N\right\rbrace $ into $\left(polynomial(T) \right)^{N^c} $ sub-events, each of which will be proved to occur with probability at most $ \left(\boldsymbol{s.e.}(T)\right)^{N^c} $. As a result, $\left(polynomial(T)*\boldsymbol{s.e.}(T) \right)^{N^c} $ is an upper bound for the probability of $\left\lbrace 0\xleftrightarrow{} \partial B(N^{\frac{1}{2d}}),\rho(0,\bar{\Gamma}_{12})>N\right\rbrace $. Then Proposition \ref{prop12} holds for large enough $T$. 
\begin{notation}We need some notations before starting our proof:
	\begin{itemize}
		\item Fix $l_0=\lfloor T^{2}\rfloor >0$ and integer $L_0>100$; $L_m=L_0*\left( l_0\right) ^m$;
		\item For $m\ge 1$, let $\widetilde{L}_m=2L_m$; $\widetilde{L}_0=\bar{n}$;
		\item $\mathbb{L}_m=L_m*\mathbb{Z}^d$; $\mathbb{I}_m=\left\lbrace m\right\rbrace\times \mathbb{L}_m $;
		\item For any $ m\ge 0$ and $x\in \mathbb{L}_m$, let $B_{m,x}=B_x(L_m)$ and $\widetilde{B}_{m,x}=B_x(\widetilde{L}_m)$;
		\item For any $ m\ge 1$ and $x\in \mathbb{L}_m$,
		 $$\mathcal{H}_1(m,x)=\left\lbrace (m-1,y)\in \mathbb{I}_{m-1}: B_{m-1,y}\subset B_{m,x}\ and \ B_{m-1,y}\cap \partial B_{m,x}\neq \emptyset \right\rbrace; $$ $$\mathcal{H}_2(m,x)=\left\lbrace (m-1,y)\in \mathbb{I}_{m-1}: B_{m-1,y}\cap \left\lbrace z\in \mathbb{Z}^d: d(z,B_{m,x})=\lfloor \frac{L_m}{2}\rfloor \right\rbrace \neq \emptyset  \right\rbrace; $$
		\item $n_0=\max\left\lbrace m\in \mathbb{N}^+: 2L_{m}\le N^{\frac{1}{2d}}\right\rbrace $, note that $n_0=O(log(N))$ and $\exists c>0$ such that $2^{n_0}>N^c$.
		\item For $m\ge 1,x\in \mathbb{L}_m$ such that $\widetilde{B}_{m,x}\subset B(N^{\frac{1}{2d}})$, 
		 \begin{equation}\label{decom}
		\begin{split}
		\Lambda_{m,x}= &\biggl\{ \mathcal{T}\subset \bigcup_{k=0}^{m}\mathbb{I}_k: \mathcal{T}\cap\mathbb{I}_m=(m,x)\ and\ every\ (k,y)\in \mathcal{T}\cap \mathbb{I}_k, 0< k\le m, \\
		&has\ two\ descendants\ (k-1, y_i(k,y))\in \mathcal{H}_i(k,y), i=1,2,\\
		&such\ that\ \mathcal{T}\cap \mathbb{I}_{k-1}=\bigcup_{(k,y)\in \mathcal{T}\cap \mathbb{I}_k}\{ (k-1, y_1(k,y)),(k-1, y_2(k,y)) \} \biggr\} .
		\end{split}
		\end{equation}
		By (2.8) of \cite{rodriguez2013phase}, one has $|\Lambda_{m,x}|\le \left(c_0l_0^{2(d-1)} \right)^{2^m}$, where $c_0(d)$ is a constant.
		\item For $m\ge 0,x\in \mathbb{Z}^d$ such that $\widetilde{B}_{m,x}\subset B(N^{\frac{1}{2d}})$, let $A_{m,x}=\left\lbrace B_{m,x} \xleftrightarrow{\mathcal{C}_0(\widetilde{B}_{n_0,0})} \partial \widetilde{B}_{m,x} \right\rbrace $. Similar to (2.14) of \cite{rodriguez2013phase}, for $m\ge 1$,  \begin{equation}
		A_{m,x} \subset \bigcup_{\mathcal{T}\in \Lambda_{m,x}}A_{\mathcal{T}}, 
		\end{equation}
		where $A_{\mathcal{T}}=\bigcap\limits_{(0,y)\in \mathcal{T}\cap \mathbb{I}_0}A_{0,y}$.
		\item For small enough $\xi >0$, we define a cut-off mapping $\pi_\xi$: for any point measure $\omega=\sum_{i=1}^{\infty}\delta_{\eta_i}\ ( \forall \eta_i\in W^{\left[0,\infty \right) }) $, let \begin{equation}
		\pi_\xi(\omega)=\sum_{i=1}^{\infty}\delta_{\eta_i}\cdot \mathbbm{1}_{\left\lbrace length(\eta_i)< T^\xi\right\rbrace }.
		\end{equation}
		\item Define an important event: $E_x^\xi=\left\lbrace \omega: B_x(L_0)\stackrel{\pi_\xi(\omega)}{\nleftrightarrow}\partial B_x(\bar{n})\right\rbrace$. Note that $E_x^\xi$ only depends on the paths starting from $B_x(\bar{n}+T^\xi)$. Furthermore, if $E_x^\xi$ and $\left\lbrace B_x(L_0)\xleftrightarrow{}\partial B_x(\bar{n})\right\rbrace$ both occur, then for any cluster hitting $B_x(L_0)$ and $\partial B_x(\bar{n})$, there must exist a path with length $\ge T^\xi$ in it.
	\end{itemize}
\end{notation}
\begin{lemma}\label{lemma10}For any $x\in \mathbb{Z}^d$ and $\xi\le \frac{1}{d+1}$, we have
	$$P\left( E_x^\xi\right) \ge 1-\boldsymbol{s.e.}(T).$$
\end{lemma}
\begin{proof}
	Let $n_{\xi}=\lfloor T^\xi\rfloor +1$, $\mathcal{V}_x^\xi:=x+n_\xi*\mathbb{Z}^d$ and $\mathcal{H}_x^\xi=\left\lbrace \{y_1,y_2\}:y_1,y_2\in \mathcal{V}_x^\xi,|y_1- y_2|_2=n_{\xi} \right\rbrace $. Consider a bond percolation $\left\lbrace Z(e)\right\rbrace_{e\in \mathcal{H}_x^\xi}$: for $e=\{y_1,y_2\}$, $Z(e)=1$ if and only if there exists a path hitting $B_{y_j}(n_\xi)$ in $\pi_\xi(\omega)$, $j=1,2$.
	For any $e\in \mathcal{H}_x^\xi$, we have \begin{equation}
	\begin{split}
		P(Z(e)=1)
	\le &P\left( \bigcap\limits_{j=1,2}\left\lbrace \exists a\ path\ starting\ from\ B_{y_j}(2n_\xi)\ in\ \pi_\xi(\omega)\right\rbrace \right) \\
	\le &1-e^{-\frac{2du}{T+1}*(1-(\frac{T}{T+1})^{T^\xi})*(5n_\xi+1)^d}\le cT^{(d+1)\xi-2}\le cT^{-1}.
	\end{split}
	\end{equation}
	For any $e=\{y_1,y_2\}\in \mathcal{H}_x^\xi$, $Z(e)$ depends only on paths starting from $B_{y_1}(2T^\xi)\cup B_{y_2}(2T^\xi)$. Thus $\left\lbrace Z(e)\right\rbrace_{e\in \mathcal{H}_x^\xi}$ is 6-independent. Using Theorem 1.3 of \cite{liggett1997domination}, we know that $\left\lbrace 1-Z(e) \right\rbrace_{e\in \mathcal{H}_x^\xi} $ stochastically dominates a Bernoulli bond percolation with parameter $1-\boldsymbol{s.e.}(T)$. Thus $\left\lbrace Z(e) \right\rbrace_{e\in \mathcal{H}_x^\xi} $ is stochastically dominated by a subcritical Bernoulli bond percolation $\left\lbrace W(e) \right\rbrace_{e\in \mathcal{H}_x^\xi} $ when $T$ is large enough. By Theorem 6.1 of \cite{grimmett2013percolation}, 
	\begin{equation}\label{4.7}
	\begin{split}
	P\left( (E_x^\xi)^c   \right)
	= &P\left(B_x(L_0)\xleftrightarrow{\pi_\xi(\omega)} \partial B_x(\bar{n}) \right)\\
	\le &P\left(0\xleftrightarrow{Z } \left\lbrace z\in \mathcal{V}_x^\xi: |z|\ge \bar{n} \right\rbrace  \right)\\
	\le &P\left(0\xleftrightarrow{W}\left\lbrace z\in \mathcal{V}_x^\xi: |z|\ge \bar{n} \right\rbrace \right) \le\boldsymbol{s.e.}(T).
	\end{split}
	\end{equation}
	Then we get Lemma \ref{lemma10} by (\ref{4.7}).
\end{proof}

For $x\in \mathbb{Z}^d$, let $x_{\mathcal{V}}$ be the closest site to $x$ in $\mathcal{V}$ (i.e. $x_{\mathcal{V}}^{(i)}$=$\min\left\lbrace m\in n*\mathbb{Z}: |m-x^{(i)}|\le 0.5n\right\rbrace $). Define an event:
\begin{equation}\label{dx}
\begin{split}
D_x
=&\{ Y_{12}(x_{\mathcal{V}})=1\}\cap E_x^\xi\cap 
\left[ \bigcap\limits_{\eta \in \mathcal{FI}_1^{0.5u,T}, \eta \cap B_x(\bar{n})\neq\emptyset, length(\eta)\ge T^\xi } \left\lbrace \eta \xleftrightarrow[B_{x_{\mathcal{V}}}(n)]{S_x^{2}}  \bar{\mathbb{B}}^{2}_{x_{\mathcal{V}}}\right\rbrace\right]\\  &\cap \left[ \bigcap\limits_{\eta \in \mathcal{FI}_2^{0.5u,T}, \eta \cap B_x(\bar{n})\neq \emptyset, length(\eta)\ge T^\xi} \left\lbrace \eta \xleftrightarrow[B_{x_{\mathcal{V}}}(n)]{S_x^{1}}  \bar{\mathbb{B}}^{1}_{x_{\mathcal{V}}}\right\rbrace\right]. 
\end{split}
\end{equation}
\begin{lemma}\label{lemma11}For any $x\in \mathbb{Z}^d$, 
	\begin{equation}\label{4.9}
	P(D_x)\ge 1-\boldsymbol{s.e.}(T).
	\end{equation}
\end{lemma}
\begin{proof}
	We denote the number of paths hitting $B_x(\bar{n})$ by $N_x$. By Lemma 2.1 of \cite{procaccia2019percolation}, Lemma \ref{lemma3.5} (in the appendix) and large deviation bound of Poisson distribution,  \begin{equation}\label{4.10}
		P\left( N_x\le cT^{\frac{d-2}{2}}\right)\ge 1-\boldsymbol{s.e.}(T).
	\end{equation}
	
	Consider a fixed path $\eta\in \mathcal{FI}^{0.5u,T}_1$ hitting $B_x(\bar{n})$ with $length(\eta)\ge T^\xi$. Choose a point $z_\eta\in B_x(\bar{n})\cap \eta$ and let $\bar{\eta}=\mathcal{C}_{z_\eta}^{\eta\cap B_{z_\eta}(T^\xi)}$. Obviously, $|\bar{\eta}|\ge T^\xi$. Assume $\mathcal{FI}_{2,2}^{0.25u,T}=\mathcal{FI}_{2,2,1}^{\frac{u}{8(d-2)},T}\cup \mathcal{FI}_{2,2,2}^{\frac{u}{8(d-2)},T}\cup...\cup\mathcal{FI}_{2,2,d-2}^{\frac{u}{8(d-2)},T}\cup \mathcal{FI}_{2,2,d-1}^{\frac{u}{8},T}$, where $\mathcal{FI}_{2,2,1}^{\frac{u}{8(d-2)},T},...,\mathcal{FI}_{2,2,d-2}^{\frac{u}{8(d-2)},T}, \mathcal{FI}_{2,2,d-1}^{\frac{u}{8},T}$ are independent. Recalling the notations $\Psi_y$ and $U_{x,i,y}^{(k)}$ in Section \ref{goodsites}, for $1\le k\le d-2$, we define that $\Psi_x(k,A)=\bigcup\limits_{\zeta \in \bar{S}_x(\mathcal{FI}_{2,2,k}^{\frac{u}{8(d-2)},T},A)}R(\zeta,2n^{2(a+\epsilon)})$; meanwhile, let $U_{x}^{(1)}(\bar{\eta})=\Psi_x(1,\bar{\eta})$ and $U_{x}^{(k)}(\bar{\eta})=\Psi_x(k,U_{x}^{(k-1)}(\bar{\eta}))$ for $2\le k\le d-2$.
	Similar to (\ref{35}), take $\xi>0$ small enough such that $\bar{\eta}\subset B_{z_\eta}(2n^{(a+\epsilon)(1+\epsilon_1)})$ and then we have \begin{equation}\label{4.11}
	P\left(\bigcup\limits_{k=1}^{d-2}U_{x}^{(k)}(\bar{\eta})\subset B_{z_{\eta}}\left(dn^{(a+\epsilon)(1+\epsilon_1)}\right), cap(U_{x}^{(d-2)}(\bar{\eta}))\ge cn^{(a+\epsilon)(d-2)(1-\epsilon_1)}\bigg|\eta\in \mathcal{FI}^{0.5u,T}_1   \right) \ge 1-\boldsymbol{s.e.}(T).
	\end{equation}
	
On the other hand, since $Y_2(x_{\mathcal{V}})=1$, the line segments $l_{z_\eta,1}$ is good for $\mathcal{FI}_{2,1}^{0.25u,T}$. We define $\bar{U}^{(k)}_{x,i,y}$ in the same say as $U^{(k)}_{x,i,y}$, using $\mathcal{FI}_{2,2,1}^{\frac{u}{8(d-2)},T},...,\mathcal{FI}_{2,2,d-2}^{\frac{u}{8(d-2)},T}$ instead of $\mathcal{FI}_{2,1}^{\frac{u}{3(d-2)},T},...,\mathcal{FI}_{2,d-2}^{\frac{u}{3(d-2)},T}$ in Section \ref{goodsites}. Similar to (\ref{35}), we have
\begin{equation}\label{4.12}
	P\left(\bigcup_{j=1}^{d-2} \bar{U}_{z_\eta,1,z_\eta}^{(j)}\subset B_{z_\eta}(dn^{(a+\epsilon)(1+\epsilon_1)}),cap(\bar{U}_{z_\eta,1,z_\eta}^{(d-2)})\ge cn^{(a+\epsilon)(d-2)(1-\epsilon_1)}\bigg|Y_2(x_{\mathcal{V}})=1 \right) \ge 1-\boldsymbol{s.e.}(T).
\end{equation}
Note that $\left\lbrace \eta\in \mathcal{FI}^{0.5u,T}_1  \right\rbrace  $ and $\left\lbrace Y_2(x_{\mathcal{V}})=1  \right\rbrace  $ are independent. By (\ref{4.11}), (\ref{4.12}) and Lemma \ref{connect}, we have
\begin{equation}\label{4.13}
\begin{split}
&P\left(\eta\xleftrightarrow[B_{z_\eta}(T^{\frac{1}{2d}})]{S_x^2} \bar{\mathbb{B}}^{2}_{x_{\mathcal{V}}}\bigg|  \eta\in \mathcal{FI}^{0.5u,T}_1 ,Y_2(x_{\mathcal{V}})=1  \right)\\
\ge &P\left(\bigcup\limits_{k=1}^{d-2}U_{x}^{(k)}(\bar{\eta})\xleftrightarrow[B_{z_\eta}(T^{\frac{1}{2d}})]{S_x(\mathcal{FI}_{2,2,d-1}^{\frac{u}{8},T})} \bigcup_{j=1}^{d-2} \bar{U}_{z_\eta,1,z_\eta}^{(j)}\bigg|  \eta\in \mathcal{FI}^{0.5u,T}_1 ,Y_2(x_{\mathcal{V}})=1  \right) \ge 1-\boldsymbol{s.e.}(T).
\end{split}
\end{equation}

Note that $B_{z_\eta}(T^{\frac{1}{2d}})\subset B_{x_{\mathcal{V}}}(n)$. By (\ref{4.10}), (\ref{4.13}) and uniform bound trick, we have
\begin{equation}\label{4.14}
	P\left( \bigcap\limits_{\eta \in \mathcal{FI}_1^{0.5u,T}, \eta \cap B_x(\bar{n})\neq \emptyset, length(\eta)\ge T^\xi} \left\lbrace \eta \xleftrightarrow[B_{x_{\mathcal{V}}}(n)]{S_x^{2}}  \bar{\mathbb{B}}^{2}_{x_{\mathcal{V}}}\right\rbrace\bigg|Y_2(x_{\mathcal{V}})=1  \right)\ge 1-\boldsymbol{s.e.}(T). 
\end{equation}
Equivalently, we have \begin{equation}\label{4.15}
	P\left( \bigcap\limits_{\eta \in \mathcal{FI}_2^{0.5u,T}, \eta \cap B_x(\bar{n})\neq \emptyset, length(\eta)\ge T^\xi} \left\lbrace \eta \xleftrightarrow[B_{x_{\mathcal{V}}}(n)]{S_x^{1}}  \bar{\mathbb{B}}^{1}_{x_{\mathcal{V}}}\right\rbrace\bigg|Y_1(x_{\mathcal{V}})=1   \right)\ge 1-\boldsymbol{s.e.}(T). 
\end{equation}
Combine (1) of Proposition \ref{Y}, Lemma \ref{lemma10}, (\ref{4.14}) and (\ref{4.15}), then we get (\ref{4.9}).
\end{proof}
\begin{lemma}\label{lemma12}
	For site percolation $\left\lbrace Y_{12}(y) \right\rbrace_{y\in \mathcal{V}}$, $\exists c(u,d,T)>0$ such that for any integer $M>0$,
	\begin{equation}
	P(0\xleftrightarrow{Y_{12} } \partial B^{\mathcal{V}}(M), 0\notin \Gamma^{(12)}_Y)\le e^{-cM},
	\end{equation}
	where $B^{\mathcal{V}}(M)=\{y\in\mathcal{V}:|y|\le M \}$.
\end{lemma}
\begin{proof}
	By (1) and (3) of Proposition \ref{Y}, we know that $\left\lbrace Y_{12}(y) \right\rbrace_{y\in \mathcal{V}}$ is $k$-independent for some $k(u,d)>0$ and for any $y\in \mathcal{V}$, $P\left(Y_{12}(y)=1 \right)\ge 1-\boldsymbol{s.e.}(T) $. Thus we can replace $\left\lbrace Z(y) \right\rbrace_{y\in \mathbb{Z}^d} $ in Lemma \ref{lemmaZ} by $\left\lbrace Y_{12}(y) \right\rbrace_{y\in \mathcal{V}}$ and all the arguments still hold.
	
	Like in Lemma \ref{lemmaZ}, we consider an auxiliary bond percolation $\left\lbrace R_{12}(e) \right\rbrace_{e\in \mathcal{H}}$: for any $e=\{y_1,y_2\}\in \mathcal{H}$, $R_{12}(e)=1$ if and only if $Y_{12}(y_1)=Y_{12}(y_2)=1$. Similarly, when $T$ is large enough, $\left\lbrace R_{12}(e) \right\rbrace_{e\in \mathcal{H}}$ stochasitcally dominates a supercritical bond percolation $\left\lbrace W_{12}(e) \right\rbrace_{e\in \mathcal{H}}$. We denote by $\Gamma_R^{(12)}$ and $\Gamma_W^{(12)}$ the unique connected clusters in $\left\lbrace R_{12}(e) \right\rbrace_{e\in \mathcal{H}}$ and $\left\lbrace W_{12}(e) \right\rbrace_{e\in \mathcal{H}}$.

	Replacing $W$ in Lemma \ref{lemmaZ} by $W_{12}$, we define that $\hat{\mathcal{C}}_0^{(12)}:=\mathcal{C}_{0}^{\left(\Gamma^{(12)}_W \right)^c}$.

	 If $\left\lbrace 0\xleftrightarrow{Y_{12} } \partial B^{\mathcal{V}}(M), 0\notin \Gamma^{(12)}_Y \right\rbrace $ , then $0\xleftrightarrow{\left(\Gamma^{(12)}_R \right)^c} \partial B^{\mathcal{V}}(M)$. Since $\Gamma^{(12)}_W\subset \Gamma^{(12)}_R$, we have $\left(\Gamma^{(12)}_R \right)^c\subset \left(\Gamma^{(12)}_W \right)^c$. Thus $0\xleftrightarrow{\left(\Gamma^{(12)}_W \right)^c} \partial B^{\mathcal{V}}(M) $, which means there exists $x\in B^{\mathcal{V}}(M)\cap \hat{\mathcal{C}}_0^{(12)}$. Therefore, $\left\lbrace D(0,\partial \hat{\mathcal{C}}_0^{(12)})\ge \lfloor \frac{M}{n} \rfloor \right\rbrace $ happends, where $D(0,\partial \hat{\mathcal{C}}_0^{(12)}):=\left\lbrace |y|_{\mathcal{V}}:y\in \hat{\mathcal{C}}_0^{(12)} \right\rbrace $.
	 
	  Recalling (\ref{useful1}), we have
	 \begin{equation}
	 	P(0\xleftrightarrow{Y_{12} } \partial B^{\mathcal{V}}(M), 0\notin \Gamma^{(12)}_Y)\le P\left(D(0,\partial \hat{\mathcal{C}}_0^{(12)})\ge \lfloor \frac{M}{n} \rfloor \right)\le e^{-cM}.
	 \end{equation}
\end{proof}
Define $F^m_x= \bigcap\limits_{\eta\in \mathcal{FI}^{u,T}, \eta(0)\in \left( B_x(l_0^m)\right) ^c}\left\lbrace \eta \cap B_{x}(\bar{n})= \emptyset \right\rbrace.$\\
\begin{lemma}\label{lemma13}For any $x\in \mathbb{Z}^d$ and integer $m\ge 1$, $\exists c(u,d)>0$ such that 
	\begin{equation}\label{4.18}
	P\left(F^m_x \right)\ge 1-e^{-cT^{-1}*l_0^m}.
	\end{equation}
\end{lemma}
\begin{proof}
	Since $\exists 0<\delta<1$ such that $\left(1-\frac{1}{T+1} \right)^{T}<\delta$ for any $T>0$, we have \begin{equation}\label{4.19}
		\begin{split}
		P\left((F_x^m)^c \right)
		\le &\sum_{|z-x|>l_0^m}P\left(\exists a\ path\ \eta\ starting\ from\ z\ and\ hitting\ B_{x}(\bar{n}) \right)\\
		\le &\sum_{|z-x|>l_0^m}\left(1-e^{-\frac{2du}{T+1}*\left(1-\frac{1}{T+1} \right)^{|z-x|-\bar{n}} }\right)\\
		\le &\frac{2du}{T+1}*\sum_{|z-x|>l_0^m}\delta^{T^{-1}*(|z-x|-\bar{n})}\\
		\le &\frac{2du}{T+1}\delta^{T^{-1}(l_0^m-\bar{n})}\sum_{k\ge 1}c(k+l_0^m)^{d-1}\delta^{T^{-1}*k}.
		\end{split}
	\end{equation}
	For integer $p\ge 0$, if $1+pT\le k\le (p+1)T$, then $(k+l_0^m)^{d-1}\delta^{T^{-1}*k}\le l_0^{m(d-1)}(p+1)^{d-1}\delta^{p}$. Thus,
	\begin{equation}\label{4.195}
		\begin{split}
		\sum_{k\ge 1}(k+l_0^m)^{d-1}\delta^{T^{-1}*k}\le T*l_0^{m(d-1)}\sum_{p\ge 0}(p+1)^{d-1}\delta^{p}\le cT*l_0^{m(d-1)}.
		\end{split}
	\end{equation}
	Recall that $l_0=\lfloor T^2\rfloor$. Combine (\ref{4.19}) and (\ref{4.195}), then (\ref{4.18}) follows.
\end{proof}
Now we are ready to conclude the proof of Proposition \ref{prop12}:
\begin{proof}
	It's sufficient to prove the case when $N>T^3$.
	
	For any $1\le m\le n_0$ and $\mathcal{T}\in \Lambda_{m,x}$, assume $\mathcal{T}\cap \mathbb{I}_0=\left\lbrace (0,x^i) \right\rbrace_{i=1}^{2^m} $. Frist, we are going to prove the following estimate (recall the definition of $D_x$ in (\ref{dx})): \begin{equation}\label{important}
		P\left[ \bigcap_{i=1}^{2^m}  \left( (D_{x^i})^c\cup\left\lbrace x^i_{\mathcal{V}} \notin \Gamma_Y^{(12)} \right\rbrace\right)\right]   \le \left(\boldsymbol{s.e.}(T) \right)^{2^m}.
	\end{equation}
	
	Define the truncated events of $D_x$:
	\begin{equation}\label{newdx}
	\begin{split}
	\bar{D}_x^m
	=&\{ Y_{12}(x_{\mathcal{V}})=1\}\cap E_x^{\xi}\cap 
	\left[\bigcap\limits_{\eta \in \mathcal{FI}_1^{0.5u,T}, \eta \cap B_x(\bar{n})\neq\emptyset, length(\eta)\ge T^\xi,\eta(0)\in B_x(l_0^m) } \left\lbrace \eta \xleftrightarrow[B_{x_{\mathcal{V}}}(n)]{S_x^{2}}  \bar{\mathbb{B}}^{2}_{x_{\mathcal{V}}}\right\rbrace\right]\\  &\cap \left[ \bigcap\limits_{\eta \in \mathcal{FI}_2^{0.5u,T}, \eta \cap B_x(\bar{n})\neq \emptyset, length(\eta)\ge T^\xi,\eta(0)\in B_x(l_0^m)} \left\lbrace \eta \xleftrightarrow[B_{x_{\mathcal{V}}}(n)]{S_x^{1}}  \bar{\mathbb{B}}^{1}_{x_{\mathcal{V}}}\right\rbrace\right]. 
	\end{split}
	\end{equation}
	It's easy to see that $D_x\cap F_x^m=\bar{D}_x^m\cap F_x^m$ and $D_x\subset \bar{D}_x^m$.
	
	Assume that $\mathcal{T}\cap \mathbb{I}_{m-1}=\left\lbrace (m-1, z_1), (m-1, z_2)\right\rbrace $. By definition, we know that $\left\lbrace x^i \right\rbrace_{i=1}^{2^m}$ can be devided into two sets $Z_1$ and $Z_2$, where both of them include $2^{m-1}$ sites and $Z_j\subset B_{z_j}(\widetilde{L}_{m-1})$, $j=1,2$. Without loss of generality, assume $Z_1=\left\lbrace x^i \right\rbrace_{i=1}^{2^{m-1}}$ and $Z_2=\left\lbrace x^i \right\rbrace_{i=2^{m-1}+1}^{2^{m}}$. Since $L_m=L_0*l_0^m$ and $L_0>100$, we know that $|z_1-z_2|>10l_0^m$ and $B_{z_1}(3l_0^m)\cap B_{z_2}(3l_0^m)=\emptyset$. We denote that $B_x^{\mathcal{V}}(M)=\{y\in \mathcal{V}:|y-x|\le M\}$. By definition of $\bar{D}_x^m$, $\bigcap\limits_{x\in Z_j} \left( (\bar{D}_{x}^m)^c\cup\left\lbrace x_{\mathcal{V}} \stackrel{Y_{12}}\nleftrightarrow \partial B^{\mathcal{V}}_{x_{\mathcal{V}}}(l_0^m) \right\rbrace\right)$ only depends on the paths starting from $B_{z_j}(3l_0^m)$, $j=1,2$. Since $B_{z_1}(3l_0^m)\cap B_{z_2}(3l_0^m)=\emptyset$, we have $\bigcap\limits_{x\in Z_1}\left( (\bar{D}_{x}^m)^c\cup\left\lbrace x_{\mathcal{V}} \stackrel{Y_{12}}\nleftrightarrow \partial B^{\mathcal{V}}_{x_{\mathcal{V}}}(l_0^m) \right\rbrace\right)$ and $\bigcap\limits_{x\in Z_2}\left( (\bar{D}_{x}^m)^c\cup\left\lbrace x_{\mathcal{V}} \stackrel{Y_{12}}\nleftrightarrow \partial B^{\mathcal{V}}_{x_{\mathcal{V}}}(l_0^m) \right\rbrace\right)$ are independent. By Lemma \ref{lemma12}, Lemma \ref{lemma13} and $(D_x)^c\cap F_x^m=(\bar{D}_x^m)^c\cap F_x^m$, 
	\begin{equation}\label{4.20}
		\begin{split}
		&	P\left[ \bigcap_{i=1}^{2^m} \left( (D_{x^i})^c\cup\left\lbrace x^i_{\mathcal{V}} \notin \Gamma_Y^{(12)} \right\rbrace\right) \right] \\
		\le &P\left[\bigcap_{i=1}^{2^m} \left( \left[ (D_{x^i})^c\cap F_{x^i}^m\right] \cup \left[ \left\lbrace x^i_{\mathcal{V}}\stackrel{Y_{12}}\nleftrightarrow \partial B^{\mathcal{V}}_{x^i_{\mathcal{V}}}(l_0^m) \right\rbrace\cap F_{x^i}^m \right] \right) \right]+P\left(\bigcup\limits_{i=1}^{2^m}(F_{x^i}^m)^c \right)  \\
		&+P\left[ \bigcup\limits_{i=1}^{2^m}\left\lbrace x^i_{\mathcal{V}} \xleftrightarrow{Y_{12}} \partial B^{\mathcal{V}}_{x^i_{\mathcal{V}}}(l_0^m),  x^i_{\mathcal{V}} \notin \Gamma_Y^{(12)}  \right\rbrace  \right] \\
		= &P\left[\bigcap_{i=1}^{2^m} \left( \left[ (\bar{D}_{x^i})^c\cap F_{x^i}^m\right] \cup \left[ \left\lbrace x^i_{\mathcal{V}}\stackrel{Y_{12}}\nleftrightarrow \partial B^{\mathcal{V}}_{x^i_{\mathcal{V}}}(l_0^m) \right\rbrace\cap F_{x^i}^m \right] \right) \right]+P\left(\bigcup\limits_{i=1}^{2^m}(F_{x^i}^m)^c \right)  \\
		&+P\left[ \bigcup\limits_{i=1}^{2^m}\left\lbrace x^i_{\mathcal{V}} \xleftrightarrow{Y_{12}} \partial B^{\mathcal{V}}_{x^i_{\mathcal{V}}}(l_0^m),  x^i_{\mathcal{V}} \notin \Gamma_Y^{(12)}  \right\rbrace  \right] \\
		\le & P\left[ \bigcap_{i=1}^{2^m}\left( (\bar{D}_{x^i}^m)^c\cup\left\lbrace x^i_{\mathcal{V}} \stackrel{Y_{12}}\nleftrightarrow \partial B^{\mathcal{V}}_{x^i_{\mathcal{V}}}(l_0^m) \right\rbrace\right) \right]+2^{m}*\left(e^{-cT^{-1}*l_0^m}+ e^{-c*l_0^m}\right) .
		\end{split}
	\end{equation}
 By independence between $\bigcap\limits_{x\in Z_1}\left( (\bar{D}_{x}^m)^c\cup\left\lbrace x_{\mathcal{V}} \stackrel{Y_{12}}\nleftrightarrow \partial B^{\mathcal{V}}_{x_{\mathcal{V}}}(l_0^m) \right\rbrace\right)$ and $\bigcap\limits_{x\in Z_2}\left( (\bar{D}_{x}^m)^c\cup\left\lbrace x_{\mathcal{V}} \stackrel{Y_{12}}\nleftrightarrow \partial B^{\mathcal{V}}_{x_{\mathcal{V}}}(l_0^m) \right\rbrace\right)$,  
 \begin{equation}\label{4.21}
	\begin{split}
&P\left[ \bigcap_{i=1}^{2^m}\left( (\bar{D}_{x^i}^m)^c\cup\left\lbrace x^i_{\mathcal{V}} \stackrel{Y_{12}}\nleftrightarrow \partial B^{\mathcal{V}}_{x^i_{\mathcal{V}}}(l_0^m) \right\rbrace\right) \right]\\
=&P\left[ \bigcap_{i=1}^{2^{m-1}}\left( (\bar{D}_{x^i}^m)^c\cup\left\lbrace x^i_{\mathcal{V}} \stackrel{Y_{12}}\nleftrightarrow \partial B^{\mathcal{V}}_{x^i_{\mathcal{V}}}(l_0^m) \right\rbrace\right) \right]
*P\left[ \bigcap_{i=2^{m-1}+1}^{2^m}\left( (\bar{D}_{x^i}^m)^c\cup\left\lbrace x^i_{\mathcal{V}} \stackrel{Y_{12}}\nleftrightarrow \partial B^{\mathcal{V}}_{x^i_{\mathcal{V}}}(l_0^m) \right\rbrace\right) \right] \\
\le &P\left[ \bigcap_{i=1}^{2^{m-1}} \left( (D_{x^i})^c\cup\left\lbrace x^i_{\mathcal{V}} \notin \Gamma_Y^{(12)} \right\rbrace\right) \right] *P\left[ \bigcap_{i=2^{m-1}+1}^{2^m} \left( (D_{x^i})^c\cup\left\lbrace x^i_{\mathcal{V}} \notin \Gamma_Y^{(12)} \right\rbrace\right) \right].	
\end{split}
	\end{equation}
	Note that $2^{m}*\left(e^{-cT^{-1}*l_0^m}+ e^{-c*l_0^m}\right)\le 2e^{-cT*2^m}$ when $T$ is large enough. By (\ref{4.20}) and (\ref{4.21}), \begin{equation}\label{4.23}
		\begin{split}
		&P\left[ \bigcap_{i=1}^{2^m} \left( (D_{x^i})^c\cup\left\lbrace x^i_{\mathcal{V}} \notin \Gamma_Y^{(12)} \right\rbrace\right) \right] \\
		\le& P\left[ \bigcap_{i=1}^{2^{m-1}} \left( (D_{x^i})^c\cup\left\lbrace x^i_{\mathcal{V}} \notin \Gamma_Y^{(12)} \right\rbrace\right) \right] *P\left[ \bigcap_{i=2^{m-1}+1}^{2^m} \left( (D_{x^i})^c\cup\left\lbrace x^i_{\mathcal{V}} \notin \Gamma_Y^{(12)} \right\rbrace\right) \right]  +2e^{-cT*2^m}.
		\end{split}
	\end{equation}
	For any $x^i$, by Lemma \ref{lemma11} and Lemma \ref{lemma12}, 
	\begin{equation}\label{4.24}
	\begin{split}
	P\left[  (D_{x^i})^c\cup\left\lbrace x^i_{\mathcal{V}} \notin \Gamma_Y^{(12)} \right\rbrace \right] \le& P\left[ (D_{x^i})^c \right]+ P\left(x^i_{\mathcal{V}}\notin \Gamma_Y^{(12)},x^i_{\mathcal{V}}\xleftrightarrow{Y_{12}} \partial B^{\mathcal{V}}_{x^i_{\mathcal{V}}}(T^2)  \right)+P\left(x^i_{\mathcal{V}}\stackrel{Y_{12}}\nleftrightarrow \partial B^{\mathcal{V}}_{x^i_{\mathcal{V}}}(T^2) \right)  \\
	\le &P\left(x^i_{\mathcal{V}}\stackrel{Y_{12}}\nleftrightarrow \partial B^{\mathcal{V}}_{x^i_{\mathcal{V}}}(T^2) \right) +\boldsymbol{s.e.}(T)\\
	\le& \sum\limits_{y\in B_{x^i_{\mathcal{V}}}(T^2)\cap \mathcal{V}}P\left(Y_{12}(y)=0 \right) +\boldsymbol{s.e.}(T)\le \boldsymbol{s.e.}(T).
	\end{split}
	\end{equation}
	By (\ref{4.23}) and (\ref{4.24}), we have \begin{equation}
		\begin{split}
		&P\left[ \bigcap_{i=1}^{2^m} \left( (D_{x^i})^c\cup\left\lbrace x^i_{\mathcal{V}} \notin \Gamma_Y^{(12)} \right\rbrace\right)\right] +2e^{-cT*2^m}\\
		\le& P\left[ \bigcap_{i=1}^{2^{m-1}} \left( (D_{x^i})^c\cup\left\lbrace x^i_{\mathcal{V}} \notin \Gamma_Y^{(12)} \right\rbrace\right)\right] *P\left[ \bigcap_{i=2^{m-1}+1}^{2^m} \left( (D_{x^i})^c\cup\left\lbrace x^i_{\mathcal{V}} \notin \Gamma_Y^{(12)} \right\rbrace\right)\right] +4e^{-cT*2^m}\\
		\le &\left(P\left[ \bigcap_{i=1}^{2^{m-1}} \left( (D_{x^i})^c\cup\left\lbrace x^i_{\mathcal{V}} \notin \Gamma_Y^{(12)} \right\rbrace\right)\right] +2e^{-cT*2^{m-1}} \right)\\
		&*\left(P\left[ \bigcap_{i=2^{m-1}+1}^{2^m} \left( (D_{x^i})^c\cup\left\lbrace x^i_{\mathcal{V}} \notin \Gamma_Y^{(12)} \right\rbrace\right)\right] +2e^{-cT*2^{m-1}} \right) \\
		\le &...\le \prod_{i=1}^{2^m} \left(  P\left[  (D_{x^i})^c\cup\left\lbrace x^i_{\mathcal{V}} \notin \Gamma_Y^{(12)} \right\rbrace \right] +2e^{-cT}\right)  \le \left(\boldsymbol{s.e.}(T) \right)^{2^m}.
		\end{split}
	\end{equation}
	Thus we get (\ref{important}).
	
	If $\left\lbrace 0\xleftrightarrow{} \partial B(N^{\frac{1}{2d}})  \right\rbrace $ happens, then $A_{n_0, 0}$ also happens. By (\ref{decom}), we have\begin{equation}
		A_{n_0, 0}\subset \bigcup\limits_{\mathcal{T}\in \Lambda_{n_0,0}}A_{\mathcal{T}}.
	\end{equation}
	Therefore,
	\begin{equation}\label{4.28}
		P\left[0\xleftrightarrow{} \partial B(N^{\frac{1}{2d}}), \rho(0,\bar{\Gamma}_{12})>N  \right]\le P\left(A_{n_0, 0},\rho(0,\bar{\Gamma}_{12})>N \right) \le \sum\limits_{\mathcal{T}\subset \Lambda_{n_0,0}}P\left(A_{\mathcal{T}}, \rho(0,\bar{\Gamma}_{12})>N \right). 
	\end{equation}
	
	For any $\mathcal{T}\subset \Lambda_{n_0,0}$, we assume that $\mathcal{T}\cap \mathbb{I}_0=\left\lbrace (0,x^i) \right\rbrace_{i=1}^{2^{n_0}} $. If there exists $x^i\in\left\lbrace x^i \right\rbrace_{i=1}^{2^{n_0}} $ such that $D_{x^i}$, $\left\lbrace x^i_{\mathcal{V}} \in \Gamma_Y^{(12)}\right\rbrace $ and $A_{0, x^i}$ happen, then $\mathcal{C}_0(\widetilde{B}_{0,n_0})\cap \left(\bar{\mathbb{B}}^{1}_{x^i_{\mathcal{V}}}\cup \bar{\mathbb{B}}^{2}_{x^i_{\mathcal{V}}} \right)\neq \emptyset $ (note that the event $D_{x^i}$ ensures that any cluster hitting $B_{0,x_i}$ and $\partial \widetilde{B}_{0,x_i}$ must contain a path with length greater than $T^\xi$ and this path must be connected to $\bar{B}^1_{x^i_{\mathcal{V}}}$ or $\bar{B}^2_{x^i_{\mathcal{V}}}$). Thus $\mathcal{C}_0(\widetilde{B}_{0,n_0})\cap \bar{\Gamma}_{12}\neq \emptyset$, which yields that $\left\lbrace \rho(0, \bar{\Gamma}_{12})\le c_1N^{0.5} \right\rbrace $ for some $c_1(d)>0$. In conclusion, for any $ \mathcal{T}\in \Lambda_{n_0,0}$, 
	\begin{equation}\label{4.32}
		A_{\mathcal{T}}\cap \left( \bigcup_{i=1}^{2^{n_0}}D_{x^i}\cap \left\lbrace x^i_{\mathcal{V}} \in \Gamma_Y^{(12)} \right\rbrace\right) \subset A_{\mathcal{T}}\cap\left\lbrace \rho(0, \bar{\Gamma}_{12})\le c_1N^{0.5} \right\rbrace.
	\end{equation}
Then by (\ref{important}) and (\ref{4.32}),
	\begin{equation}\label{4.29}
		P\left(A_{\mathcal{T}}, \rho(0,\bar{\Gamma}_{12})>c_1N^{0.5}  \right)\le P\left[ A_{\mathcal{T}},\bigcap_{i=1}^{2^{n_0}}\left( (D_{x^i})^c\cup\left\lbrace x^i_{\mathcal{V}} \notin \Gamma_Y^{(12)} \right\rbrace\right)\right] \le \left(\boldsymbol{s.e.}(T) \right)^{2^{n_0}}.
	\end{equation}
	
	 Note that $\exists c>0$ such that $2^{n_0}\ge N^c$ and when $T$ is large enough, any $N\in (T^3,\infty)$ satisfies $N>c_1N^{0.5}$. By $|\Lambda_{n_0,0}|\le \left(c_0l_0^{2(d-1)} \right)^{2^{n_0}} $,(\ref{4.28}) and (\ref{4.29}), we have: for large enough $T$,
	 \begin{equation}
		P\left[0\xleftrightarrow{} \partial B(N^{\frac{1}{2d}}), \rho(0,\bar{\Gamma}_{12})>N  \right]\le \left(c_0l_0^{2(d-1)}*\boldsymbol{s.e.}(T) \right)^{2^{n_0}}\le \boldsymbol{s.e.}(N).
	\end{equation}
	Now we complete the proof of Proposition \ref{prop12}.
\end{proof}
\section{The chemical distance on $\Gamma$ is good}
In this section, we are going to prove Theorem \ref{theorem1}. Assume $|y|=N$ and it's sufficient to prove the case when $N\ge 3T$. Note that $\bar{\Gamma}_{12}\subset \bar{\Gamma}_{1}\cup \bar{\Gamma}_{2}$ and $\forall x_1\in B_0(N^{0.5}), \forall x_2\in B_y(N^{0.5})$, $3N\ge |x_1-x_2|\ge \frac{1}{3}N$. By Proposition \ref{prop12}, , we have
 \begin{equation}\label{5.1}
\begin{split}
	&P\left(0, y\in \Gamma, \rho(0,y)>(3c+2)N \right) \\
\le &P\left(0, y\in \Gamma, \rho(0,y)>(3c+2)N , \rho(0,\bar{\Gamma}_{12})\le N^{0.5}, \rho(y,\bar{\Gamma}_{12})\le N^{0.5} \right) +\boldsymbol{s.e.}(N)\\
\le &\sum\limits_{x_1\in B_0(N^{0.5}),x_2\in B_y(N^{0.5})}P\left(x_1, x_2\in \bar{\Gamma}_{12}, \rho(x_1,x_2)> 3cN \right) +\boldsymbol{s.e.}(N)\\
\le &\sum\limits_{x_1\in B_0(N^{0.5}),x_2\in B_y(N^{0.5})}P\left(x_1, x_2\in \bar{\Gamma}_{12}, \rho(x_1,x_2)> c*|x_1-x_2|\right) +\boldsymbol{s.e.}(N).
\end{split}
\end{equation}
\begin{lemma}\label{lemma14}
	For large enough $T$, there exists $c(u,d)>0$ such that, for any $ |x_1-x_2|\ge T$, 
\begin{equation}
	P\left(x_1, x_2\in \bar{\Gamma}_{12}, \rho(x_1,x_2)> c*|x_1-x_2|\right) \le \boldsymbol{s.e.}(|x_1-x_2|).
\end{equation}
\end{lemma}
\begin{proof}
	Let $\bar{N}=|x_1-x_2|\ge T$. Since $\bar{\Gamma}_{12}\subset \bar{\Gamma}_{1}\cup \bar{\Gamma}_{2}$, we have \begin{equation}\label{5.3}
	\begin{split}
	&P\left(x_1, x_2\in \bar{\Gamma}_{12}, \rho(x_1,x_2)> c\bar{N}\right)\\
	\le & P\left(x_1, x_2\in \bar{\Gamma}_{1}, \rho_{\bar{\Gamma}_{1}}(x_1,x_2)> c\bar{N}\right)+P\left(x_1, x_2\in \bar{\Gamma}_{2}, \rho_{\bar{\Gamma}_{2}}(x_1,x_2)> c\bar{N}\right)\\
	&+P\left(x_1\in \bar{\Gamma}_{12}\cap\bar{\Gamma}_{1}, x_2 \in \bar{\Gamma}_{12}\cap\bar{\Gamma}_{2},\rho(x_1,x_2)> c\bar{N}\right) \\
	&+P\left(x_1\in \bar{\Gamma}_{12}\cap\bar{\Gamma}_{2}, x_2 \in \bar{\Gamma}_{12}\cap\bar{\Gamma}_{1},\rho(x_1,x_2)> c\bar{N}\right). 
	\end{split}		
	\end{equation}
	
Let $c'$ be the constant in Propostion \ref{bargamma} and $c> c'$. By Propostion \ref{bargamma}, for $j=1,2$, we have
\begin{equation}\label{5.4}
\begin{split}
P\left(x_1, x_2\in \bar{\Gamma}_{j}, \rho_{\bar{\Gamma}_{j}}(x_1,x_2)> c\bar{N}\right)\le \boldsymbol{s.e.}(\bar{N}).
\end{split}
\end{equation}

For the remaining parts in (\ref{5.3}), if $x_1\in \bar{\Gamma}_{12}\cap\bar{\Gamma}_{1}$, there must exist $y_1\in \Gamma_Y^{(12)}$ such that $x_1\in \bar{\mathbb{B}}^1_{y_1}$. Since $\bar{\mathbb{B}}^1_{y_1}\subset B_{y_1}(4n)$, we have $y_1\in B_{x_1}(4n)$. Similarly, there exists $y_2\in \Gamma_Y^{(12)}\cap B_{x_2}(4n)$ such that $x_2\in \bar{\mathbb{B}}^2_{y_2}$. Thus we have  \begin{equation}\label{5.5}
\begin{split}
	&P\left(x_1\in \bar{\Gamma}_{12}\cap\bar{\Gamma}_{1}, x_2 \in \bar{\Gamma}_{12}\cap\bar{\Gamma}_{2},\rho(x_1,x_2)> c\bar{N}\right)\\
	\le &\sum\limits_{y_1\in B_{x_1}(4n)\cap \mathcal{V},y_2\in B_{x_2}(4n)\cap \mathcal{V} } P\left(y_1,y_2\in \Gamma_Y^{(12)},x_1\in \bar{\mathbb{B}}^1_{y_1},x_2\in \bar{\mathbb{B}}^2_{y_2},\rho(x_1,x_2)> c\bar{N}  \right). 
\end{split}
\end{equation}
By Lemma \ref{lemmaZ}, there exists $c''(d)>0$ such that $\forall y_1, y_2\in \mathcal{V}$, \begin{equation}\label{5.6}
	P\left(y_1,y_2\in \Gamma_Y^{(12)}, \rho_{\Gamma_Y^{(12)}}(y_1,y_2)\ge c''|y_1-y_2|_{\mathcal{V}} \right)\le \boldsymbol{s.e.}(|y_1-y_2|).
\end{equation}
If $\left\lbrace \rho_{\Gamma_Y^{(12)}}(y_1,y_2)< c''|y_1-y_2|_{\mathcal{V}} \right\rbrace $ happens, we fix all paths starting from $B_{y_1}\left(2c''|y_1-y_2| \right) $ and then choose a sequence of open vertices $\left(z_0,z_2,...,z_m \right) $ in $\left\lbrace Y_{12}(y) \right\rbrace_{y\in \mathcal{V}} $ such that $z_0=y_1$, $z_m=y_2$ and $|y_1-y_2|_{\mathcal{V}}\le m< c''|y_1-y_2|_{\mathcal{V}}$. Meanwhile, we fix $\bar{\mathbb{B}}_{z_i}^1$ and $\bar{\mathbb{B}}_{z_i}^2$, for $i=0,1,...,m$. For any fixed $\bar{\mathbb{B}}_{z_i}^1$ and $\bar{\mathbb{B}}_{z_i}^2$, using the same approach as in the proof of Lemma \ref{goodsite}, we have 
\begin{equation}
	P\left(\bar{\mathbb{B}}_{z_i}^1\xleftrightarrow[B_{z_i}(T^{\frac{1}{2d}})]{S_{z_i}\left(\mathcal{FI}_{2,2}^{0.25u,T} \right) } \bar{\mathbb{B}}_{z_i}^2\bigg|Y_{12}(z_i)=1 \right) \ge 1-\boldsymbol{s.e.}(T).
\end{equation}
In $\left(z_0,z_2,...,z_m \right)$, we can choose $k=\lfloor \frac{|y_1-y_2|}{20T^{0.5}}\rfloor$ sites $\left\lbrace z_{i_1}, z_{i_2},...,z_{i_k}\right\rbrace $ such that $|z_{i_s}-z_{i_t}|\ge 20T^{0.5}$ for any $z_{i_s}\neq z_{i_t} $ in it. By (3) of Proposition \ref{Y}, we know that $\left\lbrace \bar{\mathbb{B}}_{z_{i_s}}^1\xleftrightarrow[B_{z_{i_s}}(T^{\frac{1}{2d}})]{S_{z_{i_s}}\left(\mathcal{FI}_{2,2}^{0.25u,T} \right) } \bar{\mathbb{B}}_{z_{i_s}}^2 \right\rbrace $, $s=1,2,...,k$ are independent. Thus \begin{equation}\label{5.8}
	P\left(\bigcap_{i=0}^{m} \left\lbrace \bar{\mathbb{B}}_{z_i}^1\xleftrightarrow[B_{z_i}(T^{\frac{1}{2d}})]{S_{z_i}\left(\mathcal{FI}_{2,2}^{0.25u,T} \right) } \bar{\mathbb{B}}_{z_i}^2 \right\rbrace^c\bigg|\forall 0\le i\le m,Y_{12}(z_i)=1\right) \le \left(\boldsymbol{s.e.}(T)\right)^{k}\le \boldsymbol{s.e.}(|y_1-y_2|).
\end{equation}
If there exists $z_i$ such that $\left\lbrace \bar{\mathbb{B}}_{z_i}^1\xleftrightarrow[B_{z_i}(T^{\frac{1}{2d}})]{S_{z_i}\left(\mathcal{FI}_{2,2}^{0.25u,T} \right) } \bar{\mathbb{B}}_{z_i}^2  \right\rbrace $ happens, by (5) of Proposition \ref{Y},  \begin{equation}\label{5.9}
	\rho(x_1,x_2)\le m*c'''T^{0.5}+(2T^{\frac{1}{2d}}+1)^d< 2c''c'''b^{-1}|y_1-y_2|+c_1T^{0.5},
\end{equation}
where $c'''=c'''(u,d)$ and $c_1=c_1(d)$.

Take $c=\max\left\lbrace c',20c''c'''b^{-1}+c_1\right\rbrace$. Since $|y_1-y_2|\le 10\bar{N} $, we have \begin{equation}\label{5.10}
	c\bar{N}\ge 2c''c'''b^{-1}|y_1-y_2|+c_1T^{0.5}.
\end{equation}
By (\ref{5.9}) and (\ref{5.10}), we have
\begin{equation}\label{511}
\begin{split}
\bigcup_{i=0}^{m} \left\lbrace \bar{\mathbb{B}}_{z_i}^1\xleftrightarrow[B_{z_i}(T^{\frac{1}{2d}})]{S_{z_i}\left(\mathcal{FI}_{2,2}^{0.25u,T} \right) } \bar{\mathbb{B}}_{z_i}^2 \right\rbrace \cap \left\lbrace \rho_{\Gamma_Y^{(12)}}(y_1,y_2)< c'|y_1-y_2|_{\mathcal{V}} \right\rbrace
\subset\left\lbrace \rho(x_1,x_2)\le c\bar{N} \right\rbrace.
\end{split}
\end{equation}
By (\ref{5.6}), (\ref{5.8}), (\ref{511}) and uniform bound trick, \begin{equation}\label{5.11}
\begin{split}
&P\left(y_1,y_2\in \Gamma_Y^{(12)},x_1\in \bar{\mathbb{B}}^1_{y_1},x_2\in \bar{\mathbb{B}}^2_{y_2},\rho(x_1,x_2)> c\bar{N}\right)
\le \boldsymbol{s.e.}(|y_1-y_2|).
\end{split}
\end{equation}
Combine (\ref{5.5}) and (\ref{5.11}), 
\begin{equation}\label{5.12}
\begin{split}
	&P\left(x_1\in \bar{\Gamma}_{12}\cap\bar{\Gamma}_{1}, x_2 \in \bar{\Gamma}_{12}\cap\bar{\Gamma}_{2},\rho(x_1,x_2)> c\bar{N}\right)\\
	\le &\sum\limits_{y_1\in B_{x_1}(4n)\cap \mathcal{V},y_2\in B_{x_2}(4n)\cap \mathcal{V}}\boldsymbol{s.e.}(|y_1-y_2|)\le \boldsymbol{s.e.}(\bar{N}).
\end{split}
\end{equation}
Equivalently, we also have 
\begin{equation}\label{5.14}
P\left(x_1\in \bar{\Gamma}_{12}\cap\bar{\Gamma}_{2}, x_2 \in \bar{\Gamma}_{12}\cap\bar{\Gamma}_{1},\rho(x_1,x_2)> c\bar{N}\right)\le \boldsymbol{s.e.}(\bar{N}).
\end{equation}

By (\ref{5.3}), (\ref{5.4}), (\ref{5.12}) and (\ref{5.14}), we finally get Lemma \ref{lemma14}.
\end{proof}
By (\ref{5.1}) and Lemma \ref{lemma14},
\begin{equation}
\begin{split}
&P\left(0, y\in \Gamma, \rho(0,y)>(3c+2)N \right)\\
\le &\sum\limits_{x_1\in B_0(N^{0.5}),x_2\in B_y(N^{0.5})} \boldsymbol{s.e.}(|x_1-x_2|)\le \boldsymbol{s.e.}(N).
\end{split}
\end{equation}
Now we complete the proof of Theorem \ref{theorem1}.
\section{Local uniqueness of FRI}
In this section, we are going to prove Theorem \ref{theorem2}. It's sufficient to prove the case when $T$ is large enough and $R\ge T^2$. By Proposition \ref{prop12} and $\bar{\Gamma}_{12}\subset \Gamma$, 
\begin{equation}\label{6.1}
\begin{split}
P[&\exists two\ clusters\ in\ \mathcal{FI}^{u,T}\cap B(R)\ having\ diameter\ at\ least\\  
&\frac{R}{10}\ not\ connected\ to\ each\ other\ in\ B(2R)]\\
\le &\sum\limits_{x_1,x_2\in B(R)}P\left(x_1\xleftrightarrow{} \partial B_{x_1}(\frac{R}{20}), x_2\xleftrightarrow{} \partial B_{x_2}(\frac{R}{20}), \left\lbrace x_1\xleftrightarrow[B(2R)]{} x_2\right\rbrace^c \right)\\
\le &\sum\limits_{x_1,x_2\in B(R)}P\left(x_1\xleftrightarrow{} \partial B_{x_1}(\frac{R}{20}), x_2\xleftrightarrow{} \partial B_{x_2}(\frac{R}{20}), \left\lbrace x_1\xleftrightarrow[B(2R)]{} x_2\right\rbrace^c,x_1\in \Gamma,x_2\in \Gamma\right)\\
&+P\left(x_1\xleftrightarrow{} \partial B_{x_1}(\frac{R}{20}),\rho(x_1,\bar{\Gamma}_{12})>\left(\frac{R}{20} \right)^{2d}\right)+ P\left(x_2\xleftrightarrow{} \partial B_{x_2}(\frac{R}{20}),\rho(x_2,\bar{\Gamma}_{12})>\left(\frac{R}{20} \right)^{2d}\right)\\
\le &\sum\limits_{x_1,x_2\in B(R)}P\left(x_1,x_2\in \Gamma, \left\lbrace x_1\xleftrightarrow[B(2R)]{\Gamma} x_2\right\rbrace^c\right)+\boldsymbol{s.e.}(R).
\end{split}
\end{equation}

Before showing the result of interest, we first prove a weaker estimate: $\exists c(u,d)>1$ such that for any $ x_1,x_2\in B(R)$, 
\begin{equation}\label{7.2}
	P\left(x_1,x_2\in \Gamma, \left\lbrace x_1\xleftrightarrow[B(cR)]{\Gamma} x_2\right\rbrace^c\right)\le \boldsymbol{s.e.}(R). 
\end{equation}

First, we are going to prove (\ref{7.2}) in the case when $|x_1-x_2|>\frac{R}{30}$. Let $c'(u,d)$ be the constant in Lemma \ref{lemma14} and $c=1.1c'+1$. By Proposition \ref{prop12}, 
\begin{equation}\label{7.3}
	\begin{split}
	&P\left(x_1,x_2\in \Gamma, \left\lbrace x_1\xleftrightarrow[B(cR)]{\Gamma} x_2\right\rbrace^c \right)\\
	\le &P\left(x_1,x_2\in \Gamma, \rho(x_1,x_2)>cR, \rho(x_1, \bar{\Gamma}_{12})\le \frac{R}{90}, \rho(x_2, \bar{\Gamma}_{12})\le \frac{R}{90} \right)\\
	&+P\left(x_1\xleftrightarrow{} \partial B_{x_1}\left( (\frac{R}{90})^{\frac{1}{2d}}\right), \rho(x_1, \bar{\Gamma}_{12})> \frac{R}{90} \right)+P\left(x_2\xleftrightarrow{} \partial B_{x_2}\left( (\frac{R}{90})^{\frac{1}{2d}}\right), \rho(x_2, \bar{\Gamma}_{12})> \frac{R}{90}  \right) \\
	\le &\sum\limits_{z_1\in B_{x_1}(\frac{R}{90}),z_2\in B_{x_2}(\frac{R}{90}) }P\left(z_1,z_2\in\bar{\Gamma}_{12}, \rho(z_1,z_2)>1.1c'R \right)+\boldsymbol{s.e.}(R)
	\end{split}
\end{equation}
For any $ x_1,x_2\in B(R)$, note that $\forall z_1\in B_{x_1}(\frac{R}{90}),\forall z_2\in B_{x_2}(\frac{R}{90})$, $\frac{R}{90}\le |z_1-z_2|\le 1.1R$. By Lemma \ref{lemma14}, we have
\begin{equation}\label{704}
\begin{split}
&\sum\limits_{z_1\in B_{x_1}(\frac{R}{90}),z_2\in B_{x_2}(\frac{R}{90}) }P\left(z_1,z_2\in\bar{\Gamma}_{12}, \rho(z_1,z_2)>1.1c'R \right)\\
\le &\sum\limits_{z_1\in B_{x_1}(\frac{R}{90}),z_2\in B_{x_2}(\frac{R}{90}) }P\left(z_1,z_2\in\bar{\Gamma}_{12}, \rho(z_1,z_2)>c'|z_1-z_2|\right) \\
\le &\sum\limits_{z_1\in B_{x_1}(\frac{R}{90}),z_2\in B_{x_2}(\frac{R}{90})}\boldsymbol{s.e.}(|z_1-z_2|)\le \boldsymbol{s.e.}(R).
\end{split}
\end{equation}
Combine (\ref{7.3}) and (\ref{704}), we have: for any $x_1,x_2\in B(R)$ such that $|x_1-x_2|>\frac{R}{30}$, 
\begin{equation}\label{6.2}
	P\left(x_1,x_2\in \Gamma, \left\lbrace x_1\xleftrightarrow[B(cR)]{\Gamma} x_2\right\rbrace^c \right)\le \boldsymbol{s.e.}(R).
\end{equation}

Now let's consider the remaining case ``$|x_1-x_2|\le \frac{R}{30}$''. We define an event $\widetilde{LR}(M):=\left\lbrace left\ of\ B(M)\xleftrightarrow[B(M)]{\Gamma} right\ of\ B(M) \right\rbrace$. By Theorem \ref{theorem3} and Proposition \ref{prop12}, we have \begin{equation}\label{74}
\begin{split}
P\left((\widetilde{LR}(M))^c\right)\le & P\left((\widetilde{LR}(M))^c, LR(M)\right)+P\left((LR(M))^c \right) \\
\le &\sum_{x\in\ left\ of\ B(M)}P\left((\widetilde{LR}(M))^c, x\xleftrightarrow[B(M)]{} right\ of\ \partial B(M)\right)+\boldsymbol{s.e.}(M)\\
\le &\sum_{x\in\ left\ of\ B(M)}P\left(x\xleftrightarrow[]{} B_x(M),\rho(x,\Gamma)=\infty\right)+\boldsymbol{s.e.}(M)\\
\le &\boldsymbol{s.e.}(M).
\end{split}
\end{equation}
If the event $\widetilde{LR}(R)$ occurs, there must exist two sites $z_1,z_2\in B(R)\cap \Gamma$ such that $|z_1-z_2|\ge 2R$ and thus $\exists z\in \{z_1,z_2\}$ such that $|z-x_1|\ge R$ (since $2R\le |z_1-z_2|\le |z_1-x_1|+|z_2-x_1|$). Since $|z-x_1|\ge R$ and $|x_1-x_2|\le \frac{R}{30}$, we have $|z-x_2|>\frac{R}{30}$. Therefore, by (\ref{6.2}) and (\ref{74}), we have
\begin{equation}\label{7.4}
	\begin{split}
	&P\left(x_1,x_2\in \Gamma, \left\lbrace x_1\xleftrightarrow[B(cR)]{\Gamma} x_2\right\rbrace^c \right)\\
	\le &P\left(x_1,x_2\in \Gamma, \left\lbrace x_1\xleftrightarrow[B(cR)]{\Gamma} x_2\right\rbrace^c,\widetilde{LR}(R)  \right)+\boldsymbol{s.e.}(R)\\
	\le &\sum_{z\in B(R), |z-x_1|>\frac{R}{30},|z-x_2|>\frac{R}{30}}P\left(x_1,x_2,z\in \Gamma,\left\lbrace x_1\xleftrightarrow[B(cR)]{\Gamma} x_2\right\rbrace^c\right)+\boldsymbol{s.e.}(R)\\
	\le &\sum_{z\in B(R), |z-x_1|>\frac{R}{30},|z-x_2|>\frac{R}{30}}\left[ P\left(x_1,z\in \Gamma,\left\lbrace x_1\xleftrightarrow[B(cR)]{\Gamma} z\right\rbrace^c \right)+P\left(x_2,z\in \Gamma,\left\lbrace x_2\xleftrightarrow[B(cR)]{\Gamma} z\right\rbrace^c \right)\right]+\boldsymbol{s.e.}(R)\\
	\le &\sum_{z\in B(R), |z-x_1|>\frac{R}{30},|z-x_2|>\frac{R}{30}}\boldsymbol{s.e.}(R)+\boldsymbol{s.e.}(R)\le \boldsymbol{s.e.}(R).
	\end{split}
\end{equation}
Combine (\ref{6.2}) and (\ref{7.4}), we get (\ref{7.2}).

Now we improve $c$ in (\ref{7.2}) to $2$ using the approach in Section 4.3 of \cite{rath2011transience}. Assume $R'=\lfloor \frac{R}{2c}\rfloor $. Define $A_x^1=\left\lbrace B_x(R')\cap \Gamma\neq \emptyset \right\rbrace $, $A^2_x=\left\lbrace \forall y_1,y_2\in B_x(2R')\cap \Gamma,y_1\xleftrightarrow[B_x(2cR')]{\Gamma} y_2\right\rbrace$ and $A=\bigcap\limits_{x\in B(R)}\left(A_x^1\cap A_x^2 \right) $. By Theorem \ref{theorem3} and Proposition \ref{prop12}, 
\begin{equation}\label{75}
\begin{split}
P\left((A^1_x)^c\right)=P\left((A^1_0)^c\right)\le &P\left(LR(R'),(A^1_0)^c\right)+\boldsymbol{s.e.}(R)\\
\le& \sum_{z\in B(R')}P\left(z\xleftrightarrow{}\partial B_z(R'),\rho(z,\bar{\Gamma}_{12})=\infty \right)+\boldsymbol{s.e.}(R)\\
\le &\boldsymbol{s.e.}(R). 
\end{split}
\end{equation}
By (\ref{7.2}), we have 
\begin{equation}\label{76}
P\left((A^2_x)^c \right)\le \sum_{y_1,y_2\in B_x(2R')}P\left(y_1,y_2\in B_x(2R')\cap \Gamma,  \left( y_1\xleftrightarrow[B_x(2cR')]{\Gamma} y_2\right)^c\right)\le \boldsymbol{s.e.}(R).
\end{equation}
Combine (\ref{75}) and (\ref{76}), we have \begin{equation}\label{7.7}
	P\left(A\right)\ge 1-\boldsymbol{s.e.}(R).
\end{equation}

We will prove that event $A$ implies $\left\lbrace \forall x_1,x_2\in B(R)\cap \Gamma\ and\ |x_1-x_2|=k, x_1\xleftrightarrow[B(2R)]{\Gamma}x_2 \right\rbrace $ for $1\le k\le 2R$ by induction. The cases when $1\le k\le 4R'$ are elementary. Assume that $A$ and $\bigcap\limits_{l=1}^{k}\left\lbrace \forall x_1,x_2\in B(R)\cap \Gamma\ and\ |x_1-x_2|=l, x_1\xleftrightarrow[B(2R)]{\Gamma}x_2 \right\rbrace $ occur, where $k\ge 4R'$. For any  $z_1,z_2\in B(R)\cap \Gamma$ such that $|z_1-z_2|=k+1$, we know that $z:=\left(\lfloor \frac{z_1^{(1)}+z_2^{(1)}}{2}\rfloor ,...,\lfloor \frac{z_1^{(d)}+z_2^{(d)}}{2} \rfloor\right) $ satisfying $B_z(R')\subset B_{z_1}(k)\cap B_{z_2}(k)$. Since $A_z^1$ occurs, there exists $w\in B_z(R')\cap \Gamma$. Therefore, $w\xleftrightarrow[B(2R)]{\Gamma}z_1$ and $w\xleftrightarrow[B(2R)]{\Gamma}z_2$, which implies $z_1\xleftrightarrow[B(2R)]{\Gamma}z_2$. In conclusion, the induction is completed. 

By (\ref{7.7}), for any $ x_1,x_2\in B(R) $, 
\begin{equation}\label{7.8}
	P\left(x_1,x_2\in \Gamma, \left\lbrace x_1\xleftrightarrow[B(2R)]{\Gamma}x_2\right\rbrace^c \right)\le P\left(A^c \right)\le \boldsymbol{s.e.}(R).
\end{equation}

Combine (\ref{6.1}) and (\ref{7.8}), we finish the proof of Theorem \ref{theorem2}.

\section*{Acknowledgments} 
We would like to thank Mahmood Ettehad, Burak Hatinoglu, and Eviatar Procaccia for fruitful discussions. 

\section{Appendix A: Proof of Claim \ref{claim1}}
In this seciton, we give the proof of Claim \ref{claim1} for completeness. The technique we use in this section is almost parallel to that introduced in Section 6 of \cite{vcerny2012internal}. 

Recall that $a,b,\epsilon>0$ are small enough constants (only depend on $u,d$) and $n=\lfloor bT^{0.5}\rfloor$.
\begin{lemma}\label{lemma3}
	Let $F_s=\left\lbrace x\in \mathbb{Z}^d: x\cdot e_1=s\right\rbrace $. For $0<a<1$, any integer $s\in [-n-n^a,n+n^{a}]$ and $y\in B(8T^{0.5})\setminus B(6T^{0.5})$, \begin{equation}\label{l3}
	\sum\limits_{z\in G_a^{(0,1)}\cap F_s}P_y\left(X_{H_{G_a^{(0,1)}}}=z,H_{G_a^{(0,1)}}\le T\right)\ge c*n^{2-d-a}*f_d(n), 
	\end{equation}
	where $f_3(n)=\frac{n^{a(d-2)}}{\ln(n)}$ and $f_d(n)=n^{a(d-2)}$ for $d\ge 4$.
\end{lemma}
\begin{proof}
	Without loss of generality, we assume $0\le s\le n+n^{a}$.
	
	Define $\widetilde{G}_a^{(0,1)}=\left([-4n,4n]\times [-n^a,n^a]^{d-1} \right)\cap \mathbb{Z}^d $ and $\bar{G}_a^{(0,1)}=\left([-2n,2n]\times [-n^a,n^a]^{d-1} \right)\cap \mathbb{Z}^d $. For any $0\le l \le 2n$, we have $G_a^{(0,1)}+(l-s)*e_1 \subset\widetilde{G}_a^{(0,1)}$ and $B(4n)+(l-s)e_1\subset B(6n)$. Thus \begin{equation}\label{16}
	\sum\limits_{z\in G_a^{(0,1)}\cap F_s}P_z\left(H_{\partial B(4n)}<\bar{H}_{G_a^{(0,1)}},H_{\partial B(4n)}<0.5T  \right) 
	\ge \sum\limits_{z\in \widetilde{G}_a^{(0,1)}\cap F_l}P_z\left(H_{\partial B(6n)}<\bar{H}_{\widetilde{G}_a^{(0,1)}},H_{\partial B(6n)}<0.5T  \right).
	\end{equation}
	Define $\widetilde{G}^{(0,1)}(m)=\left([-4m,4m]\times [-m,m]^{d-1} \right)\cap \mathbb{Z}^d$ and $\bar{G}^{(0,1)}(m)=\left([-2m,2m]\times [-m,m]^{d-1} \right)\cap \mathbb{Z}^d$. By Donsker's Theorem (see Theorem 8.1.11 of \cite{durrett2019probability}) and $\widetilde{G}_a^{(0,1)}\setminus \bar{G}_a^{(0,1)}\subset \widetilde{G}^{(0,1)}(n)\setminus \bar{G}^{(0,1)}(n)$,
	\begin{equation}
	\begin{split}
		&\varliminf\limits_{T \to \infty}P_0\left(H_{\partial B(6n)}\le 0.5T, H_{\widetilde{G}_a^{(0,1)}\setminus \bar{G}_a^{(0,1)}}=\infty\right)\\
		\ge &\varliminf\limits_{T \to \infty}P_0\left(H_{\partial B(6n)}\le 0.5T, H_{\widetilde{G}^{(0,1)}(n)\setminus \bar{G}^{(0,1)}(n)}=\infty\right)\\
		=&P_0^W\left(H_{\partial ([-6,6]^d)}\le 0.5b^{-2}, H_{\widetilde{G}^{(0,1)}(1)\setminus \bar{G}^{(0,1)}(1)}=\infty \right)>0,
	\end{split}
	\end{equation}
	where $P_0^W(\cdot)$ is the law of a Brownian motion starting from $0$. Thus $\exists c>0$ such that \begin{equation}\label{8.4}
	P_0\left(H_{\partial B(6n)}\le 0.5T, H_{\widetilde{G}_a^{(0,1)}\setminus \bar{G}_a^{(0,1)}}=\infty\right) \ge c.
	\end{equation}
	Let $L$ be the last time that a simple random walk starting from $0$ is in $\widetilde{G}_a^{(0,1)}$ before hitting $\partial B(6n)$. By symmetry and (\ref{16}), \begin{equation}\label{p0}
	\begin{split}
	&P_0\left(H_{\partial B(6n)}\le 0.5T, H_{\widetilde{G}_a^{(0,1)}\setminus \bar{G}_a^{(0,1)}}=\infty\right)\\
	\le&\sum\limits_{m=1}^{0.5T}\sum\limits_{l=-2n}^{2n}\sum\limits_{z\in \widetilde{G}_a^{(0,1)}\cap F_l} P_0\left(L=m,X_m=z,H_{\partial B(6n)}\le 0.5T \right)\\
	\le &\sum\limits_{l=-2n}^{2n}\left( \sum\limits_{z\in \widetilde{G}_a^{(0,1)}\cap F_l}P_z\left(H_{\partial B(6n)}<\bar{H}_{\widetilde{G}_a^{(0,1)}}, H_{\partial B(6n)}\le 0.5T \right)\right)*\left(\sup\limits_{z\in \widetilde{G}_a^{(0,1)}\cap F_l}\sum\limits_{m=1}^{0.5T}P_0(X_m=z) \right)\\
	\le &\sum\limits_{z\in G_a^{(0,1)}\cap F_s}P_z\left(H_{\partial B(4n)}<\bar{H}_{G_a^{(0,1)}},H_{\partial B(4n)}\le 0.5T  \right)*\left(\sum\limits_{l=-2n}^{2n}\sup\limits_{z\in \widetilde{G}_a^{(0,1)}\cap F_l}G(z) \right).
	\end{split}
	\end{equation}
	By Theorem 1.5.4 of \cite{lawler2013intersections},  
	\begin{equation}\label{7.5}
		 	\sum\limits_{l=-2n}^{2n}\sup\limits_{z\in \widetilde{G}_a^{(0,1)}\cap F_l}G(z)
		 \le\left\{
		\begin{aligned}
		&c*\sum\limits_{l=1}^{2n}l^{-1}  \       & d=3;\\
		&c*\left( \sum\limits_{l=n^a}^{2n} |l|^{2-d}+n^a*n^{a(2-d)}\right)\          & d\ge 4.
		\end{aligned}
		\right.
	\end{equation} 
	Recall the definition of $f_d(n)$. By (\ref{7.5}), we have: for $d\ge 3$,
	\begin{equation}\label{g0}
	\sum\limits_{l=-2n}^{2n}\sup\limits_{z\in \widetilde{G}_a^{(0,1)}\cap F_l}G(z)
	\le   c*\left( n^{-a}*f_d(n)\right)^{-1}.
	\end{equation}
	Combine (\ref{8.4}), (\ref{p0}) and (\ref{g0}), \begin{equation}\label{pz}
	\sum\limits_{z\in G_a^{(0,1)}\cap F_s}P_z\left(H_{\partial B(4n)}<\bar{H}_{G_a^{(0,1)}},H_{\partial B(4n)}\le 0.5T  \right)\ge cn^{-a}*f_d(n).
	\end{equation}
	
	We claim that \begin{equation}\label{8.9}
		\varliminf\limits_{T\to \infty}\min\limits_{x\in \partial B(4n),y\in B(8T^{0.5})\setminus B(6T^{0.5})}\left\lbrace P_y\left(H_{\partial B(4n)}\le 0.25T, |X_{H_{\partial B(4n)}}-x|_2\le 0.5n \right)\right\rbrace>0.
	\end{equation}
	If (\ref{8.9}) is not true, then $\exists$ a sequence $\left\lbrace (x_{n_j},y_{n_j}):x_{n_j}\in \partial B(4n_j),y_{n_j}\in B(8b^{-1}(n_j+1))\setminus B(6b^{-1}n_j)\right\rbrace_{j=1}^{\infty} $ such that 
	\begin{equation}\label{8.10}
		\lim\limits_{j\to \infty}P_{y_{n_j}}\left(H_{\partial B(4n_j)}\le 0.25b^{-2}n_j^2, |X_{H_{\partial B(4n_j)}}-x_{n_j}|_2\le 0.5n_j \right)=0
	\end{equation}
	Since any $(\frac{x_{n_j}}{n_j},\frac{y_{n_j}}{n_j})$ drops in a compact set $(\partial [-4,4]^d)\times ([-9b^{-1},9b^{-1}]^d\setminus(-6b^{-1},6b^{-1})^d)$, there must exists a sub-sequence $\left\lbrace (x_{m_l},y_{m_l})\right\rbrace_{l=1}^{\infty} $ such that $(x_\infty,y_\infty):=\lim\limits_{l\to \infty}(\frac{x_{m_l}}{m_l},\frac{y_{m_l}}{m_l})\in (\partial [-4,4]^d)\times ([-9b^{-1},9b^{-1}]^d\setminus(-6b^{-1},6b^{-1})^d)$. By Donsker's theorem, 
	\begin{equation}
		\begin{split}
		&\lim\limits_{l\to \infty}P_{y_{m_l}}\left(H_{\partial B(4m_l)}\le 0.25b^{-2}m_l^2, |X_{H_{\partial B(4m_l)}}-x_{m_l}|_2\le 0.5m_l \right)\\
		=&P^W_{y_\infty}\left(H_{\partial([-4,4]^d)}\le 0.25b^{-2}, |X_{H_{\partial ([-4,4]^d)}}-x_{\infty}|_2\le 0.5 \right)>0,
		\end{split}
	\end{equation}
	which is contradictory to (\ref{8.10}).
	
	 By (\ref{8.9}), there exists $c>0$ such that $P_y\left(H_{\partial B(4n)}\le 0.25T, |X_{H_{\partial B(4n)}}-x|_2\le 0.5n \right)\ge c$ for any $x\in \partial B(4n)$ and $y\in B(8T^{0.5})\setminus B(6T^{0.5})$. Then we have 
	\begin{equation}\label{21}
	\begin{split}
	&\sum\limits_{m=1}^{0.5T}P_y(X_m=x,H_{G_a^{(0,1)}}>m)\\
	\ge& P_y\left(H_{\partial B(4n)}
	\le 0.25T, |X_{H_{\partial B(4n)}}-x|_2\le 0.5n \right)\min\limits_{|x-z|_2\le 0.5n,z\in \partial B(4n)}\sum\limits_{k=0}^{0.25T}P_z\left(X_k=x, H_{G_a^{(0,1)}}>k \right)\\
	\ge &c*\min\limits_{|x-z|_2\le 0.5n,z\in \partial B(4n)}\sum\limits_{k=0}^{0.25T}P_z\left(X_k=x, H_{\partial B_x(2n)}>k \right).
	\end{split}
	\end{equation}
	By Proposition1.5.8 of \cite{lawler2013intersections} and Theorem 4.3.1 of \cite{lawler2010random}, for any $z\in \partial B(4n)$ and $|x-z|_2\le 0.5n$,\begin{equation}\label{21.5}
	\begin{split}
	&\sum\limits_{k=0}^{0.25T}P_z\left(X_k=x, H_{\partial B_x(2n)}>k \right)\\
	\ge &G_{\partial B_x(2n)}(z,x)-\sum\limits_{k=0.25T}^{\infty}P_z\left(X_k=x \right)  \\
	= &G(z-x)-\sum\limits_{w\in \partial B_x(2n)}H_{\partial B_x(2n)}(z,w)G(w,x)-\sum\limits_{m=0.25T}^{\infty}P_z\left(X_k=x \right)\\
	\ge & C_d\left(2^{d-2}-2^{2-d} \right)n^{2-d}+O(n^{-d})-\sum\limits_{k=0.25T}^{\infty}P_z\left(X_k=x \right).
	\end{split}
	\end{equation} 
	Recalling (\ref{2.3}) and (\ref{pb}), if $\sum_{i=1}^{d}z^{(i)}-x^{(i)}$ is even, 
	\begin{equation}\label{3.9}
	\begin{split}
	&	\sum\limits_{k=0.25T}^{\infty}P_z\left(X_k=x \right)\\
	\le& \sum\limits_{k=0.25T}^{\infty}\bar{p}_k\left(z-x \right)+c\sum\limits_{k=0.25T}^{\infty}k^{-\frac{d+2}{2}}\\
	\le &\int_{[\frac{T}{8}]-1}^{\infty}2\left( \frac{d}{4\pi t}\right) ^{\frac{d}{2}}dt+O(T^{-\frac{d}{2}})\\
	\le &\frac{4}{d-2}\left(\frac{d}{4\pi} \right)^{\frac{d}{2}}\left( 10b^2\right)^{\frac{d-2}{2}}*n^{2-d}+O(T^{-\frac{d}{2}}).
	\end{split}
	\end{equation}
	If $\sum_{i=1}^{d}z^{(i)}-x^{(i)}$ is odd, similar to (\ref{pbo}), 
	 \begin{equation}\label{3.10}
	\sum\limits_{k=0.25T}^{\infty}P_z\left(X_k=x \right)\le\frac{4}{d-2}\left(\frac{d}{4\pi} \right)^{\frac{d}{2}}\left( 10b^2\right)^{\frac{d-2}{2}}*n^{2-d}+O(T^{-\frac{d}{2}}).
	\end{equation}
	For any fixed $b$ satisfying $\frac{4}{d-2}\left(\frac{d}{4\pi} \right)^{\frac{d}{2}}\left( 10b^2\right)^{\frac{d-2}{2}}<C_d$, combine (\ref{21})-(\ref{3.10}),
	 \begin{equation}\label{3.11}
	\sum\limits_{m=1}^{0.5T}P_y(X_m=x,H_{G_a^{(0,1)}}>m)\ge C_d\left(2^{d-2}-2^{2-d}-1 \right)n^{2-d}-O(n^{-d})\ge cn^{2-d}.
	\end{equation}
	
	Let $L'$ be the last time in $\partial B(4n)$ before hitting $G_a^{(0,1)}$, then 
	\begin{equation}\label{l'}
	\begin{split}
	&\sum\limits_{z\in G_a^{(0,1)}\cap F_s}P_y\left(X_{H_{G_a^{(0,1)}}}=z,H_{G_a^{(0,1)}}\le T\right)\\
	\ge &\sum\limits_{z\in G_a^{(0,1)}\cap F_s}P_y\left(X_{H_{G_a^{(0,1)}}}=z,H_{G_a^{(0,1)}}\le T,L'\le 0.5T\right)\\
	\ge& \sum\limits_{x\in \partial B(4n)}\sum\limits_{m=1}^{0.5T}P_y(X_m=x,H_{G_a^{(0,1)}}>m)\sum\limits_{z\in G_a^{(0,1)}\cap F_s}P_x\left(\bar{H}_{\partial B(4n)}>H_{G_a^{(0,1)}}, H_{G_a^{(0,1)}}\le 0.5T, X_{H_{G_a^{(0,1)}}}=z \right)\\ 
	= & \sum\limits_{x\in \partial B(4n)}\sum\limits_{m=1}^{0.5T}P_y(X_m=x,H_{G_a^{(0,1)}}>m)\sum\limits_{z\in G_a^{(0,1)}\cap F_s}P_z\left(H_{\partial B(4n)}<\bar{H}_{G_a^{(0,1)}}, H_{\partial B(4n)}\le 0.5T, X_{H_{\partial B(4n)}}=x \right)\\
	\ge &\sum\limits_{z\in G_a^{(0,1)}\cap F_s}P_z\left(H_{\partial B(4n)}<\bar{H}_{G_a^{(0,1)}}, H_{\partial B(4n)}\le 0.5T\right)*\left(\min\limits_{x\in \partial B(4n)}\sum\limits_{m=1}^{0.5T}P_y(X_m=x,H_{G_a^{(0,1)}}>m) \right) 
	\end{split}
	\end{equation}
	
	Combine (\ref{pz}), (\ref{3.11}) and (\ref{l'}), then we can get (\ref{l3}).
\end{proof}
\begin{lemma}\label{lemma3.5}
	For any $0<a<0.3$, $\exists c(u,d)>0$ such that \begin{equation}
	cap^{(T)}(B(n^a))\le c*n^{a(d-2)}, 
	\end{equation} 
	where $cap^{(T)}(A)=\sum\limits_{x\in A}P_x^{(T)}(\bar{H}_{A}=\infty)$.
\end{lemma}
\begin{proof}
	For $N_T\sim Geo(\frac{1}{T+1})$ and any $x\in \partial B(n^a)$, we have \begin{equation}\label{24}
	P_x^{(T)}(\bar{H}_{B(n^a)}=\infty)\le P_x(\bar{H}_{B(n^a)}>T^{0.5})+P\left(N_T\le T^{0.5}\right)\le P_x(\bar{H}_{B(n^a)}>T^{0.5})+cT^{-0.5}.
	\end{equation}
	
	Recall the Green's function has approximation $G(x)=O(|x|^{2-d})$. By last-time decomposition and Lemma 3.4 of \cite{vcerny2012internal}, for $\frac{1}{d-2}a<a_1<0.3$, we have
	\begin{equation}\label{25}
	\begin{split}
	&P_x\left(\bar{H}_{B(n^a)}>T^{0.5} \right)-P_x\left(\bar{H}_{B(n^a)}=\infty \right)\\=&P_x\left(T^{0.5}<\bar{H}_{B(n^a)}<\infty \right) \\
	\le &P_x\left(T^{0.5}<\bar{H}_{B(n^a)}<\infty, H_{\partial B(2n^{a_1})}<T^{0.5} \right)+P_x\left(H_{\partial B(2n^{a_1})}\ge T^{0.5} \right)  \\
	\le &\sum\limits_{z\in \partial B(2n^{a_1})}G(x-z)*\sum\limits_{w\in \partial B(n^a)}P_z\left(H_{B(n^a)}<\bar{H}_{B(2n^{a_1})}, X_{H_{B(n^a)}}=w \right)+\boldsymbol{s.e.}(T)\\
	= &\sum\limits_{z\in \partial B(2n^{a_1})}G(x-z)*\sum\limits_{w\in \partial B(n^a)}P_w\left(\bar{H}_{B(n^a)}>H_{\partial B(2n^{a_1})}, X_{H_{\partial B(2n^{a_1})}}=z \right)+\boldsymbol{s.e.}(T)\\
	\le &c*n^{a_1(2-d)}*\sum\limits_{w\in \partial B(n^a)}P_w\left(\bar{H}_{B(n^a)}>H_{\partial B(2n^{a_1})} \right)+\boldsymbol{s.e.}(T)\\
	\le &c*n^{a_1(2-d)+a(d-1)}+\boldsymbol{s.e.}(T)=o(n^{a(d-2)}).
	\end{split}
	\end{equation}

Combine (\ref{24}), (\ref{25}) and (\ref{20}), 
	\begin{equation}
	cap^{(T)}(B(n^a))=\sum\limits_{x\in B(n^a)} P_x^{(T)}(\bar{H}_{B(n^a)}=\infty)\le cap(B(n^a))+o(n^{a(d-2)})\le c*n^{a(d-2)}.
	\end{equation}
\end{proof}
\begin{lemma}\label{lemma4}
	For $0<a<0.3$ and any integer $k\in [-n^{1-a}-1,n^{1-a}+1]$, there exists $c_1, c_2>0$ such that 
	\begin{equation}\label{23}
	P\left( c_1*f_d(n^{a})\le|\bar{S}_x^{(1)}(G_a^{(x,i)}, k)|\le c_2*n^{a(d-2)} \right)\ge 1-\boldsymbol{s.e.}(T).
	\end{equation}
\end{lemma}
\begin{proof}
	Without loss of generality, we fix $i=1$ and $x=0$. Denote the number of the paths hitting $U_k^{(x,i)}$ in $\mathcal{FI}_1^{\frac{1}{3}u,T}$ by $N$. By Lemma 2.1 of \cite{procaccia2019percolation}, $N\sim Pois(\frac{1}{3}u*cap^{(T)}(B(n^a)))$. Then Using Lemma \ref{lemma3.5} and $|\bar{S}_x^{(1)}(G_a^{(x,i)}, k)|\le N$, we know that $Pois(c*n^{a(d-2)})\succeq|\bar{S}_x^{(1)}(G_a^{(x,i)}, k)|$ for some $c>0$.
	  
	Meanwhile, by Lemma \ref{lemma3}, for any $y\in B_x(8T^{0.5})\setminus B_x(6T^{0.5})$, 
	\begin{equation}
	\begin{split}
	P_y\left(H_{G_a^{(x,i)}}\le T, X_{H_{G_a^{(x,i)}}}\in U_k^{(x,i)}\right)= &\sum_{z\in U_k^{(x,i)}\cap \partial G_a^{(x,i)}}P_y\left(H_{G_a^{(x,i)}}\le T, X_{H_{G_a^{(x,i)}}}=z \right)\\
	\ge &cn^a*n^{2-d-a}*f_d(n)=cn^{2-d}*f_d(n).
	\end{split}
	\end{equation}
	For any trajectory $\eta\in S_x^{(1)}$, if $\eta(0)\in B_x(8T^{0.5})\setminus B_x(6T^{0.5})$ is given, then 
	\begin{equation}
	\begin{split}
	P_{\eta(0)}\left(\pi_{G_a^{(x,i)}}(\eta)\in \bar{S}_x^{(1)}(G_a^{(x,i)}, k) \right)
	=&P_{\eta(0)}\left(length(\pi_{G_a^{(x,i)}}(\eta))\ge T, X_{H_{G_a^{(x,i)}}}\in U_k^{(x,i)}\right) \\
	\ge &P_{\eta(0)}\left(length(\eta)\ge 2T,H_{G_a^{(x,i)}}\le T,X_{H_{G_a^{(x,i)}}}\in U_k^{(x,i)} \right)\\
	\ge &\left(1-\frac{1}{T+1} \right)^{2T+1}*cn^{2-d}*f_d(n)\ge cn^{2-d}*f_d(n)
	\end{split}
	\end{equation}
	Since $|S_x^{(1)}|\sim Pois\left(\frac{2du}{T+1}*O(n^d)\right) $ and all paths in $S_x^{(1)}$ are independent given their starting points, $|\bar{S}_x^{(1)}(G_a^{(x,i)}, k)|\succeq Pois(cn^{2-d}*f_d(n)*\frac{2du}{T+1}*O(n^d))\succeq Pois\left(cf_d(n)\right) $.
	
	Finally, due to the large deviation bound for Poisson distribution, we get (\ref{23}).
\end{proof}
We cite Proposition 4.2 of \cite{vcerny2012internal} here, since it is repeatedly referred to in this section.  
\begin{lemma}\label{prop4.2}(Proposition 4.2, \cite{vcerny2012internal})
	Let $\hat{\eta}_1^{(m)}\le \hat{\eta}_2^{(m)}$ be two random variables satisfying that  $P\left(\hat{\eta}_1^{(m)}\ge cm^{d-2-h} \right)\ge 1-\boldsymbol{s.e.}(m) $ and $P\left(\hat{\eta}_2^{(m)}\le cm^{M} \right)\ge 1-\boldsymbol{s.e.}(m) $, where $0<h<\frac{2}{d}$ and $M>0$. Assume $X^{(1)},X^{(2)},...,X^{(\hat{\eta}_2^{(m)})}$ are $\hat{\eta}_2^{(m)}$ independent simple random walks starting from $x^{(1)},x^{(2)},...,x^{(\hat{\eta}_2^{(m)})}\in B(m)$. Then there are inequalities \begin{equation}
		P\left( \forall i,j\le \hat{\eta}_1^{(m)}, X^{(i)}\ and\ X^{(j)}\ are\ \left(2\beta+1,2m^2 \right)-connected  \right) \ge 1-\boldsymbol{s.e.}(m)
	\end{equation}
	and \begin{equation}
		P\left(\forall i\le \hat{\eta}_2^{(m)},p\in L^{(i)},\exists j\le \hat{\eta}_1^{(m)}\ such\ that\ \left\lbrace X^{(i)}_{3pm^2},...,X^{(i)}_{3(p+1)m^2-1}\right\rbrace\cap R(X^{(j)},2m^2)\neq \emptyset  \right) \ge 1-\boldsymbol{s.e.}(m), 
	\end{equation}
	where $\beta(d,h)$ is a constant, $X^{(i)}$ is (s,r) connected to $X^{(j)}$ if there exists a sequence $i=k_0,k_1,...,k_s=j$ in $[1,\hat{\eta}_1^{(m)}]$ such that $R_{k_t}(r)\cap R_{k_{t-1}}(r)\neq 0$ for all $1\le t\le s$ and $L^{(i)}:=\left\lbrace p\ge 1:\left\lbrace X^{(i)}_{3pm^2},...,X^{(i)}_{3(p+1)m^2-1}\right\rbrace\cap B(m)\neq \emptyset\right\rbrace $.
\end{lemma}

Now we are ready to finish the proof of Claim \ref{claim1}. \begin{proof}[proof of Claim \ref{claim1}:]
	For (1) of Claim \ref{claim1}, take $m=n^{a}$, $\hat{\eta}_1^{(m)}=\hat{\eta}_2^{(m)}=|\bar{S}_x^{(1)}(G_a^{(x,i)}, k)|$ in Lemma \ref{prop4.2}. By Lemma \ref{lemma4}, we know that for any $0<h<\frac{2}{d}$, \begin{equation}
	P\left(\hat{\eta}_1^{(m)}\ge Cm^{d-2-h} \right) \ge 1-\boldsymbol{s.e.}(m)
	\end{equation}
	and for $M=d-2$, \begin{equation}
	P\left(\hat{\eta}_2^{(m)}\le Cm^M \right) \ge 1-\boldsymbol{s.e.}(m).
	\end{equation}
	Without loss of generality, denote the paths in $\bar{S}_x^{(1)}(G_a^{(x,i)}, k)$ by $X^{(1)},X^{(2)},...,X^{(\hat{\eta}_2^{(m)})} $. Then $\hat{X}^{(1)},\hat{X}^{(2)},...,\hat{X}^{(\hat{\eta}_2^{(m)})}$ are simple random walks starting from $X^{(1)}(0), X^{(2)}(0),...,X^{(\hat{\eta}_2^{(m)})}(0) \in B_{x+kn^ae_i}(m)$. Thus \begin{equation}\label{7.21}
	P\left( \forall X^{(p)}, X^{(q)}\in \bar{S}_x^{(1)}(G_a^{(x,i)}, k), \hat{X}^{(p)}\ and\ \hat{X}^{(q)}\ are\ \left(2\beta+1,2m^2 \right)-connected  \right) \ge 1-\boldsymbol{s.e.}(T).
	\end{equation}
	If $\hat{X}^{(p)}\ and\ \hat{X}^{(q)}\ are\ \left(2\beta+1,2m^2 \right)-connected$, then \begin{equation}\label{7.22}
	\rho_{\hat{\mathcal{I}}_{x,i}}\left(R(\hat{X}^{(p)},2m^2),R(\hat{X}^{(q)},2m^2) \right)\le 2\left(2\beta+1\right)*m^{2}.
	\end{equation}
	Combine (\ref{7.21}) and (\ref{7.22}), we know (1) of Claim \ref{claim1} happends with probability $1-\boldsymbol{s.e.}(T)$.
	
	For (2) of Claim \ref{claim1}, if $j_k^{x,i}\ge n^\epsilon$, then $\forall 0\le j\le n^\epsilon-1$, $\left\lbrace \hat{X}^{(k)}_{n^{2(a+\epsilon)}+(j-1)n^{2a}+i}:0\le i\le n^{2a}\right\rbrace\cap G_a^{(x,i)}\neq \emptyset$. Define that $t_j=\min\left\lbrace 0\le i\le n^{2a}:\hat{X}^{(k)}_{n^{2(a+\epsilon)}+(j-1)n^{2a}+i}\in G_a^{(x,i)} \right\rbrace $. 
	For any $ 0\le j\le n^\epsilon-3$, the event $\left\lbrace d\left( \hat{X}^{(k)}_{t_j+n^{2a}}, G_a^{(x,i)}\right)\ge n^a, X^{(k)}_{t_j+n^{2a}+l}\notin G_a^{(x,i)}\ for\ all\ 0\le l\le 2n^{2a} \right\rbrace $ doesn't happen (otherwise, $\left\lbrace \hat{X}^{(k)}_{n^{2(a+\epsilon)}+(j+1)n^{2a}+i}:0\le i\le n^{2a}\right\rbrace\cap G_a^{(x,i)}\neq \emptyset$). By Donsker's Theorem, one can see that $\exists c>0$ such that for any $w\in G_a^{(x,i)}$,
	\begin{equation}\label{7.23}
	P_w\left(d\left( X_{n^{2a}}, G_a^{(x,i)}\right)\ge n^a \right)\ge c
	\end{equation}
	and for any $ y\in \left\lbrace z\in \mathbb{Z}^d: d\left(G_a^{(x,i)}, y\right)\ge n^a \right\rbrace$,  \begin{equation}\label{7.24}
	P_y\left(X_l\notin G_a^{(x,i)}\ for\ any\ 0\le l\le 2n^{2a} \right)\ge c. 
	\end{equation}
	Combining (\ref{7.23}) and (\ref{7.24}), by Markov property, we have \begin{equation}
	P\left( d\left( \hat{X}^{(k)}_{t_j+n^{2a}}, G_a^{(x,i)}\right)\ge n^a, \hat{X}^{(k)}_{t_j+n^{2a}+l}\notin G_a^{(x,i)}\ for\ all\  0\le l\le 2n^{2a}  \right)\ge c^2.
	\end{equation}
	By the strong Markov property, \begin{equation}
	\begin{split}
	&P\left(j_k^{x,i}\ge n^\epsilon \right)\\
	= &P\left(j_k^{x,i}\ge n^\epsilon, \bigcap\limits_{0\le j\le n^\epsilon-3,\ 3|j}\left\lbrace d\left( \hat{X}^{(k)}_{t_j+n^{2a}}, G_a^{(x,i)}\right)\ge n^a, \hat{X}^{(k)}_{t_j+n^{2a}+l}\notin G_a^{(x,i)}\ for\ all\ 0\le l\le 2n^{2a} \right\rbrace^c  \right)\\
	\le &(1-c^2)^{\frac{n^\epsilon-3}{3}}=\boldsymbol{s.e.}(T).
	\end{split}
	\end{equation}   
	Thus, (2) of Claim \ref{claim1} happens with probability $1-\boldsymbol{s.e.}(T)$. 
	
	Now we start to consider (3) of Claim \ref{claim1}. We will prove a stronger result: $\forall y\in l_{x,i}$ and $m\ge 0$,
	\begin{equation}\label{8.33}
	P\left(|y-\hat{\varphi}_1^{x,i}(y)|>m\right)\le \boldsymbol{s.e.}(T)+\boldsymbol{s.e.}(m).
	\end{equation}
	Assume $y\in l_{x,i}\cap U_k^{(x,i)}$ and $\mathcal{L}_y^m=\left\lbrace y+l*e_i:1\le l\le m \right\rbrace$. By Lemma 3.2 of \cite{vcerny2012internal}, for integer $j\in [-n^{1-a}-1,n^{1-a}+1]$ and $z\in U_j^{x,i}\cap\partial G_a^{(x,i)}$, 
	\begin{equation}\label{ly}
	P_z\left(H_{\mathcal{L}_y^m}\le n^{2(a+\epsilon)}\right) 
	\le\left\{
	\begin{aligned}
	&\frac{cm}{\left(|j-k|+1\right)^{d-2}n^{a(d-2)}*\ln(m) }  \       & d=3;\\
	&\frac{cm}{ \left(|j-k|+1\right)^{d-2}n^{a(d-2)} }\          & d\ge 4.
	\end{aligned}
	\right.
	\end{equation} 
	
	For the case $d=3$, combine (\ref{ly}) and Lemma \ref{lemma4}, 
	\begin{equation}\label{733}
	\begin{split}
	&P\left(there\ exists\ no\ path\ in\ \bar{S}_x^{(1)}(G_a^{(x,i)})\ hitting\ \mathcal{L}_y\ within\ n^{2(a+\epsilon)}\ steps \right) \\
	\le &\boldsymbol{s.e.}(T)+\prod_{j\in [-n^{1-a}-1,n^{1-a}+1],2|j}\left(1-\frac{cm}{\left(|j-k|+1\right)^{d-2}n^{a(d-2)}*\ln(m)}\right)^{c'f_d(n)} \\
	\le &\boldsymbol{s.e.}(T)+\exp(-c'f_d(n^a)*\sum_{j\in [-n^{1-a}-1,n^{1-a}+1],2|j}\frac{cm}{\left(|j-k|+1\right)^{d-2}n^{a(d-2)}*\ln(m)})\\
	\le &\boldsymbol{s.e.}(T)+\exp(\frac{cm}{\ln(m)})=\boldsymbol{s.e.}(T)+\boldsymbol{s.e.}(m).
	\end{split}
	\end{equation}
	For the case $d\ge 4$, the calculation is very similar to (\ref{733}) so we omit it.
	
	If there exists a path in $\bar{S}_x^{(1)}(G_a^{(x,i)})$ hitting $\mathcal{L}_y^m$ within $n^{2(a+\epsilon)}$ steps, since $\forall \hat{t}_k^{x,i}\ge n^{2(a+\epsilon)}$, then $|\hat{\varphi}_1^{x,i}(y)-y|\le T^\epsilon$. Take $m=T^\epsilon$ in (\ref{8.33}) and then we know (3) of Claim \ref{claim1} happens with probability $1-\boldsymbol{s.e.}(T)$.
	
	For (4) of Claim \ref{claim1}, by (\ref{7.21}), we have: 
	\begin{equation}\label{7.28}
	\begin{split}
	&P\left(\bigcap\limits_{k\in [-n^{1-a}-1,n^{1-a}+1]\cap \mathbb{Z}} \left\lbrace \forall X^{(p)}, X^{(q)}\in \bar{S}_x^{(1)}(G_a^{(x,i)}, k), \hat{X}^{(p)}\ and\ \hat{X}^{(q)}\ are\ \left(2\beta+1,2n^{2a} \right)-connected\right\rbrace   \right)\\
	\ge &1-\boldsymbol{s.e.}(T).
	\end{split}
	\end{equation}
	
	On the other hand, we also need to confirm that there exists common paths between $\bar{S}_x^{(1)}(G_a^{(x,i)}, k)$ and $\bar{S}_x^{(1)}(G_a^{(x,i)}, k+1)$. By Lemma \ref{lemma3}, we have $\left| \left\lbrace X^{(m)}\in \bar{S}_x^{(1)}(G_a^{(x,i)}): X^{(m)}(0)\in U_k^{(x,i)}\cap U_{k+1}^{(x,i)} \right\rbrace\right| \succeq Pois\left(c*f_d(n)\right) $. Thus for any integer $k\in [-n^{1-a}-1,n^{1-a}]$, \begin{equation}
	P\left(\left| \left\lbrace X^{(m)}\in \bar{S}_x^{(1)}(G_a^{(x,i)}): X^{(m)}(0)\in U_k^{(x,i)}\cap U_{k+1}^{(x,i)} \right\rbrace\right| \ge c*f_d(n) \right)\ge 1-\boldsymbol{s.e.}(T). 
	\end{equation}
	Since $\left\lbrace X^{(m)}\in \bar{S}_x^{(1)}(G_a^{(x,i)}): X^{(m)}(0)\in U_k^{(x,i)}\cap U_{k+1}^{(x,i)} \right\rbrace \subset \bar{S}_x^{(1)}(G_a^{(x,i)}, k)\cap \bar{S}_x^{(1)}(G_a^{(x,i)}, k+1)$, \begin{equation}\label{7.30}
	P\left(\bigcap\limits_{k\in [-n^{1-a}-1,n^{1-a}]\cap \mathbb{Z}}\left\lbrace \bar{S}_x^{(1)}(G_a^{(x,i)}, k)\ and\ \bar{S}_x^{(1)}(G_a^{(x,i)}, k+1)\ have\ common\ paths\right\rbrace \right)\ge 1-\boldsymbol{s.e.}(T). 
	\end{equation}
	For any integer $k\in [-n^{1-a}-1,n^{1-a}]$, if $\bar{S}_x^{(1)}(G_a^{(x,i)}, k)$ and $\bar{S}_x^{(1)}(G_a^{(x,i)}, k+1)$ have common paths and $\bigcap\limits_{j=k,k+1}\left\lbrace \forall X^{(p)}, X^{(q)}\in \bar{S}_x^{(1)}(G_a^{(x,i)}, k), \hat{X}^{(p)}\ and\ \hat{X}^{(q)}\ are\ \left(2\beta+1,2n^{2a} \right)-connected  \right\rbrace  $ happens, then $\bigcup\limits_{X^{(m)}\in \bar{S}_x^{(1)}(G_a^{(x,i)}, k)\cup \bar{S}_x^{(1)}(G_a^{(x,i)}, k+1)}R_m(n^{2(a+\epsilon)}) $ is connected. Thus $\hat{\mathcal{I}}_{x,i}$ is connected if the events in (\ref{7.28}) and (\ref{7.30}) both occur. By (\ref{7.28}) and (\ref{7.30}), (4) of Claim \ref{claim1} happens with probability $1-\boldsymbol{s.e.}(T)$.
	
	Finally, let's focus on (5) of Claim \ref{claim1}. Here are some auxiliary notations we need:
	\begin{itemize}
		\item $\widetilde{t}_m^{x,i}=\hat{t}_m^{x,i}\land \inf\left\lbrace j\ge 0: diam(R_m(j))\ge 2n^{2(a+\epsilon)} \right\rbrace $. (Recall the definition of $R_m(j)$ in Setion 3.1.)
		\item $ \widetilde{\mathcal{I}}_{x,i}=\bigcup_{k=1}^{|S_x^{(1)}(G_a^{(x,i)})|}R_k(\widetilde{t}_m^{x,i})$ and $\widetilde{\rho}_{x,i}=\rho_{\widetilde{\mathcal{I}}_{x,i}}$.
		\item For any integer $k\in [-n,n]$, let $L_k=\left\lbrace p\in [-n^{1-a}-1,n^{1-a}+1]\cap \mathbb{Z}: |pn^a-k|\le 2.1n^{2(a+\epsilon)} \right\rbrace  $ and $ \widetilde{\mathcal{I}}_{x,i,k}=\bigcup\limits_{p\in L_k }\bigcup\limits_{X^{(m)}\in \bar{S}_x^{(1)}(G_a^{(x,i)},p)}R_m(\widetilde{t}_m^{x,i})$.
		\item Define $\widetilde{\rho}_{x,i,k}$, $\widetilde{\varphi}_j^{(x,i,k)}$ in the same way as $\hat{\rho}_{x,i}$, $\hat{\varphi}_j^{(x,i)}$ by replacing $\hat{\mathcal{I}}_{x,i}$ with $\widetilde{\mathcal{I}}_{x,i,k}$.
		\item For any integer $k\in [-n,n]$,
		\begin{equation}
		\widetilde{T}_k^{x,i} = \left\{
		\begin{aligned}
		&\widetilde{\rho}_{x,i,k}(x+ke_i,\widetilde{\varphi}_1^{(x,i,k)}(x+ke_i)) ,&if\ x+ke_i\in \widetilde{\mathcal{I}}_{x,i,k},|x+ke_i-\widetilde{\varphi}^{(x,i,k)}_1(x+ke_i)|\le n^a;\\
		&0 ,    &otherwise.
		\end{aligned}
		\right.
		\end{equation}
	\end{itemize}

	Here we need an estimate like Lemma 6.3 of \cite{vcerny2012internal}. Before proving it, we need the following lemma:
	\begin{lemma}\label{lemma8.5} For any integer $m\le 0.5n^a$,
		\begin{equation}
		P\left(\widetilde{\rho}_{x,i,k}(x+ke_i,\widetilde{\varphi}_1^{(x,i,k)}(x+ke_i))> 4(\beta+3)m^2 \right)\le \boldsymbol{s.e.}(m)+\boldsymbol{s.e.}(T),
		\end{equation} 
		where $\beta$ is the constant in Lemma \ref{prop4.2}.
	\end{lemma}
	\begin{proof}
		Denote that $\widetilde{R}(k,m)=\left\lbrace X^{(j)}\in \bigcup\limits_{p\in L_k }\bar{ S}_x^{(1)}(G_a^{(x,i)},p): R_j(\widetilde{t}_j^{x,i})\cap B_{x+ke_i}(m)\neq\emptyset\right\rbrace $.
		
		We claim that $\exists c_1,c_2>0$ such that 
		\begin{equation}\label{8.40}
		P\left(c_1m^{d-2}\le |\widetilde{R}(k,m)|\le c_2m^{d-2} \right)\ge 1-\boldsymbol{s.e.}(m).
		\end{equation}
		In deed, for any $p\in L_k$ and $y\in U_p^{(x,i)}\cap \partial G_a^{(x,i)}$, by Lemma 3.3 of \cite{vcerny2012internal}, we have \begin{equation}
		P_y\left(H_{B_{x+ke_i}(m)}\le n^{2(a+\epsilon)} \right) \ge \frac{cm^{d-2}}{\left(|[pn^a]-k|+n^a \right)^{d-2}}.
		\end{equation} 
		By Lemma \ref{lemma4}, $|\widetilde{R}(k,m)|$ stochasically dominates a Poisson distribution with parameter at least:
		\begin{equation}\label{7.34}
		\sum_{p\in L_k,3|p}c*f_d(n^a)*\frac{m^{d-2}}{\left(|[pn^a]-k|+n^a \right)^{d-2}}\ge c*m^{d-2}.
		\end{equation}
		Meanwhile, $|\widetilde{R}(k,m)|\le \bigg|\left\lbrace \eta\in \mathcal{FI}_1^{\frac{1}{3}u,T}:\eta\cap B_{x+ke_i}(m)\neq\emptyset\right\rbrace\bigg|\sim Pois(\frac{1}{3}u*cap^{(T)}(B(m)))$. By Lemma \ref{lemma3.5}, we know that $|\widetilde{R}(k,m)|$ is stochasically dominated by $Pois(c*m^{d-2})$ for some $c>0$. Then using the large deviation bound for the Poisson distribution, we can get (\ref{8.40}).
		
		Denote $\widetilde{R}(k,m)=\left\lbrace  X^{(j)}:1\le j\le |\widetilde{R}(k,m)| \right\rbrace $ and $\widetilde{X}^{(j)}=\left\lbrace \hat{X}^{(j)}_l:H_{B_{x+ke_i}(m)}\le l\le \widetilde{t}_j^{x,i} \right\rbrace $. By the definition of $\widetilde{t}_j^{x,i}$, we know that $length(\widetilde{X}^{(j)})\ge 2n^{2a}>2m^2$. Apply Lemma \ref{prop4.2} for $\left\lbrace \widetilde{X}^{(l)}:1\le l\le |\widetilde{R}(k,m)| \right\rbrace $ (note that $\hat{\eta}_1^{(m)}=\hat{\eta}_2^{(m)}=|\widetilde{R}(k,m)|$) and then we have \begin{equation}\label{7.36}
		P\left(\forall 1\le p,q\le|\widetilde{R}(k,m)|, \widetilde{X}^{(p)}\ and\ \widetilde{X}^{(q)}\ are\ \left(2\beta+1,2m^2 \right)-connected  \right)\ge 1-\boldsymbol{s.e.}(m)
		\end{equation}
		and 
		\begin{equation}\label{8.44}
		\begin{split}
		P\bigg(&\forall q\le |\widetilde{R}(k,m)|,p\in L^{(q)},\exists j\le |\widetilde{R}(k,m)|\ such\ that\ \left\lbrace \widetilde{X}^{(q)}_{3pm^2},...,\widetilde{X}^{(q)}_{3(p+1)m^2-1}\right\rbrace\\
		&\cap R(\widetilde{X}^{(j)},2m^2)\neq \emptyset  \bigg)\ge 1-\boldsymbol{s.e.}(m),
		\end{split}
		\end{equation}
		where $L^{(q)}=\left\lbrace p\ge 1:\left\lbrace X^{(q)}_{3pm^2},...,X^{(q)}_{3(p+1)m^2-1}\right\rbrace\cap B_{x+ke_i}(m)\neq \emptyset\right\rbrace $.
		
		When $\widetilde{T}_k^{x,i}>0$ (note that $x+ke_i\in \widetilde{\mathcal{I}}_{x,i,k}$), $|x+ke_i-\widetilde{\varphi}^{(x,i,k)}_1(x+ke_i)|\le m $ and the events in (\ref{7.36}), (\ref{8.44}) happen, there exists $\widetilde{X}^{(q_1)},\widetilde{X}^{(q_2)}\in \widetilde{R}(k,m)$ such that $x+ke_i\in \widetilde{X}^{(q_1)}$ and $\widetilde{\varphi}^{(x,i,k)}_1(x+ke_i)\in \widetilde{X}^{(q_2)}$. If $x+ke_i\notin R(\widetilde{X}^{(q_1)},3m^2)$, then $\exists p_1\in L^{(q_1)}$ such that $x+ke_i\in \left\lbrace \widetilde{X}^{(q_1)}_{3p_1m^2},...,\widetilde{X}^{(q_1)}_{3(p_1+1)m^2-1} \right\rbrace$, by the event in (\ref{8.44}), $\exists \widetilde{X}^{(j)}$ such that $\left\lbrace \widetilde{X}^{(q_1)}_{3p_1m^2},...,\widetilde{X}^{(q_1)}_{3(p_1+1)m^2-1}\right\rbrace
		\cap R(\widetilde{X}^{(j)},2m^2)\neq \emptyset$ and thus $d(x+ke_i,R(\widetilde{X}^{(j)},2m^2))\le 3m^2$. If $x+ke_i\in R(\widetilde{X}^{(q_1)},3m^2)$, we also have $d(x+ke_i,R(\widetilde{X}^{(q_1)},2m^2))\le m^2<3m^2$. In conclution, there exists $\widetilde{X}^{(j_1)}\in \widetilde{R}(k,m)$ such that $d(x+ke_i,R(\widetilde{X}^{(j_1)},2m^2))\le 3m^2$. Similarly, there also exists $\widetilde{X}^{(j_2)}\in \widetilde{R}(k,m)$ such that $d(\widetilde{\varphi}^{(x,i,k)}_1(x+ke_i),R(\widetilde{X}^{(j_2)},2m^2))\le 3m^2$. By the event of (\ref{7.36}), $\widetilde{X}^{(j_1)}$ and $\widetilde{X}^{(j_2)}$ are $(2\beta+1,2m^2)$-connected. Thus 
		\begin{equation}\label{8.45}
		\widetilde{\rho}_{(x,i,k)}(x+ke_i,\widetilde{\varphi}_1^{(x,i,k)}(x+ke_i))\le (2\beta+1)*2m^2+4m^2+6m^2=4(\beta+3)m^2,
		\end{equation}
		which means if $\widetilde{\rho}_{(x,i,k)}(x+ke_i,\widetilde{\varphi}_1^{(x,i,k)}(x+ke_i))>4(\beta+3)m^2$ and $|x+ke_i-\widetilde{\varphi}_1^{(x,i,k)}(x+ke_i)|\le m$ both occur, then the events in (\ref{7.36}) and (\ref{8.44}) won't happend at the same time.
		
		By (\ref{8.33}), (\ref{7.36}) and (\ref{8.44}), we have \begin{equation}\label{8.47}
		\begin{split}
		&P\left(\widetilde{\rho}_{(x,i,k)}(x+ke_i,\widetilde{\varphi}_1^{x,i,k}(x+ke_i))> 4(\beta+3)m^2 \right)\\
		\le& P\left(\widetilde{\rho}_{(x,i,k)}(x+ke_i,\widetilde{\varphi}_1^{(x,i,k)}(x+ke_i))> 4(\beta+3)m^2, |x+ke_i-\widetilde{\varphi}_1^{(x,i,k)}(x+ke_i)|\le m\right)\\
		&+P\left(|x+ke_i-\widetilde{\varphi}_1^{(x,i,k)}(x+ke_i)|> m \right) \\
		\le&\boldsymbol{s.e.}(m)+\boldsymbol{s.e.}(T).
		\end{split}
		\end{equation}
	\end{proof}
	For any $l\le (\beta+3)n^{2a}$, there exists $m\le 0.5n^a$ such that $4(\beta+3)m^2\le l< 4(\beta+3)(m+1)^2$. By Lemma \ref{lemma8.5}, we have
	\begin{equation}\label{743}
	P\left(\widetilde{T}_k^{x,i}> l\right)\le P\left(\widetilde{T}_k^{x,i}> 4(\beta+3)m^2\right)\le  \boldsymbol{s.e.}(m)+\boldsymbol{s.e.}(T).
	\end{equation}
	For the situation $l> (\beta+3)n^{2a}$, 
	\begin{equation}\label{744}
	P\left(\widetilde{T}_k^{x,i}> l\right)\le P\left(\widetilde{T}_k^{x,i}>(\beta+3)n^{2a}\right)\le \boldsymbol{s.e.}(T).
	\end{equation}
	Combine (\ref{743}) and (\ref{744}), we have: for any integer $k\in [-n,n]$ and $l\ge 0$,
	\begin{equation}\label{hattk}
	P\left(\widetilde{T}_k^{x,i}\ge l \right)\le \boldsymbol{s.e.}(l)+\boldsymbol{s.e.}(T). 
	\end{equation} 
	
	If $j_k^{x,i}< n^\epsilon$, then $\hat{t}_k^{x,i}=n^{2(a+\epsilon)}+j_k^{x,i}*n^{2a}<2n^{2(a+\epsilon)}$ and $diam(R_k(\hat{t}_k^{x,i}))<2n^{2(a+\epsilon)}$. Consequently, if for any $k$, the event $j_k^{x,i}< n^\epsilon$ happens, then $\widetilde{\mathcal{I}}_{x,i}=\hat{\mathcal{I}}_{x,i}$. Since $P\left(j_k^{x,i}< n^\epsilon\right)\ge 1-\boldsymbol{s.e.}(T) $ and $P\left(|S_x^{(1)}(G_a^{(x,i)})|\le cn^d \right)\ge 1-\boldsymbol{s.e.}(T) $, we have \begin{equation}\label{widehat}
	P\left(\widetilde{\mathcal{I}}_{x,i}=\hat{\mathcal{I}}_{x,i} \right)\ge 1-\boldsymbol{s.e.}(T). 
	\end{equation}
	
	By the definition of $\widetilde{t}_m^{x,i}$ and $\widetilde{\mathcal{I}}_{x,i,k}$, we know that $B_{x+ke_i}\left(n^a \right)\cap \widetilde{\mathcal{I}}_{x,i,k}=B_{x+ke_i}\left(n^a \right)\cap \widetilde{\mathcal{I}}_{x,i}$. Therefore, When $\widetilde{\mathcal{I}}_{x,i}=\hat{\mathcal{I}}_{x,i}$, $x+ke_i\in \hat{\mathcal{I}}_{x,i}$ and $|x+ke_i-\hat{\varphi}_1^{(x,i)}(x+ke_i)|\le n^a$ happen, we have $x+ke_i\in \widetilde{\mathcal{I}}_{x,i}$ and $|x+ke_i-\widetilde{\varphi}_1^{(x,i,k)}(x+ke_i)|\le n^a$. Thus if $\widetilde{\mathcal{I}}_{x,i}=\hat{\mathcal{I}}_{x,i}$ and $\bigcap\limits_{-n\le k\le n}\left\lbrace |x+ke_i-\hat{\varphi}_1^{(x,i)}(x+ke_i)|\le n^a \right\rbrace $ both happen, for any $y_1,y_2\in l_{x,i}\cap \hat{\mathcal{I}}_{x,i}$, we have \begin{equation}\label{7.40}
	\hat{\rho}_{x,i}(y_1,y_2)\le \sum\limits_{k=-n}^{n}\widetilde{T}_k^{x,i}.
	\end{equation}
	
	Let $b_n=\lfloor 5n^{2(a+\epsilon)}\rfloor $. By the definition of $\widetilde{\mathcal{I}}_{x,i,k}$, the random variables $\left\lbrace \widetilde{T}_{l*b_n+j}^{x,i}:|l*b_n+j|\le n \right\rbrace $ are independent for $1\le j\le b_n$. By (\ref{widehat}), (\ref{7.40}) and $P\left(|x+ke_i-\hat{\varphi}_1^{(x,i)}(x+ke_i)|\le n^a \right)\ge 1-\boldsymbol{s.e.}(T)$, we have  \begin{equation}\label{7.41}
	\begin{split}
	P\left(\hat{\rho}_{x,i}(y_1,y_2)>cn \right)\le& P\left(\sum\limits_{k=-n}^{n}\widetilde{T}_k^{x,i}>cn \right) +\boldsymbol{s.e.}(T)\\
	\le &\sum\limits_{j=1}^{b_n}P\left(\sum\limits_{l:|l*b_n+j|\le n}\widetilde{T}_{l*b_n+j}^{x,i}>\frac{cn}{b_n} \right) +\boldsymbol{s.e.}(T)\\
	\le &\sum\limits_{j=1}^{b_n}P\left(\sum\limits_{l:|l*b_n+j|\le n}\left( \widetilde{T}_{l*b_n+j}^{x,i}\land n\right) >\frac{cn}{b_n} \right) +\sum_{k=-n}^{n}P\left(\widetilde{T}_k^{x,i}>n \right) +\boldsymbol{s.e.}(T).
	\end{split}
	\end{equation}
	
	We need a large diviation bound in \cite{nagaev1979large}. 
	\begin{lemma}\label{colo1.5}(Colollary 1.5, \cite{nagaev1979large})
		For $0<t\le 1$, assume that $X_1,X_2,...,X_m$ are independent random variables such that $A_t^+=\sum_{j=1}^{m}E\left(|X_j|^t*1_{\{X_j>0\}} \right)<\infty $. For $x,y_1,y_2,...,y_m>0$ and $y\ge \max\left\lbrace y_1,y_2,...,y_m \right\rbrace $, \begin{equation}
		P\left(X_1+X_2+...+X_m\ge x \right)\le \sum_{j=1}^{m}P\left( X_j\ge y_j\right)+\left(\frac{eA_t^+}{xy^{t-1}} \right)^{x/y}.
		\end{equation}
	\end{lemma}

	Take $m=|\left\lbrace l:|l*b_n+j|\le n\right\rbrace |$, $X_l=\widetilde{T}_{l*b_n+j}^{x,i}\land n$, $x=\frac{cn}{b_n}$, $y=y_1=...=y_m=n^\epsilon$ and $t=1$ in Lemma \ref{colo1.5}, then we have \begin{equation}\label{7.43}
	P\left(\sum\limits_{l:|l*b_n+j|\le n}\widetilde{T}_{l*b_n+j}^{x,i}\land n>\frac{cn}{b_n} \right)\le \sum_{j=1}^{m}P\left( \widetilde{T}_{l*b_n+j}^{x,i}\land n\ge n^\epsilon \right)+\left(\frac{5e\widetilde{A}_1^+b_n}{cn}\right) ^{0.2c*n^{1-2a-3\epsilon}}, 
	\end{equation}
	where $\widetilde{A}_1^+=\sum\limits_{l:|l*b_n+j|\le n}E|\widetilde{T}_{l*b_n+j}^{x,i}\land n|$ and $1-2a-3\epsilon>0$.
	
	By (\ref{hattk}), there exists $c'(u,d)>0$ such that \begin{equation}
	E|\widetilde{T}_{k}^{x,i}\land n|\le \sum\limits_{m=1}^{n}P\left(\widetilde{T}_{k}^{x,i}\ge m \right)\le \sum\limits_{m=1}^{n}\left( \boldsymbol{s.e.}(m)+\boldsymbol{s.e.}(n)\right) \le c'. 
	\end{equation}
	Thus if we take $c$ large enough in (\ref{7.43}), then
	\begin{equation}\label{7.45}
	\frac{5e\widetilde{A}_1^+b_n}{cn}\le \frac{10ec'}{c}<1.
	\end{equation}
	
	Meanwhile, using (\ref{hattk}) again, for any integer $k\in [-n,n]$,  \begin{equation}\label{7.46}
	P\left( \widetilde{T}_{k}^{x,i}\land n\ge n^\epsilon \right)= P\left( \widetilde{T}_{k}^{x,i}\ge n^\epsilon \right)\le \boldsymbol{s.e.}(T).
	\end{equation}
	
	Combine (\ref{7.41}), (\ref{7.43}), (\ref{7.45}) and (\ref{7.46}),  \begin{equation}
	P\left(\hat{\rho}_{x,i}(y_1,y_2)>cn \right)\le \boldsymbol{s.e.}(T).
	\end{equation}
	Thus we know that (5) of Claim \ref{claim1} happens with probability $1-\boldsymbol{s.e.}(T)$ and finally, the proof of Claim \ref{claim1} is complete.
\end{proof}

\section{Appendix B: Proof of Corollary \ref{quenched invariance principle}}

Let $0 < \alpha< \beta < \infty$. We first show that there exists $0 < T_5(d, \alpha, \beta) <\infty$ such that for all $T> T_5$, (\ref{local uniqueness}) holds and $P^{u,T} (0 \in \Gamma) >0$ for all $u \in (\alpha, \beta)$. By the proof of Theorem \ref{theorem2}, there exists $T' (d, \alpha, \beta)>0$ such that (\ref{local uniqueness}) holds for all $u \in (\alpha, \beta)$ and all $T>T'$. By the proof of Proposition \ref{prop12},
\begin{equation}
\begin{aligned}
& P^{u,T} (0 \in \Gamma) \\
& \geq P^{u,T} \Big(0 \leftrightarrow \partial B(0, T^{1/3}) \Big) - P^{u,T} \Big( 0 \leftrightarrow \partial B(0, T^{1/3}), \rho (0, \bar{\Gamma}_{12}) > T^{2d/3} \Big) \\
& \geq 1 - s.e. (T) - s.e.(T)^{T^{c_3}} \\
& \geq 1 - s.e.(T).
\end{aligned}
\end{equation} 
Let $T'' <\infty$ such that for all $T> T''$, $\eta^{T} (u) >0$ for all $u \in (\alpha, \beta)$. We could choose $T_5 = \max \{T', T''\}$.

Define $\eta^T (u) := P^{u,T} (0 \in \Gamma)$. The next lemma shows that $\eta^T (u)$ is continuous for $T >T_5$.

\begin{lemma}
\label{continuity}
Let $d \geq 3$, $0 < \alpha < \beta < \infty$, and $T_5 (d, \alpha, \beta)$ be the same critical value in Corollary \ref{quenched invariance principle}. For all $T > T_5$, $\eta^T (u)$ is continuous on $(\alpha, \beta)$.
\end{lemma}

\begin{proof}
We follow the proof of Corollary $1.2$ of Teixeira \cite{teixeira2009uniqueness} closely. First we prove the right-continuity of $\eta^T$. Define the event
$$
C^{u,T}_r = \Big\{ 0 \xleftrightarrow{\mathcal{FI}^{u,T}}\partial B(0,r) \Big\}.
$$ 
Denote the complement of $C^{u,T}_r$ by $D^{u,T}_r$. Similar to its counterpart in vacant set of random interlacements, $P(C^{u,T}_r)$ is real analytic from inclusion-exclusion principle and Corollary $2.1$ in \cite{procaccia2019percolation}. Note that 
$$
1- \eta^T (u) = P^{u,T} (0 \notin \Gamma) = P \Bigg( \bigcup_{r \geq 1} D^{u,T}_r \Bigg) = \lim_{r \rightarrow \infty} P(D^{u,T}_r)
$$
is an increasing limit of continuous functions and hence is lower-semicontinuous on $\mathbb{R}_{+}$. Since $1 - \eta^T(u)$ is monotone non-increasing in $u$, it is right-continuous on $\mathbb{R}_{+}$. Therefore, $\eta^T(u)$ is also right-continuous on $\mathbb{R}_{+}$. To show that $\eta^T(u)$ is left-continuous, we consider the event $$C^{u,T}_{\infty} := \bigcap_{r \geq 1} C^{u,T}_r$$ for $u \in (\alpha, \beta)$. We could couple FRI $\mathcal{FI}^{v,T}$ for all $v \in \mathbb{R}_{+}$ and for a fixed $T$. Similar to the definition of random interlacements in Section $5$ of \cite{drewitz2014introduction}, consider a Poisson Point Process on the space $W^{[0, \infty)} \times \mathbb{R}_{+}$ with intensity measure $v^{(T)} \times m$, where $m$ is the Lebesgue measure on $\mathbb{R}_{+}$. Note that $C^{v,T}_{\infty}$ is monotone non-decreasing with respect to $v$, so
\begin{equation}
\label{left limit of eta}
\lim_{v \uparrow u} \eta^T(v) = \lim_{v \uparrow u} P(C^{v,T}_{\infty}) = P \Bigg( \bigcup_{v<u} C^{v,T}_{\infty} \Bigg).
\end{equation}
It suffices to prove that the limit in (\ref{left limit of eta}) is $\eta^T(u)$, or it is equivalent to prove that
\begin{equation}
\label{prob diff equal 0}
P \Bigg( C^{u,T}_{\infty} \setminus \bigcup_{v<u} C^{v,T}_{\infty} \Bigg) = 0.
\end{equation}
Fix $v_0 \in (\alpha,u)$. Define an event
\begin{equation}
F := \left\{ w = \sum_{i =1}^{\infty} \delta_{(\eta_i, u_i)} : \begin{aligned}
& w \text{ is a point measure on } W^{[0, \infty)} \times \mathbb{R}_{+}, \text{ and there exists unique} \\ &\text{ infinite clusters } \Gamma_{v_0} \text{ and } \Gamma_{u} \text{ of } \mathcal{FI}^{v_0, T} \text{ and } \mathcal{FI}^{u, T}, \text{ respectively } \end{aligned} \right\}.
\end{equation}
Since $T > T_5$, $\mathcal{FI}^{v,T}$ has a unique infinite cluster almost surely for all $v \in (\alpha, \beta)$ (see Theorem $1$ of \cite{procaccia2019percolation}). By uniqueness, $\Gamma_{v_0} \subset \Gamma_{u}$ for all $w \in F$. If $w \in F \cap C^{u,T}_{\infty}$, then there exists a finite path $\tau$ in $\Gamma_{u}$ connecting $0$ to $\Gamma_{v_0}$. By Definition \ref{definition1} of FRI, there are finite number of pairs $(\eta_i, u_i)$ with $v_0 < u_i < u$ such that $\eta_i$ intersects $\tau$. Thus there exists $v_1 \in (v_0,u)$ such that $w \in C^{v_1,T}_{\infty}$. Therefore, 
\begin{equation}
\label{intersect with F}
F \cap C^{u,T}_{\infty} \subset \bigcup_{v<u} C^{v,T}_{\infty}.
\end{equation}
(\ref{prob diff equal 0}) follows from (\ref{intersect with F}) and the fact that $P(F)=1$.

\end{proof}

\begin{proof}[Proof of Corollary \ref{quenched invariance principle}]
Let $0 < \alpha < \beta <\infty$. We show that FRI $\mathcal{FI}^{u,T}$ satisfies conditions $\bold{P1} - \bold{P3}$ and $\bold{S1} - \bold{S2}$ listed in \cite{procaccia2016quenched} for $T > T_5(d, \alpha, \beta)$ and $u \in (\alpha, \beta)$. The reader is referred to \cite{drewitz2014chemical, procaccia2016quenched} for a detailed descriptions of these $5$ conditions. FRI satisfies $\bold{P1}$ (ergodicity) by Proposition $6.1$ in \cite{procaccia2019percolation} and $\bold{P2}$ (monotonicity) by definitions of FRI. Fix $T > T_5$. For $i \in \{1,2\}$, let $x_i \in \mathbb{Z}^d$ and $A_i \in \sigma (\Psi_y: y \in B(x_i, 10 L))$, where $\Psi_y : \{0,1\}^{\mathbb{Z}^d} \rightarrow \{0,1\}$ is the coordinate map at $y \in \mathbb{Z}^d$ and $L \in \mathbb{Z}_{+}$. Recall Definition \ref{definition2} of FRI. Define an event
\begin{equation}
F_L^{v,T} := \left\{ 
\begin{aligned}
&\text{There exists a geometrically killed random walk starting from } \\ & B(x_1, 20L) \text{ intersecting } B(x_1, 10L) \text{ in } \mathcal{FI}^{v,T}
\end{aligned}
\right\}.
\end{equation}
One can easily adapt the proof of Lemma \ref{lemma13} (or the proof of Lemma $4.10$ of \cite{procaccia2019percolation}) and show that
\begin{equation}
\label{small prob from faraway}
P(F_L^{v,T}) \leq 1 - e^{-c_1 L},
\end{equation}
for all $v \in (\alpha, \beta)$, where $c_1 (T, d, \alpha, \beta)>0$ is a constant.  By (\ref{small prob from faraway}) and Definition \ref{definition2}, for all $x_1, x_2$ such that $|x_1 - x_2| \geq 50L$, and for all $v \in (\alpha, \beta)$, we have
$$
| P^{v,T} (A_1 \cap A_2) - P^{v,T} (A_1) P^{v,T} (A_2) | \leq e^{-c_2 L}.
$$
Thus condition $\bold{P3}$ (decoupling) is satisfied. By Proposition \ref{Y} and Theorem \ref{theorem2}, condition $\bold{S1}$ (local uniqueness) is satisfied for $T> T_5$. For condition $\bold{S2}$ (continuity), $\eta^{T}$ is positive and continuous on $(\alpha, \beta)$ by the choice of $T_5$ and Lemma \ref{continuity}. The result of Corollary \ref{quenched invariance principle} follows from Theorem $1.1$ in \cite{procaccia2016quenched}.
\end{proof}

\bibliographystyle{plain}
\bibliography{ref}

\begin{thebibliography}{10}

\bibitem{antal1996chemical}
P.~Antal and A.~Pisztora.
\newblock On the chemical distance for supercritical bernoulli percolation.
\newblock {\em The Annals of Probability}, pages 1036--1048, 1996.

\bibitem{bowen2019finitary}
L.~Bowen.
\newblock Finitary random interlacements and the gaboriau--lyons problem.
\newblock {\em Geometric and Functional Analysis}, 29(3):659--689, 2019.

\bibitem{vcerny2012internal}
J.~{\v{C}}ern{\`y} and S.~Popov.
\newblock On the internal distance in the interlacement set.
\newblock {\em Electronic Journal of Probability}, 17, 2012.

\bibitem{drewitz2014introduction}
A.~Drewitz, B.~R{\'a}th, and A.~Sapozhnikov.
\newblock {\em An introduction to random interlacements}.
\newblock Springer, 2014.

\bibitem{drewitz2014chemical}
A.~Drewitz, B.~R{\'a}th, and A.~Sapozhnikov.
\newblock On chemical distances and shape theorems in percolation models with
  long-range correlations.
\newblock {\em Journal of Mathematical Physics}, 55(8):083307, 2014.

\bibitem{durrett2019probability}
R.~Durrett.
\newblock {\em Probability: theory and examples}, volume~49.
\newblock Cambridge university press, 2019.

\bibitem{garet2017continuity}
Olivier Garet, Regine Marchand, Eviatar~B Procaccia, Marie Th{\'e}ret, et~al.
\newblock Continuity of the time and isoperimetric constants in supercritical
  percolation.
\newblock {\em Electronic Journal of Probability}, 22, 2017.

\bibitem{grimmettpercolation}
G.~R. Grimmett, A.~E. Holroyd, and G.~Kozma.
\newblock Percolation of finite clusters and infinite surfaces.
\newblock {\em arXiv preprint arXiv:1303.1657}, 2013.

\bibitem{grimmett2013percolation}
G.R. Grimmett.
\newblock {\em Percolation}.
\newblock Grundlehren der mathematischen Wissenschaften. Springer Berlin
  Heidelberg, 2013.

\bibitem{lawler2013intersections}
G.~F. Lawler.
\newblock {\em Intersections of random walks}.
\newblock Springer Science \& Business Media, 2013.

\bibitem{lawler2010random}
G.~F. Lawler and V.~Limic.
\newblock {\em Random walk: a modern introduction}, volume 123.
\newblock Cambridge University Press, 2010.

\bibitem{liggett1997domination}
T.~M. Liggett, R.~H. Schonmann, and A.~M. Stacey.
\newblock Domination by product measures.
\newblock {\em The Annals of Probability}, 25(1):71--95, 1997.

\bibitem{nagaev1979large}
S.~V. Nagaev.
\newblock Large deviations of sums of independent random variables.
\newblock {\em The Annals of Probability}, pages 745--789, 1979.

\bibitem{procaccia2016quenched}
E.~B. Procaccia, R.~Rosenthal, and A.~Sapozhnikov.
\newblock Quenched invariance principle for simple random walk on clusters in
  correlated percolation models.
\newblock {\em Probability theory and related fields}, 166(3-4):619--657, 2016.

\bibitem{procaccia2014range}
E.~B. Procaccia and E.~Shellef.
\newblock On the range of a random walk in a torus and random interlacements.
\newblock {\em The Annals of Probability}, 42(4):1590--1634, 2014.

\bibitem{procaccia2019percolation}
E.~B. Procaccia, J.~Ye, and Y.~Zhang.
\newblock Percolation for the finitary random interlacements.
\newblock {\em arXiv preprint arXiv:1908.01954}, 2019.

\bibitem{rath2011transience}
B.~R{\'a}th and A.~Sapozhnikov.
\newblock On the transience of random interlacements.
\newblock {\em Electronic Communications in Probability}, 16:379--391, 2011.

\bibitem{rodriguez2013phase}
P.~Rodriguez and A.~S. Sznitman.
\newblock Phase transition and level-set percolation for the gaussian free
  field.
\newblock {\em Communications in Mathematical Physics}, 320(2):571--601, 2013.

\bibitem{Sznitman2009Vacant}
A.~S. Sznitman.
\newblock Vacant set of random interlacements and percolation.
\newblock {\em Annals of Mathematics}, 171(3):págs. 2039--2087, 2009.

\bibitem{teixeira2009uniqueness}
A.~Teixeira.
\newblock On the uniqueness of the infinite cluster of the vacant set of random
  interlacements.
\newblock {\em The Annals of Applied Probability}, 19(1):454--466, 2009.

\end{thebibliography}
\end{document}